\documentclass[11pt]{article}
\usepackage[pdftex]{graphicx}
\pdfoutput=1

\usepackage[
    persons={Nima, Fred, June},
    authors=footnotes,
    authorsfont=normalsize,
    bibliosources={refs.bib},
    ifclass={{acmart,sigconf,sigconf}},
]{pomegranate}

\DeclareDelimiter{\dTV}[\mathnormal{d}_{\operatorname{TV}}]{\lparen}{\rparen}
\DeclareDelimiter{\opnorm}{\lVert}{\rVert_{\operatorname{op}}}
\DeclareDelimiter{\Normal}[\mathcal{N}]{\lparen}{\rparen}
\DeclareDelimiter{\Ber}[\operatorname{Ber}]{\lparen}{\rparen}
\DeclareDelimiter{\Berpm}[\operatorname{Ber}_{\pm}]{\lparen}{\rparen}
\DeclareDelimiter{\DKL}[\mathcal{D}_{\operatorname{KL}}]{\lparen}{\rparen}
\DeclareDelimiter{\Ent}[\operatorname{Ent}]{\lbrack}{\rbrack}
\DeclareDelimiter{\chis}[\chi^2]{\lparen}{\rparen}
\DeclareDelimiter{\Var}[\operatorname{Var}]{\lbrack}{\rbrack}
\DeclareDelimiter{\Cov}[\operatorname{Cov}]{\lbrack}{\rbrack}
\DeclareDelimiter{\Ehat}[\widehat{\mathbb E}]{\lbrack}{\rbrack}
\DeclareDelimiter{\cov}[\operatorname{cov}]{\lparen}{\rparen}
\DeclareDelimiter{\mean}[\operatorname{mean}]{\lparen}{\rparen}

\DeclareDocumentMathCommand{\observe}{m}{
	\mathrel{}\mathclose{}\lvert\mathopen{}\mathrel{}#1\vartriangleright
}
\DeclareDocumentMathCommand{\corMat}{}{\Psi^{\operatorname{cor}}}
\DeclareDocumentMathCommand{\tilt}{}{\mathcal{T}}
\DeclareDocumentMathCommand{\F}{}{\mathcal{F}}
\DeclareOperator{\diag}
\DeclareOperator{\Tr}
\DeclareOperator{\grad}
\DeclareOperator{\polar}

\title{Trickle-Down in Localization Schemes and Applications}
\date{}

\Tag{
	\author[1]{Nima Anari}
	\author[2]{Frederic Koehler}
	\author[1]{Thuy-Duong Vuong}
	\affil[1]{Stanford University, \url{{anari,tdvuong}@stanford.edu}}
        \affil[2]{University of Chicago, \url{fkoehler@uchicago.edu}}
}
\Tag<sigconf>{
	\author{Nima Anari}
	\email{anari@cs.stanford.edu}
	\orcid{0000-0002-4394-3530}
	\affiliation{
	  \institution{Stanford University}
	  \city{Stanford}
	  \state{California}
	  \country{USA}
	}
	
	\author{Frederic Koehler}
	\email{fkoehler@uchicago.edu}
	\orcid{0000-0001-5220-9680}
	\affiliation{
	  \institution{University of Chicago}
	  \city{Chicago}
	  \state{Illinois}
	  \country{USA}
	}
	
	\author{Thuy-Duong Vuong}
	\email{tdvuong@stanford.edu}
	\orcid{0000-0003-0271-9687}
	\affiliation{
	  \institution{Stanford University}
	  \city{Stanford}
	  \state{California}
	  \country{USA}
	}
}
\Tag<anon>{
	\author{}
}

\Tag<sigconf>{


\setcopyright{acmlicensed}
\acmDOI{10.1145/3618260.3649622}
\acmYear{2024}
\copyrightyear{2024}
\acmSubmissionID{stoc24main-p76-p}
\acmISBN{979-8-4007-0383-6/24/06}
\acmConference[STOC '24]{Proceedings of the 56th Annual ACM Symposium on Theory of Computing}{June 24--28, 2024}{Vancouver, BC, Canada}
\acmBooktitle{Proceedings of the 56th Annual ACM Symposium on Theory of Computing (STOC '24), June 24--28, 2024, Vancouver, BC, Canada}
\received{10-NOV-2023}
\received[accepted]{2024-02-11}
}

\Tag<sigconf>{
	\DeclareFieldFormat{doi}{%
		\ifhyperref
	    	{\href{https://doi.org/##1}{\nolinkurl{https://doi.org/##1}}}
		    {\nolinkurl{https://doi.org/##1}}}
}

\begin{document}
	\Tag<sigconf>{
		\begin{abstract}
			Trickle-down is a phenomenon in high-dimensional expanders with many important applications — for example, it is a key ingredient in various constructions of high-dimensional expanders or the proof of rapid mixing for the basis exchange walk on matroids and in the analysis of log-concave polynomials. We formulate a generalized trickle-down equation in the abstract context of linear-tilt localization schemes. Building on this generalization, we improve the best-known results for several Markov chain mixing or sampling problems — for example, we improve the threshold up to which Glauber dynamics is known to mix rapidly in the Sherrington-Kirkpatrick spin glass model. Other applications of our framework include improved mixing results for the Langevin dynamics in the $O(N)$ model, and near-linear time sampling algorithms for the antiferromagnetic and fixed-magnetization Ising models on expanders. For the latter application, we use a new dynamics inspired by polarization, a technique from the theory of stable polynomials.
		\end{abstract}
		\maketitle
	}
	\Tag{
		\maketitle
		\begin{abstract}
			Trickle-down is a phenomenon in high-dimensional expanders with many important applications — for example, it is a key ingredient in various constructions of high-dimensional expanders or the proof of rapid mixing for the basis exchange walk on matroids and in the analysis of log-concave polynomials. We formulate a generalized trickle-down equation in the abstract context of linear-tilt localization schemes. Building on this generalization, we improve the best-known results for several Markov chain mixing or sampling problems — for example, we improve the threshold up to which Glauber dynamics is known to mix rapidly in the Sherrington-Kirkpatrick spin glass model. Other applications of our framework include improved mixing results for the Langevin dynamics in the $O(N)$ model, and near-linear time sampling algorithms for the antiferromagnetic and fixed-magnetization Ising models on expanders. For the latter application, we use a new dynamics inspired by polarization, a technique from the theory of stable polynomials.
		\end{abstract}	
		\clearpage
	}
    \newpage
    \tableofcontents
    \newpage
    \section{Introduction}\label{sec:intro}
Recent years have seen an explosion of work on understanding \emph{high-dimensional expansion}, in other words, higher-dimensional analogues of expansion in graphs. This has contributed to significant breakthroughs in our understanding of Markov chains, error-correcting codes, the geometry of polynomials, and other areas. See, e.g., \cite{CGM19,AL20,KO18,ALOV19,dinur2017high} for a few relevant references.

The present work is inspired by one of the key ingredients in the theory of high-dimensional expanders: \Citeauthor{Opp18}'s \emph{trickle-down} (or trickling-down) theorem and the related \citeauthor{garland1973p} method \cite{Opp18,garland1973p}. Trickle-down expresses the following geometric idea: if a simplicial complex is locally very well-connected (in the sense of expansion), then the complex must be composed of disjoint pieces, each globally very well-connected.

To be a bit more precise, a pure simplicial complex of dimension $k - 1$ with $n$ vertices is described by a (possibly weighted) subset of $\binom{[n]}{k}$, which are the highest-dimensional faces of the complex. The ``top links'' of the complex are induced by conditioning these faces to contain a specific set of $k - 2$ vertices and forming the family/distribution of the additional $2$ vertices; these top links can be thought of as a graph whose edges are these $2$-sized subsets. \Citeauthor{Opp18}'s trickle-down tells us that if the top links of the complex are very good spectral expanders, then this ``trickles down'' to ensure the complex itself is a spectral high-dimensional expander provided the mild condition that it is not disconnected at any level. Spectral expansion of the complex means that a natural $1 \to k \to 1$ \emph{up-down walk}, which starts at a vertex, moves to an adjacent face, and then goes to a uniformly random vertex of that face, has a large spectral gap.

Important Markov chains like the Glauber dynamics can naturally be viewed as types of down-up walks on a weighted simplicial complex, and trickle-down has revealed itself to be an invaluable tool for proving sharp mixing time bounds for important classes of distributions like spanning trees, bases of matroids, more generally, distributions induced by log-concave polynomials, and quite recently edge colorings of graphs; see, e.g., \cite{ALOV19,ALOVV21,CGM19,abdolazimi2022matrix}.

\paragraph{Our contribution.} In this work, we extend the reach of the trickle-down approach beyond the context of high-dimensional expanders. We do this by formulating a trickle-down equation in the more general context of \emph{linear-tilt localization schemes}. Localization schemes, introduced by \textcite{chen2022localization}, constitute a framework that naturally generalizes the concept of iterative ``pinning'' in high-dimensional expanders (conditioning on vertices, as in the discussion of ``links'' above). In short, linear-tilt localization schemes generalize the operation of pinning to iterative reweighing of a measure by linear functionals (a.k.a.\ tilts).  The generalized trickle-down equation gives a systematic way to perform backward induction and to show that the original measure is ``well-connected'' as long as the simpler/localized measures are. See \cref{subsec:techniques} for more details.

The new perspective arising from the trickle-down equation allows us to make progress on several important problems in the sampling literature, and also derive important structural consequences like concentration of measure which have so far been beyond the reach of existing tools. We briefly overview a couple of representative applications here, but there are more --- see \cref{subsec:results} and \cref{subsec:examples}.

\paragraph{Example application: mixing in the Sherrington-Kirkpatrick (SK) model.} The SK model is a celebrated spin glass model from statistical physics \cite{sherrington1975solvable,mezard1987spin}. This is an Ising model where the interaction matrix is sampled from the GOE (Gaussian Orthogonal Ensemble). In other words, the matrix $J$ is a random $n \times n$ symmetric matrix with independent entries (for $i < j$)
\[ J_{ij} \sim \Normal{0, \beta^2/n} \]
where $\beta \ge 0$ is the \emph{inverse temperature}, and the induced probability measure $\nu$ on $\set{\pm 1}^n$ is given by
\[\nu(x) \propto \exp\parens*{\dotprod{x, J x}/2}.\]
It was previously known for $\beta < 0.25$ that, with high probability over the choice of $J$, the Glauber dynamics\footnote{This is a Markov chain which at every step, picks a random site $i$ and resamples the spin $X_i$ according to the conditional law $\nu\parens*{\cdot \given X_{\sim i}}$. It is also known as the Gibbs sampler.} mixes in nearly-linear time \cite{bauerschmidt2019very,eldan2021spectral,anari2021entropic}. This threshold does not have any apparent physical meaning (for example, there is no phase transition associated with the Gibbs measure until $\beta = 1$). However, mathematically it was an inherent limit of previous techniques because it is the sharp threshold for an associated functional, arising from the Hubbard-Stratonovich transform, to be \emph{convex}. Analyzing the trickle-down equation for this model arising from stochastic localization, we can break the convexity barrier and still obtain optimal mixing times:
\begin{theorem}[Informal]
    For the SK model, there exists absolute constant $c > 0$ such that up to $\beta = 0.25 + c \approx 0.295$, with high probability over the choice of $J$, Glauber dynamics mixes in $O(n \log n)$ time.
\end{theorem}

This is obtained as a special case of a general result that improves the mixing time bounds for many Ising models in terms of the spectrum of their interaction matrix, \cref{thm:ising-main}.

Another feature of the generalized trickle-down is that it makes sense for continuous models as well as discrete ones. We obtain, analogously to the improvement in the aforementioned result for Ising models, an improvement for higher-spin-dimension variants, namely, the well-known $O(N)$ model, under both the Langevin and Glauber dynamics. See \cref{subsec:examples} for details.

\paragraph{Example application: mixing for antiferromagnetic Ising models on expander graphs.} 
The \emph{antiferromagnetic} Ising model at inverse temperature $\beta \ge 0$ with external field $h \in \R^n$ on a graph $G$ is the probability measure on the hypercube $\set{\pm 1}^n$ given by
\[ \nu(x) \propto \exp \parens*{-\beta \sum_{i \sim j} x_i x_j + \dotprod{h, x}} \]
where $i \sim j$ is the adjacency relation in graph $G$. For a \emph{worst-case} graph, the largest value of $\beta$ up to which polynomial time sampling is possible is known and it corresponds to what is called the ``tree uniqueness threshold'' for the Gibbs measure on the infinite $d$-regular tree \cite{weitz2006counting,mossel2009hardness,CLV20}. The uniqueness threshold is asymptotically $\simeq 1/d$ as $d \to \infty$.
We show this can be greatly improved in expander graphs.
\begin{theorem}[Informal]
Consider a $d$-regular graph with adjacency matrix $A$ on $n$ vertices, and suppose that $\max\set{\abs{\lambda_2(A)}, \abs{\lambda_n(A)}} \le \lambda$. For all $\beta < 1/2\lambda$ there exists an algorithm, which we call the polarized walk, which samples from the antiferromagnetic Ising model in $\widetilde{O}(nd)$ time.
\end{theorem}
For example, in a random $d$-regular graph this means we can handle $\beta$ up to a threshold of $\Theta(1/\sqrt{d})$ instead of $\Theta(1/d)$. We also obtain results for Glauber dynamics when $d$ is small. See \cref{subsec:results} and \cref{subsec:examples} for more discussion.

\Tag{\subsection{Our results}}\Tag<sigconf>{\subsection{Our Results}}\label{subsec:results}
We now proceed to formally state the main results of this work, obtained based on the generalized trickle-down equation. Since these results apply to large classes of models, we give more specialized and concrete examples illustrating their application in the later \cref{subsec:examples}. 

Our main results involve systems of $n$ spins, with each of the $n$ spins taking values on the unit sphere $S^{N-1}=\set{x\in \R^N\given \norm{x}=1}$. We use $\sigma^{N-1}$ to denote the natural ``uniform'' probability measure on $S^{N-1}$. As a special case, the $0$-dimensional unit sphere is the discrete set $\set{\pm 1}$, and we use $\Berpm{}$ to denote the uniform distribution in that case. Our main results bound the covariance matrix $\cov{\nu}$ of the underlying distribution $\nu$, from which approximate tensorization of entropy, abbreviated as ATE, \Tag{(see \cref{sec:fi-and-mixing})} and appropriate mixing time results follow.

For a distribution $\mu$ on (a subset of) $\R^N$, such as $\sigma^{N-1}$ or $\Berpm{}$, and vector $w\in \R^N$, we use $\tilt_w \mu$ to denote the exponentially tilted distribution given by the following density \Tag{(see \cref{def:exponential-tilts})}:
\[ \frac{d\tilt_w \mu}{d\mu}(x)\propto \exp(\dotprod{w, x}). \]

\paragraph{Improved spectral condition for rapid mixing of Glauber dynamics in Ising models.} A distribution $\nu$ on $\set{\pm 1}^n$ is an \emph{Ising model} with interaction matrix $J$ and external field $h$ if 
\[ \nu(x) \propto \exp\parens*{\frac{1}{2} \dotprod{x, J x} + \dotprod{h, x}} \]
for all $x \in \set{\pm 1}^n$. Note that the diagonals of $J$ are arbitrary; they do not affect the measure $\nu$. 

\Tag{\begin{theorem}[\cref{thm:ising} below]}\Tag<sigconf>{\begin{theorem}}\label{thm:ising-main}
Suppose that $\nu$ is an Ising model parameterized by some external field $h \in \R^n$ and interaction matrix $J$ satisfying $J \succeq 0$.
Suppose, without loss of generality, that the diagonal of $J$ is constant, so there exists $\alpha$ such that $J_{ii} = \alpha \ge 0$. Let $\eta = \alpha/\opnorm{J} \in [0,1]$. 
Then
\[ \opnorm{\cov{\nu}} \le q_{\eta}(\opnorm{J}) \]
where $q_{\eta} : \R_{\geq 0} \to \R_{\geq 0} \cup \set{\infty}$ solves the Volterra integral equation
\[ q_\eta(z)=r(\eta z)+\int_0^z q_\eta(y)^2dy \]
with 
\Tag{
\[ r(t)=\E*_{x\sim \Berpm{}, g\sim \Normal{0, 1}}{\cov*{\tilt_{tx+\sqrt{t}g}\Berpm{}}}=\E*_{x\sim \Berpm{}, g\sim \Normal{0, 1}}{1-\tanh^2\parens*{tx+\sqrt{t}g}}\leq 1. \]
}%
\Tag<sigconf>{
\begin{multline*}
r(t)=\E*_{x\sim \Berpm{}, g\sim \Normal{0, 1}}{\cov*{\tilt_{tx+\sqrt{t}g}\Berpm{}}}=\\
\E*_{x\sim \Berpm{}, g\sim \Normal{0, 1}}{1-\tanh^2\parens*{tx+\sqrt{t}g}}\leq 1. 
\end{multline*}
}%
Furthermore, $\nu$ satisfies approximate tensorization of entropy with constant at most
\[ \exp\parens*{\int_0^{\opnorm{J}} q_{\eta}(z) dz }. \]
\end{theorem}
\begin{figure}
    \centering
    \includegraphics[width=\columnwidth]{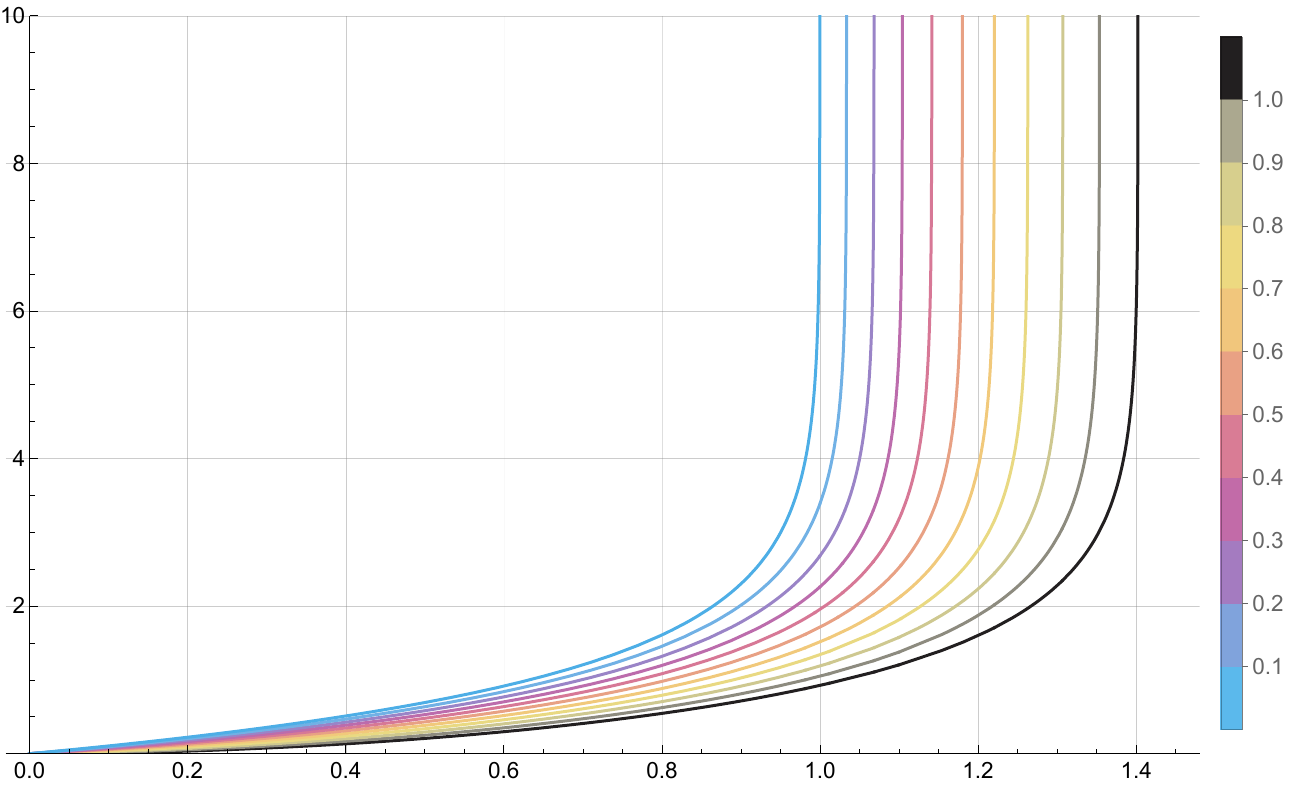}
    \caption{$\log q_{\eta}(z)$ from \cref{thm:ising-main} plotted using numerical integration in Mathematica for $\eta = 0,0.1,\ldots,0.9,1$. The curve for $\eta = 0$ was, before this work, the best known bound for all values of $\eta$ in the Ising model; it asymptotes to $\infty$ at $z = 1$, which is tight due to the phase transition in the Curie-Weiss model \cite{ellis2006entropy}. For $\eta = 0.5$ the function $\log q_{\eta}$ asymptotes to $\infty$ at $z$ slightly above $1.18$ and for $\eta = 1.0$ it asymptotes around $1.40$.}
    \label{fig:qeta}
\end{figure}

\begin{remark}[Spectral interpretation]\label{rmk:interpretation}
As mentioned above, there is no a priori significance to the diagonal entries of $J$. For this reason, some works, including this one, use the convention that $J$ is positive semidefinite without loss of generality. On the other hand, it is probably most common to parameterize the Ising model in terms of an interaction matrix with zero diagonal. Such a matrix will typically have both positive and negative eigenvalues because the mean of its spectral distribution is zero. Fortunately, the above theorem has a very clean interpretation in terms of the zero-diagonal parameterization.

Let $\tilde J$ be a matrix with zero diagonals which specifies the interactions in the Ising model, so $\nu(x) \propto \exp(\dotprod{x, \tilde J x}/2 + \dotprod{h, x})$. If we add a copy of the identity to $\tilde J$, we can always ensure the resulting matrix is positive semidefinite without changing the corresponding measure. More specifically, we can define $J = \tilde J + \alpha I$ where $\alpha = -\lambda_{\min}(\tilde J)$. Then when we apply \cref{thm:ising-main}, our $\alpha$ will be the same as the one appearing in the theorem statement. We have 
\[ \opnorm{J} = \alpha + \lambda_{\max}(\tilde J) \]
so
\[ \eta = \frac{\alpha}{\opnorm{J}} = \frac{1}{1 + \abs{\lambda_{\max}(\tilde J)/\lambda_{\min}(\tilde J)}} \in [0,1]. \]
 We have $\eta \approx 0$ when the spectrum of $\tilde J$ has only very tiny negative eigenvalues. In that case, \cref{thm:ising-main} reduces to the same bound as the existing results \cite{bauerschmidt2019very,eldan2021spectral,anari2021entropic}. However, in most applications, $\tilde J$ has non-negligible negative eigenvalues, so $\eta > 0$ and then this result significantly improves on the previous work --- see \cref{fig:qeta} and examples in \cref{subsec:examples}. For example, in many cases, $\tilde J$ has an approximately symmetrical spectrum and so $\eta \approx 1/2$. 
\end{remark}

\paragraph{Langevin dynamics in the $O(N)$ model.} The $O(N)$ model \cite{stanley1968dependence} is a famous class of models in statistical physics parameterized by an ambient dimension $N$. In this model, there are $n$ interacting spins and each spin lives on the $(N-1)$-dimensional unit sphere $S^{N-1}$; the name $O(N)$ is a reference to symmetry under the $N$-dimensional orthogonal group. In the case $N = 1$, the $0$-dimensional unit sphere is simply the points $\set{\pm 1}$ and so the $O(1)$ model is just the Ising model. The case $N=2$ is the XY model and the case $N=3$ is the (classical) Heisenberg model, which are very famous models in statistical physics; for example, the XY model on a two-dimensional lattice exhibits the celebrated BKT phase transition \cite{kosterlitz2018ordering,berezinskii1971destruction} for which the 2016 Nobel Prize in Physics was awarded.

When $N \ge 2$ the unit sphere becomes connected, so the (manifold) Langevin dynamics is a natural sampling algorithm to analyze. \Textcite{bauerschmidt2019very} proved a log-Sobolev inequality for this dynamics under a bounded operator norm assumption on the interaction matrix $J$. Analogous to the Ising case, we prove a new result with a refined spectral condition that yields improved bounds in most applications. An interesting feature of this result is that for each $N$, the bound is given by an explicit rational function.
\Tag{\begin{theorem}[\cref{thm:ON} below]}\Tag<sigconf>{\begin{theorem}}\label{thm:ON-main}
Suppose $N \ge 2, n \ge 1$ and that $\nu$ is the probability distribution on $(S^{N - 1})^{n}$ with probability density
\[ \frac{d\nu}{d\mu}(x) \propto \exp\parens*{\frac{1}{2} \sum_{1 \le i,j \le n} J_{ij} \dotprod{ x_i, x_j } + \sum_{i = 1}^n \dotprod{ h_i, x_i } } \]
with respect to the uniform measure $\mu=(\sigma^{N-1})^n$ on $(S^{N - 1})^{n}$. Here $J$ and $h$ are parameters and we assume the interaction matrix $J$ satisfies $J \succeq 0$. Suppose, without loss of generality, that the diagonal of $J$ is constant, so there exists $\alpha$ such that $J_{ii} = \alpha \ge 0$. Let $\eta = \alpha/\opnorm{J} \in [0,1]$. 
Then\footnote{This result is improved compared to the conference version of this paper \cite{conference-version}.}
\[ \opnorm{\cov{\nu}} \le q_{\eta,1/N}(\opnorm{J}) \]
where $q_{\eta,1/N} = (1/N)Q_{\eta}(z/N)$ and
$Q_{\eta} : [0,s(\eta)) \to \mathbb{R}_{\ge 0}$ is given by
\begin{equation} Q_{\eta}(z) = \eta \frac{-\lambda_1 \lambda_2^2 (\eta z + 1)^{\lambda_1 - \lambda_2} + \lambda_1^2 \lambda_2}{\lambda_2^2(\eta z + 1)^{\lambda_1 - \lambda_2 + 1} - \lambda_1^2 (\eta z + 1)}, \label{eqn:big-Q}
\end{equation}
with $\lambda_1 > 0 > \lambda_2$ the two roots of the quadratic equation 
$\eta \lambda^2 - \eta \lambda - 1 = 0$, explicitly
\[ \lambda_1 = \lambda_1(\eta) = \frac{\eta + \sqrt{\eta^2 + 4\eta}}{2 \eta}, \qquad \lambda_2 = \lambda_2(\eta) = \frac{\eta - \sqrt{\eta^2 + 4\eta}}{2 \eta}, \]
and where 
\[ s(\eta) = \frac{1}{\eta} \left[\left|\frac{\lambda_1}{\lambda_2}\right|^{2/(\lambda_1 - \lambda_2)} - 1\right]. \]
Furthermore, the log-Sobolev constant of the Langevin dynamics is at least
\[ C_N \exp\parens*{-\int_0^{\opnorm{J}} q_{\eta, 1/N}(z) dz } \]
where $C_N > 0$ is a constant depending only on $N$, inherited from \cite{zhang2011uniform}. 
\end{theorem}
\begin{figure}
    \centering
    \includegraphics[width=\columnwidth]{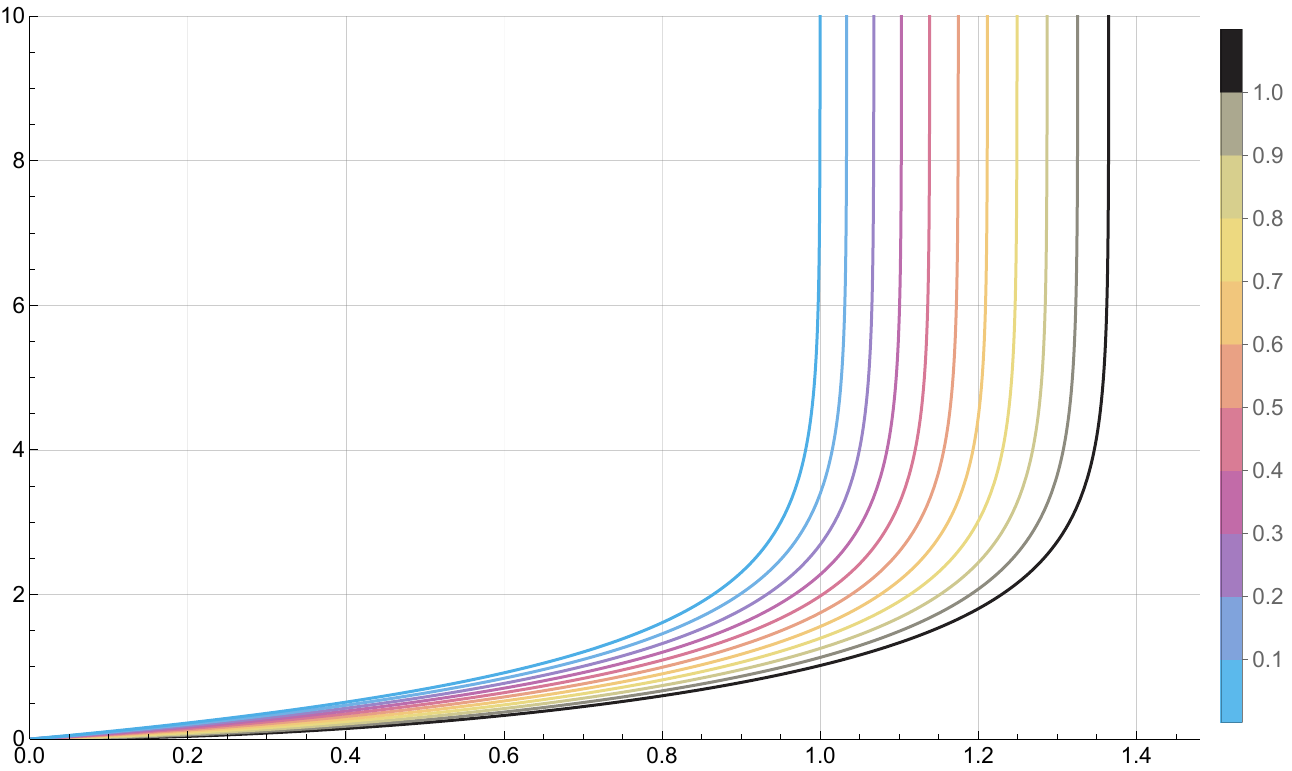}
    \caption{$\log Q_{\eta}(z)$ from \cref{thm:ON-main} plotted using numerical integration in Mathematica for $\eta = 0,0.1,\ldots,0.9,1$. For the $O(N)$ model, the actual function bounding the covariance is $q_{\eta,1/N}(z) = (1/N)Q_{\eta}(z/N)$. So for example, in the case of the XY model which is $N = 2$, the $\eta = 0.5$ curve asymptotes to $\infty$ slightly above $2.34$ and for the Heisenberg model it asymptotes to $\infty$ slightly above $3.52$. Similarly for $\eta = 1$, when $N = 2$ the resulting curve asymptotes to $\infty$ slightly above $2.73$ for XY and at around $4.09$ for Heisenberg.}
    \label{fig:big-q-eta}
\end{figure}
\begin{example}
The equations simplify considerably in the case $\eta = 1/2$: we have that $\lambda_1 = 2$ and $\lambda_2 = -1$,
\[ Q_{1/2}(z) = \frac{-(z/2 + 1)^{3} - 1}{(z/2 + 1)^{4} - 2 (z/2 + 1)} = \frac{48 + 24 z + 12 z^2 + 2 z^3}{48 - 24 z^2 - 8 z^3 - z^4}, \]
and 
\[ s(1/2) = 2 \left[4^{1/3} - 1\right] \approx 1.1748. \]
\end{example}
\begin{remark}
    Our methods automatically yield an additional result, \cref{thm:hard-to-use}, which for each particular value of $N$ can likely yield a small improvement in the threshold obtained by \cref{thm:semilogconcave}. We expect the improvement is relatively small, and determining the quantitative guarantee requires a nontrivial computation for each value of $N$, so we did not further pursue this direction. 
\end{remark}

As stated, this result applies to the continuous-time (manifold) Langevin dynamics. To get an algorithm, some discretization scheme is usually applied. Fortunately, the log-Sobolev inequality also applies polynomial time guarantees for the discretized Langevin dynamics on manifolds \cite{li2020riemannian}. 

\paragraph{Glauber dynamics in the $O(N)$ model: nearly linear time sampling.} Besides the Langevin dynamics, there is another very natural dynamics to consider in the $O(N)$ model --- the Glauber dynamics.

Compared to Langevin, Glauber dynamics has the advantage that it doesn't require discretization or choice of step size to implement. All known discretization schemes for the Langevin dynamics incur extra dimension-dependent factors, even in the simplest Euclidean setting \cite[see, e.g.,][]{lee2021lower}. Concretely, this means that while we established a ``dimension-free'' LSI constant, combining it with the discretization analysis of Langevin as in \cite{li2020riemannian} pays an extra dimension factor of $n$, yielding a runtime guarantee of $\widetilde{O}(n^2 d)$ for a graphical model of maximum degree $d$ (i.e., row-sparsity $d$ in $J$) even in the simplest case $N = 2$. 

We improve on this by proving that under the same conditions as before, the mixing time of Glauber dynamics is $O(n \log n)$, i.e., as fast as possible, and that each update step of the dynamics can be implemented very quickly. Together, these yield a nearly linear time sampler for our class of models. Here is the precise result:
\Tag{\begin{theorem}[\cref{thm:ON glauber} below]}\Tag<sigconf>{\begin{theorem}}
In the same setting as \cref{thm:ON-main}, the probability measure $\nu$ satisfies approximate tensorization of entropy (ATE) with constant
\[ \exp\parens*{\int_0^{\opnorm{J}} q_{\eta, N}(z) dz }. \]
Furthermore, single steps of Glauber dynamics can be $\epsilon$-approximately implemented using $O(dN + \log(1/\epsilon)\cdot (\log N + \log \log (1+B))$ arithmetic operations assuming at most $d$ nonzeros in the corresponding row of $J$, and where $B = \max_i \sum_j \abs{J_{ij}} + \norm{h_i}$. In other words, we can sample from a distribution that is $\epsilon$-close in total variation distance to the conditional distribution at the spin chosen to be updated. To sample from a distribution $\varepsilon$-close to $\nu$ in total variation distance, we run the Glauber dynamics for $T=n \log(n/\varepsilon)$ steps and set $\epsilon =\frac{\varepsilon}{T}$, resulting in an overall runtime of $\max\set{\operatorname{nnz}(J), n}\cdot N\cdot \poly\log(1/\varepsilon, n, N, \log(1+B))$.
\end{theorem}

\paragraph{Glauber dynamics over semi-log-concave base measures.} In the  mathematical physics literature, it is common to study more general distributions such as ``soft spin'' systems where the spin is $\mathbb R$-valued. See e.g. \cite{simon1973varphi} for some references and discussion in the context of field theory.
Our methods can give a general bound in terms of semi-log-concavity properties of the base measure of the spin system. We say that a measure is $\beta$-semi-log-concave if $\cov{\mathcal T_w \mu} \preceq \beta I$ for every $w \in \mathbb{R}^n$ \cite{eldan2022log}; see the preliminaries for more discussion.
\begin{theorem}[\cref{thm:semilogconcave} below]
Suppose that $\mu = \bigotimes_{i = 1}^n \mu^{(i)}$ is a product measure where each $\mu^{(i)}$ is a probability measure on $\mathbb R^N$. Let $J \succeq 0$ with $J_{ii} = \alpha$ for all $i \in [n]$, and define the measure $\nu$ by its density
\[ \frac{d\nu}{d\mu}(x) \propto \exp\left(\frac{1}{2} \sum_{i,j} J_{ij} \langle x_i, x_j \rangle \right). \]
Let $\eta = \alpha/\|J\|_{OP}$.
Suppose that for all $i \in [n]$ and $s \in [0,\alpha]$ the measure $\mathcal{G}_{-s} \mu^{(i)}$ defined by
\begin{equation}\label{eqn:strong-slc=intro}
\frac{d\mathcal{G}_{-s} \mu^{(i)}}{d\mu^{(i)}}(x) \propto \exp(s\|x\|^2/2) 
\end{equation}
exists and is $\rho(s) > 0$ semi-log-concave. Then $\|\cov{\nu}\|_{OP} \le q_{\eta,\rho}(\|J\|_{OP})$ where $q_{\eta,\rho}$ solves the integral equation
\begin{equation}\label{eqn:qsemilogconcave-intro}
q_{\eta,\rho}(z) = \frac{1}{\eta z + [\rho(\eta z)]^{-1}} + \int_0^z q_{\eta,\rho}(s)^2 ds
\end{equation}
and ATE holds with constant at most
\[ \exp\left(\int_0^{\|J\|_{OP}} q_{\eta,\rho}(z) dz\right).  \]
\end{theorem}
The covariance bounds we obtain for the $O(N)$ model are a special case of this result --- when $\rho(s) = \rho$ is constant, the integral equation admits an explicit solution $q_{\eta,\rho}(z) = \rho Q_{\eta}(\rho z)$ where $Q_{\eta}$ is from \cref{eqn:big-Q}. See \cref{thm:fixed-semilogconcave}.

\paragraph{Nearly linear time sampling for a class of Ising models via polarized walks.} Antiferromagnetic Ising models are those where neighboring spins repel rather than attract each other. An impressive line of work has shown that for the antiferromagnetic Ising model on a worst-case $d$-regular graph, there is a computational phase transition at the uniqueness threshold of the infinite $d$-regular tree; see, e.g., \cite{sly2012computational,weitz2006counting,galanis2016inapproximability} for a few of the relevant works. This means that above a certain precise threshold for the strength of interactions in the model (inverse temperature), no polynomial time sampling algorithm exists under standard hypotheses from computational complexity.

What about non-worst-case graphs? Based on the picture coming from statistical physics as well as related rigorous results, we do not expect the uniqueness threshold to be relevant for random graphs; see \cite{coja2022ising} and references within. In particular, for $d$-regular random graphs sampling should be possible up to a threshold of $\Theta(1/\sqrt{d})$ instead of the $\Theta(1/d)$ scaling of the uniqueness threshold. For constant degrees $d$, this was recently proven by \textcite{koehler2022sampling}, where a polynomial time algorithm inspired by variational inference was developed. This was recovered as part of a general result improving sampling guarantees for all expander graphs.

Can we improve this result to hold for all degrees $d$, to run in nearly linear time, and to be achieved by a simple Markov chain? Our \cref{thm:ising-main}, while applicable to antiferromagnetic Ising models, is not well-adapted to this setting because adjacency matrices of expander graphs have a large ``trivial'' eigenvalue, and its effect on the model is not so trivial. 

Nevertheless, there is a natural class of models we can analyze using our techniques which captures both the antiferromagnetic Ising model on expanders as well as some other important applications like fixed-magnetization models. This class of models can be understood as spectrally bounded Ising models with an arbitrarily strong \emph{confining potential} proportional to $ \parens*{\sum_i x_i - k}^2$ for any $k \in \Z$. This potential\footnote{In the actual statement, we do not include the parameter $k$ explicitly, since it can be absorbed into the external field of the model.} allows the measure to localize near a slice $\set{x\given \sum_i x_i \approx k}$. Since this class of models includes those with fixed magnetization, Glauber dynamics may not even be ergodic. Fortunately, an interesting connection to the geometry of polynomials suggests a different Markov chain which does rapidly mix! We call this chain the \emph{polarized walk} and describe it in more detail in \cref{subsec:techniques}.

Concretely, we establish the following result, which tells us that the polarized walk samples successfully in \emph{nearly linear time} from a large class of models:
\Tag{\begin{theorem}[\cref{thm:ising projected downup} below]}\Tag<sigconf>{\begin{theorem}}\label{thm:ising projected downup intro}\label{thm:ising-projected-downup-main}
Suppose that $\nu$ is a measure on the hypercube $\set{\pm 1}^n$ of the form
\[ \nu(x) \propto \exp\parens*{\frac{1}{2} \dotprod{x, J x} + \dotprod{h, x} - \frac{\gamma}{2n}\parens*{\sum_i x_i}^2} \]
for some $J \succeq 0$, $\gamma \ge 0$, and $h \in \R^n$, and suppose that $\opnorm{J} < 1/2.$ Then:
\begin{enumerate}
    \item The polarized down-operator has $\parens*{1-\frac{1-2\opnorm{J}}{n}}$-entropy contraction with respect to $\nu$. 
    \item Hence, the polarized walk on $\nu$ mixes within $\epsilon$ total variation distance in $O\parens*{\frac{1}{1-2\opnorm{J}}n\log(n/\epsilon)}$ steps. 
    \item Let 
\[ Q = J - \frac{\gamma}{n} \1\1^\intercal \]
and let $d$ be the maximum number of non-diagonal nonzero entries in a row of $Q$. Each step of the projected down-up walk can be implemented in $ O(d\log n)$ amortized time, so the polarized walk outputs a sample within $\epsilon$ total variation distance of $\nu$ in total runtime $O\parens*{\frac{1}{1-2\opnorm{J}} nd\log(n)\log(n/\epsilon)}$
\end{enumerate}
\end{theorem}
\Tag{\begin{corollary}[\cref{cor:random d regular} below]}\Tag<sigconf>{\begin{corollary}}
Let $\nu$ be the antiferromagnetic Ising model with parameter $\beta$ on a random $d$-regular graph $G$. Suppose $0\leq \beta \leq (1-\delta)/(8\sqrt{d-1})$ for a  constant $\delta > 0.$  With probability $1-o(1)$ over the random instance $G$, we can sample from $\hat{\nu}$ s.t.\ $\dTV{\hat{\nu}, \nu} \leq \epsilon$ in time $O(\delta^{-1} n d\log(n)\log(n/\epsilon))$.
\end{corollary}

\paragraph{Glauber dynamics on low-degree expanders.} As discussed above, there is no hope of Glauber dynamics mixing for the general class of models studied in \cref{thm:ising-projected-downup-main}. However, it is still possible Glauber mixes rapidly for some subclasses of models. In the case of antiferromagnetic Ising models of bounded degree on expander graphs, we prove that the Glauber dynamics does indeed mix in $O(n \log n)$ time. One proof of this result is based on ideas from \cite{koehler2022sampling} and in fact, yields mixing for a larger class of models\Tag<sigconf>{; details are deferred to the full version of this paper. }\Tag{(see \cref{thm:klr-based}); in \cref{rmk:trickledown-alternative} we explain a variant of the result which is obtained using a trickle-down approach instead.}
\Tag{\begin{theorem}[\cref{cor:glauber ising low degree} below]}\Tag<sigconf>{\begin{theorem}}
	Suppose that 
	\[ \nu(x) \propto \exp\parens*{-\frac{\beta}{2} \dotprod{x, A x} + \dotprod{h, x}} \]
	where $A$ is the adjacency matrix of a $d$-regular graph on $n$ vertices and suppose that $\max\set{\abs{\lambda_2(A)}, \abs{\lambda_n(A)}} \leq \lambda$. If $\lambda \beta \le 1 - 1/c$ for some $c > 0$, then $\nu$ satisfies approximate tensorization of entropy with constant at most $e^{c e^{O(\beta d)}}$ and the Glauber dynamics on $\nu$ mixes in $O\parens*{e^{c e^{O(\beta d)}}\cdot n\log n}$ steps.     
\end{theorem}

\Tag{\subsection{Our techniques}}\Tag<sigconf>{\subsection{Our Techniques}}\label{subsec:techniques}
In this section, we give a brief overview of some of the techniques in this paper. In this part, we largely focus on the motivating applications involving the Glauber dynamics and polarized walk in the Ising model.

\paragraph{Challenge: nonconvexity of the Hubbard-Stratonovich transform.} In the continuous domain, there has been a long and mathematically deep study of sampling \emph{log-concave} distributions, i.e., probability densities $\propto e^{f(x)}$ where $f$ is a concave function. In comparison, no distribution supported on the hypercube $\set{\pm 1}^n$ or products of unit spheres (as in the $O(N)$ model) can be log-concave because its support is not even a convex set. 

Nevertheless, many previous works \cite{brascamp2002some,dyson1978phase,bauerschmidt2019very,eldan2021spectral,chen2022localization,anari2021entropic,koehler2022sampling} have been able to analyze Ising and $O(N)$ models using the celebrated \emph{Hubbard-Stratonovich (HS) transform} \cite{hubbard1959calculation}. There are different ways to view this trick; at its heart, it essentially corresponds to applying the Fourier transform. However, a more probabilistic view is to say that we take $X \sim \nu$ for $\nu$ an Ising model with interaction matrix $J$ and define a random vector 
\[ Y = X + G, \qquad G \sim \Normal{0,\Sigma}, \qquad \Sigma = J^{-1}. \]
For this carefully chosen covariance matrix $\Sigma$, it can be checked via the Bayes rule that the posterior law $X \mid Y$ becomes a product measure. This is the Hubbard-Stratonovich trick. As a result, the task of sampling the \emph{discrete random vector} $X$ and \emph{continuous random vector} $Y$ become computationally equivalent. In particular, we can sample $X$ and $Y$ efficiently if the law of $Y$ ends up being log-concave. Moreover, more sophisticated analyses use the log-concavity of $Y$ to derive rapid mixing of the Glauber dynamics for $X$. 

This trick is elegant and it is \emph{tight} in a certain sense: for the Curie-Weiss/mean-field Ising model,  
\begin{equation}\label{eqn:cw}
p(x) \propto \exp\parens*{\frac{\beta}{2n} \parens*{\sum_i x_i}^2} 
\end{equation}
with $x \in \set{\pm 1}^n$ and inverse temperature $\beta \ge 0$, the high-temperature or rapid mixing regime of the model ($\beta \le 1$) exactly corresponds to the set of parameters where the Hubbard-Stratonovich transform is log-concave.

However, it is less clear if this analysis is tight for other models. For the Sherrington-Kirkpatrick model, the Hubbard-Stratonovich transform is log-concave exactly when $\beta \le 0.25$. Is this barrier fundamental? No. To go beyond this threshold (and prove all of our main results listed above), we need to develop new arguments which \emph{do not rely on log-concavity.} 

\begin{remark}
The papers \cite{eldan2021spectral,anari2021entropic} use an alternative to the Hubbard-Stratonovich transform trick called \emph{needle decomposition}, which decomposes the Ising model into Ising models with rank-one interaction matrices \cite[cf.\ analogous notions in convex geometry,][]{kannan1995isoperimetric}. Although this is a different decomposition, it runs into the same barrier --- the Curie-Weiss model is a rank-one model and the HS transform analysis is tight for it. The works \cite{pspin,anari2023universality} use an inductive approach instead, but only obtain results up to comparatively small values of $\beta$. 
\end{remark}

\paragraph{Trickle-down equation in localization schemes.} The way we go beyond the log-concavity threshold is by finding a more ``intrinsic'' approach to analyze the mixing time of discrete distributions. From a long line of previous work \cite[see, e.g.,][]{KO18,cryan2019modified,AL20,ALO20,anari2021entropic,alimohammadi2021fractionally,chen2021optimal,anari2021entropic2,bauerschmidt2022log,chen2022localization}, we know that there are deep connections between proving mixing time bounds and establishing control of the covariance matrix of a distribution \emph{under arbitrary tilts or pinnings}. As a reminder, exponential tilts are reweighings of a distribution $p(x)$ by a factor proportional to $\exp(\dotprod{w, x})$ for some external field vector $w$. \Tag{See also \cref{subsec:sl} of the preliminaries.} So to prove our results, we develop a systematic method to control the covariance matrix of all of these tilts.

To do this, we take inspiration from the celebrated \emph{trickle-down phenomenon} in high-dimensional expanders \cite{Opp18}, and develop a general trickle-down method that applies in the abstract setting of \emph{linear-tilt localization schemes}. A linear-tilt localization scheme \cite{chen2022localization} for a probability measure $\nu_0$ induced by a martingale difference sequence $Z_t$ corresponds to a random sequence of measures
\[ \frac{d\nu_{t + 1}}{d\nu_t}(x) = 1 + \dotprod{x - \mean{\nu_t}, Z_{t + 1}}. \]
This general framework captures the usual pinning localization scheme used in high-dimensional expanders, stochastic localization, and other localization schemes like negative-field localization, etc.\ \cite{chen2022localization}. In this context, we prove the following generalized trickle-down equation (\cref{eqn:trickledown-1}, see \cref{sec:trickledown} for precise definitions and notation):
\[ \cov{\nu_t}-\cov{\nu_t}\cov{Z_{t+1}\given \F_t}\cov{\nu_t}=\E*{\cov{\nu_{t+1}}\given \F_t} \]
The key point is that this recursion allows us to perform backward induction over time to control the covariance matrix at time $t$ by the covariance matrix at time $t + 1$. In particular, we identify the exact analogue of trickle-down in stochastic localization, which we then use to prove our results. 

\paragraph{Analyzing trickle-down via geometry of the base measure.} A na\"ive application of the trickle-down method will only reprove the existing results obtained via the log-concavity of the Hubbard-Stratonovich transform. The reason for this is that for natural analogues of the distributions we study, these thresholds would actually be tight; we expand on this in the below example.

\begin{example}[Gaussian SK model] For example, we could consider a ``Gaussian SK model'' which is defined just like the usual SK model but w.r.t.\ the Gaussian measure instead of the hypercube, so that its probability density function is 
\[ p(x) \propto \exp\parens*{\beta \dotprod{x, J x}/2 - \norm{x}_2^2/2}, \]
where $J$ is a scaled GOE matrix plus twice the identity, i.e., it is symmetric and independently $J_{ij} \sim \Normal{0,1/n}$ for $i < j$ and $J_{ii} = 2$ for $i \in [n]$. The diagonal is included to ensure $J$ is approximately PSD --- this is a necessary convention because the stochastic localization process we use with trickle-down needs to use the driving matrix $J^{1/2}$. Both the base Gaussian measure $\propto e^{-\norm{x}_2^2/2}dx$ in this example and the uniform measure on the hypercube have the same covariance, and in the Gaussian model the condition $\beta < 1/4$ is clearly sharp for rapid mixing. More precisely, the Gaussian model does not even exist when $\beta > 1/4$, since the top eigenvalue of $J$ is concentrated about $4$, and if $\beta J/2-I$ has eigenvalues $\geq 0$, then the integral $\int p(x)dx$ becomes infinite.
\end{example} 

So to obtain any actual improvement for SK and other Ising models, we \emph{need} to use something fundamental about the fact that our distribution is supported on the hypercube $\set{\pm 1}^n$. 
Ultimately, the improvement arises from specific concentration/anticoncentration properties which the Gaussian measure lacks.
Here is an example, which helps illustrate the key role that the diagonal entries of $J$ play in our analysis:
\begin{lemma}
If $X \sim \nu$ where the probability measure $\nu$ on $\set{\pm 1}^n$ is the Ising model with interaction matrix $J \succeq 0$ and external field $h$, then for any site $i \in [n]$ we have
\[ \P_{\nu}{\abs{\dotprod{J_i, X} + h_i} \ge J_{ii}} \ge 1/2. \]
\end{lemma}
\begin{proof}
From the definition of the Ising model, we can observe that the conditional law at site $i$ satisfies
\[ \E_{\nu}{X_i \given X_{\sim i}} = \tanh(\dotprod{J_{i,\sim i}, X_{\sim i}} + h_i), \]
where $X_{\sim i}$ denotes $X$ with coordinate $i$ removed. From this equation and the monotonicity of the $\tanh$ function, we see that $X_i$ and $\dotprod{J_{i,\sim i}, X_{\sim i}} + h_i$ are positively correlated. So conditional on any value of $\dotprod{J_{i,\sim i}, X_{\sim i}} + h_i$, the random variable $X_i$ has at least a $50\%$ chance of having the same sign as this quantity, which proves the result. 
\end{proof}
The reason this is useful is that the random vector $J X + h$ naturally arises as part of the external field for large times in the trickle-down procedure we use\Tag{ (see proof of \cref{lem:tbound}; in fact, what we really do is observe a connection between a multi-spin and single-spin stochastic localization processes)}; the fact that it is large means the relevant product measure on the hypercube has a large bias, which means it has a small covariance. The actual anticoncentration analysis we use is more sophisticated and yields a tighter quantitative bound\Tag<sigconf>{.}\Tag{: see \cref{sec:ising} for details, and \cref{sec:semi} for more general results.}

\paragraph{Polarized walks and the geometry of polynomials.} Our general result for Glauber dynamics in the Ising model cannot handle models of the form studied in \cref{thm:ising projected downup intro}. This reflects a limitation of the Glauber dynamics itself --- in the fixed-magnetization limit $\gamma \to \infty$ with $h = 0$, the distribution becomes close to supported on the slice of the hypercube $\set{ x \given \sum_i x_i = 0 }$ and the Glauber dynamics will not be ergodic. This is because every two points on the slice differ in at least two positions.

Fortunately, our analysis based upon the geometry of polynomials naturally suggests a different random walk, the \emph{polarized walk}, which we show will sample in nearly linear time. We formally describe the dynamics in terms of its transition matrix in \Tag{\cref{sec:expander}}\Tag<sigconf>{the full version of this paper}, but \Tag{first} explain the intuition behind it here. The ordinary Glauber dynamics can naturally be viewed as the composition of a down step (erasing the spin at a randomly chosen site) and an up step (resampling the erased spin). Since this is not ergodic for large values of $\gamma$, we instead use the following \emph{polarized down step}:  
\begin{enumerate}
    \item Given as input state $x \in \set{\pm 1}^n$, select a coordinate $i$ uniformly at random from $[n]$.
    \item Output $x$ with entry $x_i$ set to $-1$.
\end{enumerate}
Compared to the down step in the Glauber dynamics, which erases the spin at a particular site, the \emph{polarized} down step is different in two key ways: first, it can only change coordinates of $x$ which are equal to $+1$, and second, if it does pick such a coordinate $i$ where $x_i = +1$, it sets the coordinate to $-1$ instead of erasing it. 

Given this down step, the \emph{polarized up step}  takes as input a vector $y \in \set{\pm 1}^n$ and samples from the posterior distribution on $x \sim \nu$ assuming that $y$ is the output of the projected down operator. This is completely analogous to the definition of the usual up step. Because the projected down step only replaces $+1$ spins by $-1$ spins, the projected up operator only replaces $-1$ spins by $+1$ spins. \Tag{See \cref{def:polarized} for the explicit transition matrix of this operator.} Note that the polarized walk does not treat $+1$ and $-1$ symmetrically, even though they are symmetrical in the definition of the model. This means that an alternative dynamics with $+1$ and $-1$ reversed can also be used.

Once we choose a way to break the symmetry, these dynamics naturally arise from a construction in the geometry of polynomials called \emph{polarization}, which is a natural way to map a polynomial to a multiaffine polynomial while preserving the property of being \emph{log-concave}/\emph{Lorentzian} \cite{anari2021log,branden2020lorentzian}. These polynomials show up in our work by identifying a spin vector $x \in \set{\pm 1}^n$ with the set $x_+$ of indices with $+1$ spins, and then looking at the \emph{generating polynomial} $\sum_i \mu(S) x^S$ of the corresponding probability measure $\mu$ on sets. \Tag{See \cref{subsec:poly} for details.}

By setting up the right stochastic localization process, we decompose our original Ising measure into a mixture of negatively-spiked rank one Ising models. As a remark, this is in contrast to \cite{eldan2021spectral,anari2021entropic}, where the decomposition used positively-spiked rank one models. More specifically, we obtain models of the form
\[ p(x) \propto \exp\parens*{-\frac{\gamma}{2n}\parens*{\sum_i x_i}^2 + \dotprod{h, x}}. \]
This can be thought of as the antiferromagnetic analogue of the classical Curie-Weiss model (\cref{eqn:cw}); interestingly, this simple-looking statistical physics model ends up connecting in a very elegant way with the geometry of polynomials. We use this connection to prove that the polarized dynamics mixes rapidly in this model, and then using trickle-down along with some new facts about entropy contraction in stochastic localization\Tag{ (see \cref{lem:supermartingale})}, we establish the mixing time bound in \cref{thm:ising projected downup intro}.

Quickly implementing the up step of the polarized dynamics is nontrivial --- to do this, we build a data structure based on self-balancing binary search trees\Tag{ (see \cref{prop:data structure})}. This completes the construction of the nearly linear time sampling algorithm.

\Tag{\subsection{Some example applications}}\Tag<sigconf>{\subsection{Some Example Applications}}\label{subsec:examples}
\paragraph{Sherrington-Kirkpatrick (SK) model} 
Recalling the discussion earlier, in the SK model the matrix $J$ is a random $n \times n$ symmetric matrix with independent entries (for $i < j$)
\[ J_{ij} \sim \Normal{0, \beta^2/n} \]
where $\beta \ge 0$ is the \emph{inverse temperature}. It was previously known for $\beta < 0.25$ that the Glauber dynamics mixes in nearly linear time \cite{bauerschmidt2019very,eldan2021spectral,anari2021entropic}. From the famous results of Wigner \cite[see, e.g.,][]{anderson2010introduction}, we know that as $n \to \infty$, the smallest and largest eigenvalues of the zero-diagonal interaction matrix $J$ are $-2 + o(1)$ and $2 + o(1)$ respectively. Hence, applying \cref{thm:ising-main} to the equivalent interaction matrix $J-\lambda_{\min}(J)\cdot I$ with $\eta = 1/2 + o(1)$ (see \cref{rmk:interpretation}) shows that our general result implies nearly linear time mixing up to $\beta \approx 0.295$; see \cref{fig:qeta}.

\paragraph{$d$-regular diluted SK model.} In spin glass theory, \emph{diluted} versions of mean-field spin glasses, which are supported on sparse random graphs, have long been studied \cite[see, e.g.,][]{talagrand2010mean}; for example, with the motivation that sparse interactions are more realistic models of various physical phenomena. One natural diluted version of the SK model has its support on a random $d$-regular graph; for example, in the case of Rademacher disorder, for each edge $(i,j)$ the edge weight $J_{ij}$ would be sampled i.i.d.\ and uniformly from $\set{\pm \beta}$. In this case, the zero-diagonal interaction matrix has operator norm at most $2\sqrt{d - 1} + o(1)$ with high probability by a version of \citeauthor{Friedman2003APO}'s theorem \cite{deshpande2019threshold,Friedman2003APO}. Thus, the best previous results yield $O(n\log n)$ mixing time up to the threshold $\frac{0.25}{\sqrt{d - 1}}$ and our result with $\eta = 1/2 + o(1)$ improves the guarantee to hold up to $\approx \frac{0.295}{\sqrt{d - 1}}$.

\paragraph{Hopfield networks.} Hopfield networks \cite{little1974existence,hopfield1982neural,pastur1977exactly} are a neural model of associative memory that have been hugely influential and extensively studied. Formally, given i.i.d.\ random patterns $\eta_1,\ldots,\eta_m$ sampled uniformly from $\set{\pm 1}^n$, the Hopfield network at inverse temperature $\beta \ge 0$ is the Ising model with interaction matrix
\[ J = \frac{\beta}{2n} \sum_{v = 1}^m \eta_v \eta_v^\intercal. \]
This is thought of as a ``Hebbian'' learning rule because for each memory $\eta_v$ and for ``neurons'' $i$ and $j$, the term $(\eta_v)_i (\eta_v)_j$ is positive if $(\eta_v)_i = (\eta_v)_j$ and negative otherwise. Therefore if $J$ is thought of as the ``wiring'' of the neurons, then for each pattern all of the neurons which ``fire together'', i.e., have the same spin, are ``wired together''.

When the number of memories $m$ is a constant, there is a polynomial time sampling algorithm for this model for all $\beta$ \cite{koehler2022sampling}. But this is not the only regime of interest. Perhaps the most studied setting is when $m$ and $n$ jointly go to infinity at a fixed aspect ratio $m/n \simeq \lambda$. In random matrix theory, statistics, and other areas this is called a ``proportional scaling limit'', and in this limit the spectrum of the interaction matrix $J$ will go to a scaled version of the celebrated \emph{\citeauthor{marchenko1967distribution} law} \cite{marchenko1967distribution}. From the definition of the model, we know that the diagonal entries all equal $\frac{\beta m}{2n} = \frac{\beta \lambda}{2}$. This means the spectrum of the zero-diagonal matrix $\tilde J=J-\frac{\beta\lambda}{2}\cdot I$ is a shifted and scaled version of the \citeauthor{marchenko1967distribution} law. For every value of $\lambda$, applying our main result will improve the threshold for rapid mixing compared to prior works \cite{bauerschmidt2019very,eldan2021spectral,anari2021entropic}.

Unlike the previous examples, the optimal choice of $\eta$ will vary depending on $\lambda$ because the \citeauthor{marchenko1967distribution} law is asymmetrical. The precise improvement will depend on the particular value of $\lambda$. For example, when $\lambda = 1$, the largest eigenvalue of $J$ will be $2\beta + o(1)$ asymptotically almost surely, and the diagonal entry will be $\beta/2$. Applying our result with $\eta \approx 0.25$ will improve the guaranteed rapid mixing threshold from around $\beta = 0.5$ to around $\beta \approx 0.54$.  

\paragraph{$O(N)$ model with bounded row norms.} 
The following discussion is related to a conjecture of \textcite{dyson1978phase}; see Remark D.1 there. Suppose that the interaction matrix $J$ has small row norms off the diagonal; more specifically, let
\[ R = \norm{J}_{\infty \to \infty} = \max\set*{\sum_{j : j \ne i} \abs{J_{ij}}\given i\in [n]}. \]
We can assume the diagonal entries of $J$ all equal $R$ without loss of generality; then by Gershgorin's circle theorem, $J$ is positive semidefinite with operator norm at most $2R$. Applying \cref{thm:ON-main}, we therefore obtain $O(n \log n)$ time mixing for all $R$ up to $N s_{\eta}(1/2)/2 \approx 0.5874 N$.
In general, for every value of $N$ it was proved by \textcite{dyson1978phase} that the covariance matrix has operator norm $O(1)$, i.e., satisfies a bound independent of the number of sites $n$, as long as $R < 0.5 N$, and they conjectured that, at least in nice cases like ferromagnetic models on lattices, this bound can be improved to hold when $R < N$. \Textcite{bauerschmidt2019very} improved the covariance bound result to a log-Sobolev inequality up to the same threshold. Our result improves the threshold by an $N$-independent multiplicative constant without making any further assumptions on $J$. 

\paragraph{Antiferromagnetic Ising on random graphs.}
For random $d$-regular graphs on $n$ vertices, as long as $\beta< \frac{1-\delta}{8 \sqrt{d-1}},$ we can sample from the antiferromagnetic Ising model with parameter $\beta$ using the \emph{polarized walk}\Tag{; see \cref{cor:random d regular}}. Our runtime is nearly linear in the number of edges, thus optimal up to log factors. In \Tag{\cref{cor:glauber ising low degree}}\Tag<sigconf>{the full version of this paper}, we show that the Glauber dynamics mixes in $O\parens*{e^{O(\sqrt{d})} \cdot n \log n}$ steps. Our results also apply to other random graph models such as Erd\H{o}s-R\'enyi.

\paragraph{Fixed magnetization (ferromagnetic or antiferromagnetic) Ising on random graphs.}
For random $d$-regular graphs on $n$ vertices, as long as $\beta< \frac{1-\delta}{8 \sqrt{d-1}}$, we can sample from the fixed magnetization (ferromagnetic or antiferromagnetic) Ising model with parameter $\beta$ at an arbitrary magnetization level $k$ using $O(\delta^{-1} k \log n)$ steps of the down-up walk\Tag{ (\cref{cor:d regular expander fixed magnetization})}. \Textcite{Bauerschmidt2023KawasakiDB} concurrently and independently gave a bound on the mixing time of the down-up walk in the same parameter regime, but their bound has an exponential dependency on the degree $d$, i.e., $\exp(O(\sqrt{d}))$; see Lemma 2.6 in \cite{Bauerschmidt2023KawasakiDB}. The main goal in \cite{Bauerschmidt2023KawasakiDB} is to bound the mixing time of the Kawasaki dynamics, a localized version of the down-up walk when one can only exchange the spin of neighboring vertices. They do so by comparing Kawasaki dynamics with the down-up walk, incurring an extra $\exp(O(\sqrt{d} + h)) \log^4 n$  factor \cite[Lemma 4.1]{Bauerschmidt2023KawasakiDB} and thus showing that the Kawasaki dynamics mixes in $ O(\exp(O(\sqrt{d}) + h)\cdot n\log^6 n)$ steps. The extra $\log^5 n \cdot \exp(O(\sqrt{d}+ h))$ is unlikely to be necessary; our mixing time bound for the down-up walk is the first step to remove this dependency.

Using the equivalence between sampling and counting \cite{jerrum1986random}, we have an FPRAS for the partition function of the fixed magnetization antiferromagnetic and ferromagnetic Ising at any magnetization level and up to the same threshold for $\beta$. The sum of these partition functions over all magnetization levels is precisely the partition function for the Ising model, thus we also obtain FPRAS for the partition function of Ising models, and thus algorithms to sample from these models using counting to sampling reduction. This will work even in the case that the external field has inconsistent signs, which is \Class{\#BIS}-hard without further assumptions \cite{goldberg2007complexity}. In the ferromagnetic case and antiferromagnetic with fixed degree $d$ cases, similar results were obtained in prior work of \textcite{koehler2022sampling} with a different algorithm; in the antiferromagnetic case, the polarized walk gives another possible algorithm with nearly linear runtime. 

Our results also apply to other random graph models such as Erd\H{o}s-R\'enyi\Tag{ (\cref{cor:fixed magnetization erdos renyi})}, and more generally to expanders\Tag{ (\cref{prop:fixed magnetization expander general})}. 

\Tag{\subsection{Other related work}}\Tag<sigconf>{\subsection{Other Related Work}}
There have been many works studying Glauber dynamics in continuous state spaces. For example, in convex bodies Glauber dynamics is the well-known ``coordinate hit-and-run'' Markov chain, see, e.g., \cite{narayanan2022mixing,laddha2023convergence}. Glauber dynamics has also been analyzed in spin systems \cite{wu2006poincare,marton2013inequality} under conditions similar to Dobrushin's condition. 

The recent works \cite{alaoui2022sampling,celentano2022sudakov} (see also \cite{montanari2023posterior}) developed and analyzed a sampler for the SK model based on AMP (Approximate Message Passing) and stochastic localization. An advantage of this approach is that it provably works up to $\beta < 1$, which is the entire replica-symmetric regime of the model with zero external field. On the other hand, this method ends up sampling in a much weaker sense (achieving sublinear Wasserstein distance), does not apply to related models like diluted spin glasses, and does not imply structural properties of the measure like subgaussian concentration. It is worth noting that the work \cite{alaoui2022sampling} also shows an impossibility result for sampling via a class of Lipschitz algorithms in the replica symmetry-breaking regime ($\beta > 1$); it is possible that this is an impossible task for all polynomial time algorithms. Besides the SK model, there have also been recent works on mixing in spin glasses with higher order interactions ($p$-spin models) on the sphere and hypercube at sufficiently high temperature \cite{gheissari2019spectral,pspin,anari2023universality}.  

Also specifically in the context of the SK model, recent works \cite{alaoui2022bounds,brennecke2022two,brennecke2023operator} have obtained $O(1)$ bounds on the operator norm of the covariance matrix of the SK model when $\beta < 1$, improving polylogarithmic bounds due to \citet{talagrand2011mean}. To contrast with our analyses, we establish (and crucially need) bounds which \emph{hold uniformly over all external fields $h$} for a fixed/``quenched'' realization of the interaction matrix $J$. This is because an $O(1)$ bound on the operator norm of the covariance matrix is, by itself, obviously not a sufficient condition for rapid mixing, Lipschitz concentration, and so on to hold for a probability measure.  Because of the difficulties involved in handling all external fields $h$ at once, there is unfortunately no obvious way to apply the techniques based on the cavity method and TAP equations used in those works.

Recently, \citet{kunisky2023optimality} gave some evidence that the constant in the spectral condition used to guarantee polynomial time mixing of the Glauber dynamics in \cite{eldan2021spectral,anari2021entropic} is sharp not just for Glauber dynamics (where it is tight in the Curie-Weiss model) but in fact for all polynomial time algorithms. It would be interesting if anything along these lines can be said for the more sophisticated spectral conditions established in \cref{thm:ising-main} and \cref{thm:semilogconcave}.

\subsection{Organization}
\Tag{In \cref{sec:preliminaries}, we state general facts and preliminaries which are used in the following sections. In \cref{sec:trickledown}, we formulate the general trickle-down equation for linear-tilt localization schemes, which is used in the remainder of the body of the paper. In \cref{sec:ising}, we prove \cref{thm:ising-main} concerning rapid mixing of the Glauber dynamics in Ising models. In \cref{sec:semi} we prove the analogous results with semi-log-concave base measures.
In \cref{sec:ON}, we apply the semi-log-concave theory to obtain our results for both the Langevin and Glauber dynamics in the $O(N)$ model. \cref{sec:expander} proves rapid mixing of the polarized walk. Finally, in \cref{a:glauber-lowdegree} we prove rapid mixing of the Glauber dynamics on low-degree expanders.}
\Tag<sigconf>{
	In \cref{sec:preliminaries}, we state general facts and preliminaries about Markov kernels. In \cref{sec:trickledown}, we formulate the general trickle-down equation for linear-tilt localization schemes, which is used in the full version of this paper to derive all of the results in \cref{sec:intro}.
}

\Tag{
	\subsection*{Acknowledgment}
	This work was supported by NSF CAREER Award CCF-2045354. Thuy-Duong Vuong was supported by a Microsoft Research Ph.D.\ Fellowship. Frederic Koehler was supported in part by NSF award CCF-1704417, NSF award IIS-1908774, and N. Anari’s Sloan Research Fellowship, and thanks the participants of the Simons Institute program on Probability, Geometry, and Computation in High Dimensions for related discussions.

}
    \section{Preliminaries}
\label{sec:preliminaries}

\Tag{

In this section, we summarize, in a mostly self-contained fashion, the relevant facts about Markov chains, localization schemes, geometry of polynomials, etc., which are needed to prove our results. Some of the facts we state  --- in particular, \cref{lem:factorization-rhs-preserve} --- appear to be more general than what has previously appeared in the literature, and we will need this level of generality.
}

We use $[n]$ to denote the set of $n$ elements $\set{1,\dots,n}$. We also use $[\bar{n}]$ to denote $\set{\bar{1},\dots,\bar{n}}$, where we think of $\bar{i}$ as an element distinct from $i$. We use $2^{[n]}$ to denote the family of subsets of $[n]$ and $\binom{[n]}{k}$ to denote subsets of size $k$. We use $\1$ to denote the all-ones vector, where the dimension is inferred from context. For a set $S\subseteq [n]$, we use $\1_S\in \R^n$ to denote the indicator vector of the set $S$. We use $\1_1,\1_2,\dots,\1_n$ to denote the standard basis vectors in $\R^n$.

By default, we treat vectors as column vectors. For vectors $u, v \in \R^n$ we use $\dotprod{u, v}\in \R$ to denote the inner product and $u\otimes v =uv^\intercal \in \R^{n\times n}$ to denote the outer product. We use $u^{\otimes 2}$ to abbreviate $u\otimes u=uu^\intercal$.

For a probability measure $\mu$ on a space $\Omega$ and function $f:\Omega\to \R^d$ we use the shorthands $\E_\mu{f}$ and $\mu f$ to denote the expectation $\E_{x\sim \mu}{f(x)}$.

\begin{definition}[Mean and covariance]
	Suppose that $\mu$ is a probability measure on (a subset of) $\R^n$. We use $\mean{\mu}$ to denote $\E_{x\sim \mu}{x}$. Similarly, we define the covariance matrix as follows:
	\[ \cov{\mu}=\E_{x\sim \mu}{x^{\otimes 2}}-\E_{x\sim \mu}{x}^{\otimes 2}. \]
\end{definition}
With some abuse of notation, for a random variable $X$, we use $\mean{X}$ and $\cov{X}$ to denote the mean and covariance of the law of $X$. We also naturally allow conditioning in $\mean{}$ and $\cov{}$, which result in vector/matrix-valued random variables. So for a $\sigma$-algebra $\F$ and random variable $X$, we have $\cov{X\given \F}=\E{X^{\otimes 2}\given \F}-\E{X\given \F}^{\otimes 2}$.

We use $\Normal{m, \Sigma}$ to denote the Gaussian distribution with mean $m$ and covariance matrix $\Sigma$. We use $\Ber{p}$ to denote the Bernoulli distribution on $\set{0, 1}$ having mean $p$ and use the shorthand $\Ber{}$ to denote $\Ber{1/2}$. We use $\Berpm{m}$ to denote the distribution on $\set{\pm 1}$ whose mean is $m$ and use the shorthand $\Berpm{}$ to denote $\Berpm{0}$.

\Tag{\subsection{Markov kernels}}\Tag<sigconf>{\subsection{Markov Kernels}}

\begin{definition}[Markov kernel] \label{def:markov-kernel}
A Markov kernel $\kappa$ from $\Omega$ to $\Omega'$ assigns to every point $x\in \Omega$ a probability measure $\kappa(x,\cdot)$ on $\Omega'$.
\end{definition}
We assume $\Omega$ and $\Omega'$ are measurable spaces equipped with $\sigma$-algebras $\F, \F'$, which we usually omit for the sake of brevity. The kernel $\kappa$ would be a function $\Omega\times \F'\to [0, 1]$ such that for each $x$, $\kappa(x,\cdot)$ is a probability measure on $(\Omega',\F')$, and for each $S\in \F'$ the function $x\mapsto \kappa(x,S)$ is measurable w.r.t.\ $(\Omega,\F$).

In the literature, sometimes Markov kernels are called channels. For finite spaces $\Omega, \Omega'$, we view $\kappa$ as a row-stochastic matrix. Note that when $\Omega'=\Omega$, $\kappa$ could be viewed as defining the transitions of a Markov chain on $\Omega$.

A Markov kernel $\kappa$ from $\Omega$ to $\Omega'$ can be combined with a Markov kernel $\kappa'$ from $\Omega'$ to $\Omega''$ in order to give a Markov kernel $\kappa\kappa'$ from $\Omega$ to $\Omega''$:
\[\kappa\kappa'(x, S)=\int_{y\in \Omega'} \kappa'(y, S)\kappa(x, dy)=\E_{y\sim \kappa(x, \cdot)}{\kappa'(y, S)}. \]
For finite spaces, this corresponds to matrix multiplication. We can also take the direct product of two kernels $\kappa$ and $\gamma$ from $\Omega$ to $\Omega'$ and from $\Omega$ to $\Omega''$ respectively, and define $\kappa\times\gamma$ to be a kernel from $\Omega$ to $\Omega'\times\Omega''$ with $\kappa\times\gamma(x,\cdot)=\kappa(x,\cdot)\times\gamma(x,\cdot)$.

A measure $\mu$ on $\Omega$ can be combined with a Markov kernel $\kappa$ from $\Omega$ to $\Omega'$ to get another measure $\mu\kappa$ on $\Omega'$:
\[ \mu\kappa(S)=\int_{x\in \Omega} \kappa(x, S)\mu(dx). \]
For finite spaces, this corresponds to a vector-matrix multiplication, viewing $\mu$ as a row vector. Note that if $\mu$ is a probability measure, i.e., if $\mu(\Omega)=1$, then so is $\mu\kappa$.

While distributions multiply on the l.h.s.\ of a Markov kernel, functions multiply on the r.h.s. Given a Markov kernel $\kappa$ from $\Omega$ to $\Omega'$ and a measurable function $f:\Omega'\to \R^d$, define $\kappa f:\Omega\to \R^d$ as
\[ \kappa f(x)=\int_{y\in \Omega'} f(y)\kappa(x, dy)=\E_{y\sim \kappa(x,\cdot)}{f(y)}. \]

\begin{definition}[Semidirect product]
    Consider a distribution $\mu$ on $\Omega$ and a Markov kernel $\kappa$ from $\Omega$ to $\Omega'$. We can generate random variables $X, Y$, by first sampling $X\sim \mu$ and then, conditioned on $X$, sampling $Y\sim \kappa(X, \cdot)$. Note that the marginal distribution of $Y$ is $\mu\kappa$. The joint distribution of $(X, Y)$ is the semidirect product $\mu \rtimes \kappa$. We also use $\kappa \ltimes \mu$ to denote the distribution of $(Y, X)$.
\end{definition}

\begin{definition}[Time-reversal]
	We say that a Markov kernel $\kappa'$ is the time-reversal of the Markov kernel $\kappa$ w.r.t.\ the distribution $\mu$ if $\kappa'(Y, \cdot)$ is the (regular) conditional probability distribution of $X$ given $Y$, when $(X, Y)\sim \mu\rtimes \kappa$. A shorthand definition is that $\mu$ and $\kappa$ define the same joint distribution as $\mu\kappa$ and $\kappa'$ define:
	\[ \mu \rtimes \kappa = \kappa' \ltimes (\mu\kappa) \]
	Although time-reversal, being a conditional probability, is technically not necessarily unique, with slight abuse of notation, we use the following convenient shorthand to denote $\kappa'(y, \cdot)$:
	\[\mu\observe{\kappa}y\]
\end{definition}

In Bayesian terminology, $\mu$ would be a prior distribution, and $y$ an observation obtained by passing a sample from $\mu$ through the channel $\kappa$; $\mu \observe{\kappa} y$ is simply the posterior distribution, given this observation. For finite spaces, we can write this posterior as
\[ (\mu\observe{\kappa} y)(x)  = \frac{\mu(x) \kappa(x,y)}{\mu \kappa(y)}. \] 
If $\gamma$ is another kernel from $\Omega$ to $\Omega''$, then observations of $\kappa$ and $\gamma$ commute with each other:
\[ (\mu\observe{\kappa}y)\observe{\gamma}z=(\mu\observe{\gamma}z)\observe{\kappa}y. \]
This is the posterior of $X\sim \mu$ given that independent observations $Y\sim \kappa(X,\cdot)$ and $Z\sim \gamma(X, \cdot)$ were $Y=y, Z=z$.

\Tag{
While a kernel $\kappa$ allows us to map measures from $\Omega$ to measures on $\Omega'$, the densities of these measures are mapped according to the time-reversal kernel $\kappa'$.

\begin{proposition}
	Suppose that $\mu$ is a probability measure on $\Omega$ and $\kappa$ is Markov kernel from $\Omega$ to $\Omega'$ whose time-reversal is $\kappa'$. If $\nu$ is another measure on $\Omega$ absolutely continuous w.r.t.\ $\mu$, then
	\[ \frac{d(\nu\kappa)}{d(\mu\kappa)}=\kappa' \frac{d\nu}{d\mu}. \]
\end{proposition}
\begin{proof}
	For any test function $f$ on $\Omega'$, we have
	\begin{align*} \E*_{y\sim \mu \kappa}{\parens*{\kappa'\frac{d\nu}{d\mu}}(y)\cdot f(y)}&=\E*_{y\sim \mu \kappa}{\E*_{x\sim \kappa'(y, \cdot)}{\frac{d\nu}{d\mu}(x)}\cdot f(y)}=\E*_{x\sim \mu}{\frac{d\nu}{d\mu}(x)\cdot \E*_{y\sim \kappa(x,\cdot)}{f(y)}} \\ &=\int_{x\in \Omega} \E_{y\sim \kappa(x, \cdot)}{f(y)}\; \nu(dx)=\int_{y\in \Omega'}f(y)\;\nu\kappa(dy).	
	\end{align*}
	This proves that $\kappa'\frac{d\nu}{d\mu}$ is the density of $\nu\kappa$ w.r.t.\ $\mu\kappa$.
\end{proof}

\begin{definition}[Mixture decomposition]
	A mixture decomposition of a distribution $\mu$ on $\Omega$ is a collection of distributions $\mu_\theta$ indexed by $\theta\in \Omega'$, together with a distribution $\pi$ on $\Omega'$, such that for all events $S\subseteq \Omega$, $\mu_\theta(S)$ is measurable in $\theta$ and $\mu(S)=\E_{\theta\sim \pi}{\mu_\theta(S)}$. We abbreviate this as
	\[\mu=\E_{\theta\sim \pi}{\mu_\theta}.\]
\end{definition}
Mixture decompositions are in direct correspondence with Markov kernels from $\Omega$ to $\Omega'$. Given a mixture decomposition we define a Markov kernel $\kappa$: first let the kernel $\kappa'$ be $\kappa'(\theta,\cdot)=\mu_\theta$, and note that $\pi\kappa'=\mu$. Now let $\kappa$ be the time-reversal of $\kappa'$ w.r.t.\ $\pi$. Going the opposite direction, for any Markov kernel $\kappa$ from $\Omega$ to $\Omega'$, we can let $\pi=\mu\kappa$, and $\kappa'$ be the time-reversal of $\kappa$ w.r.t.\ $\mu$, and define $\mu_\theta=\kappa'(\theta,\cdot)$. We automatically get $\pi\kappa'=\mu$, i.e., $\E_{\theta\sim \pi}{\mu_\theta}=\mu$.
}

\Tag{
\subsection{Entropy contraction}

\begin{definition}[$\varphi$-entropy and $\varphi$-divergence]
	Given a convex function $\varphi:\R_{\geq 0}\to \R$ and measure $\mu$ on space $\Omega$, $\varphi$-entropy is defined as an operator that takes functions $f:\Omega\to \R_{\geq 0}$ and outputs
	\[ \Ent_\mu^\varphi{f}=\E_\mu{\varphi\circ f}-\varphi\parens*{\E_\mu{f}}. \]
	We also define the related notion of a $\varphi$-divergence.\footnote{$f$-divergence is more commonly used in the literature, but for consistency with $\varphi$-entropies, we use $\varphi$-divergence.} For a measure $\nu$ that has density $d\nu/d\mu$ w.r.t.\ $\mu$ we define the $\varphi$-divergence as
	\[ \D_{\varphi}{\nu\river \mu}=\Ent*_\mu^\varphi{\frac{d\nu}{d\mu}}. \]
\end{definition}
Note that $\varphi$-entropy and $\varphi$-divergence are always nonnegative, by Jensen's inequality. Two special cases of $\varphi$ give the familiar variance and entropy functionals:

\begin{definition}[Variance and $\chis{}$-divergence]
	Specializing to $\varphi(x)=x^2$, we define variance as
	\[ \Var_\mu{f}=\Ent_\mu^\varphi{f}=\E_\mu{f^2}-\E_\mu{f}^2, \]	
	and the $\chis{}$-divergence as
	\[ \chis{\nu\river \mu}=\D_\varphi{\nu\river \mu}=\Var*{\frac{d\nu}{d\mu}}. \]
\end{definition}

\begin{definition}[Relative entropy and KL-divergence]
	Specializing to $\varphi(x)=x\log x$, we define the relative entropy as
	\[\Ent_\mu{f} = \Ent_\mu^\varphi{f}=\E_{\mu}{f\log f}-\E_{\mu} {f} \log (\E_{\mu} {f} ),\]
	and the Kullback–Leibler (KL) divergence as
	\[\DKL{\nu \river \mu} = \Ent*_{\mu}{\frac{d\nu}{d\mu}}.\] 
\end{definition}

\begin{definition}[Entropy contraction] \label{def:entropy-contraction}
    We say a Markov kernel $\kappa$ from $\Omega$ to $\Omega'$ has $(1-\rho)$ contraction of $\varphi$-entropy w.r.t.\ a distribution $\mu$ on $\Omega$ if for all measures $\nu$ absolutely continuous w.r.t.\ $\mu$ we have
	\[ \D_{\varphi}{\nu \kappa\river \mu \kappa} \leq (1-\rho) \D_{\varphi}{\nu \river \mu}, \]
	or equivalently, assuming $\kappa'$ is the time-reversal of $\kappa$ w.r.t.\ $\mu$, if for all density functions $f$ on $\Omega$,
	\[ \Ent_{\mu\kappa}^\varphi{\kappa'f}\leq (1-\rho)\Ent_\mu^\varphi{f}. \]
\end{definition}

We will use the following useful characterization of shrinkage of $\varphi$-entropies under $\kappa$:
\begin{proposition} \label{prop:data-processing-slack}
	Suppose that $\nu$ is a measure absolutely continuous w.r.t.\ a probability measure $\mu$ on $\Omega$, and $\kappa$ is a Markov kernel from $\Omega$ to $\Omega'$. Then
	\[ \D_{\varphi}{\nu\river \mu}-\D_{\varphi}{\nu\kappa\river \mu\kappa}=\Ent*_\mu^\varphi{\frac{d\nu}{d\mu}}-\Ent*_{\mu\kappa}^\varphi{\frac{d(\nu\kappa)}{d(\mu\kappa)}}= \E*_{y\sim \mu\kappa}{\Ent*^\varphi_{\mu\observe{\kappa} y}{\frac{d\nu}{d\mu}}}. \]
\end{proposition}
\begin{proof}
	Let $f=\frac{d\nu}{d\mu}$. If $\kappa'$ is the time-reversal of $\kappa$ w.r.t.\ $\mu$, we have $\frac{d(\nu\kappa)}{d(\mu\kappa)}=\kappa'f$, so
	\[ \D_{\varphi}{\nu\kappa\river \mu\kappa}=\Ent_{\mu\kappa}^\varphi{\kappa'f}=\E_{\mu\kappa}{\varphi\circ \kappa'f}-\varphi(\nu\kappa(\Omega'))=\E_{\mu\kappa}{\varphi\circ \kappa'f}-\varphi(\nu(\Omega)), \]
	where we used the fact that $\nu\kappa(\Omega')=\nu(\Omega)$. Similarly, we can write
	\[ \D_{\varphi}{\nu\river \mu}=\Ent_\mu^\varphi{f}=\E_{\mu}{\varphi\circ f}-\varphi(\nu(\Omega))=\E_{y\sim \mu\kappa}{\E_{x\sim \kappa'(y,\cdot)}{\varphi\circ f}}-\varphi(\nu(\Omega)). \]
	Subtracting we get
	\[ \D_{\varphi}{\nu\river \mu}-\D_{\varphi}{\nu\kappa\river \mu\kappa}=\E*_{y\sim \mu\kappa}{\E_{x\sim \kappa'(y, \cdot)}{\varphi\circ f}-\varphi(\kappa' f(y))}=\E*_{y\sim \mu\kappa}{\Ent_{\kappa'(y,\cdot)}^\varphi{f}}, \]
	which finishes the proof.
\end{proof}
Note that \cref{prop:data-processing-slack} implies this difference is $\geq 0$ for convex $\varphi$, since $\Ent^\varphi{}$ is always $\geq 0$. In other words, we always have (weak) contraction of $\varphi$-entropy with $\rho\geq 0$. This is known as the data processing inequality. Using \cref{prop:data-processing-slack} we can write $(1-\rho)$ contraction of $\varphi$-entropy in a useful form in terms of density functions:
\begin{proposition}
	Suppose $\mu$ is a probability measure on $\Omega$ and $\kappa$ is a Markov kernel from $\Omega$ to $\Omega'$. Then we have $(1-\rho)$ contraction of $\varphi$-entropy iff for all densities $f$,
	\[ \E*_{y\sim \mu\kappa}{\Ent*^\varphi_{\mu\observe{\kappa} y}{f}}\geq \rho\cdot \Ent_\mu^\varphi{f}. \]	
\end{proposition}

\begin{remark}
    Entropy contraction for a Markov kernel $\kappa$ and \emph{all} probability measures $\mu$ is referred to as a strong data processing inequality \cite[see, e.g.,][]{polyanskiy2017strong}. In this work, we are concerned primarily with entropy contraction w.r.t.\ a fixed probability measure $\mu$. 
\end{remark}

We now prove an important concavity property of the $\varphi$-entropy decrement.
\begin{lemma}\label{lem:factorization-rhs-preserve}
	Let $\kappa$ be a Markov kernel from $\Omega$ to $\Omega'$, and let $f$ be a density function on $\Omega$. The following functional on distributions $\mu$ is concave:
	\[ \mu \mapsto \E*_{y\sim \mu\kappa}{\Ent*^\varphi_{\mu\observe{\kappa} y}{f}}. \]
	In other words, if we have a mixture decomposition for $\mu$, given by $\mu=\E_{\theta\sim \pi}{\mu_\theta}$, then
	\[ \E*_{y\sim \mu\kappa}{\Ent*^\varphi_{\mu\observe{\kappa} y}{f}}\geq \E*_{\theta\sim \pi}{\E*_{y\sim \mu_\theta\kappa}{\Ent*^\varphi_{\mu_\theta\observe{\kappa} y}{f}}}. \]
\end{lemma}
\begin{proof}
	A mixture decomposition can be described by a Markov kernel $\gamma$ such that $\pi=\mu\gamma$, and $\mu_\theta=\mu\observe{\gamma}\theta$. This means we can write the r.h.s.\ of the desired inequality as
	\[ \E*_{\theta\sim \mu\gamma}{\E*_{y\sim (\mu\observe{\gamma}\theta)\kappa}{\Ent*^\varphi_{(\mu\observe{\gamma}\theta)\observe{\kappa} y}{f}}}=\E*_{y\sim \mu\kappa}{\E*_{\theta\sim (\mu\observe{\kappa}y)\gamma}{\Ent_{(\mu\observe{\kappa}y)\observe{\gamma} \theta}^\varphi{f}}}. \]
	Here we switched the order of expectation by using the fact that both sides take expectation w.r.t.\ $(y,\theta)\sim \mu\cdot (\kappa\times \gamma)$. For each $y$, by applying \cref{prop:data-processing-slack} to the distribution $\mu'=\mu\observe{\kappa}y$ and Markov kernel $\gamma$, and using nonnegativity of $\Ent^\varphi{}$, we get
	\[ \Ent^\varphi_{\mu'}{f}\geq \E*_{\theta\sim \mu'\gamma}{\Ent^\varphi_{\mu'\observe{\gamma}\theta}{f}}. \]
	Taking expectation w.r.t.\ $y\sim \mu\kappa$, and expanding $\mu'=\mu\observe{\kappa}y$ finishes the proof:
	\[ \E*_{y\sim \mu\kappa}{\Ent^\varphi_{\mu\observe{\kappa}y}{f}}\geq \E*_{y\sim \mu\kappa}{\E*_{\theta\sim (\mu\observe{\kappa}y)\gamma}{\Ent^\varphi_{(\mu\observe{\kappa}y)\observe{\gamma}\theta}{f}}}.\qedhere \]
\end{proof}

\subsection{Functional inequalities and mixing}\label{sec:fi-and-mixing}

\paragraph{Glauber Dynamics and ATE.} Our results for the Glauber dynamics go via establishing a fundamental information-theoretic inequality called \emph{Approximate Tensorization of Entropy} (ATE). Let $\Omega=\Omega_1\times\cdots\Omega_n$ be a product space. For $i\in [n]$, define the Markov kernel $D_{\sim i}$ as one that maps $x=(x_1,\dots,x_n)\in \Omega$ deterministically to $x_{\sim i}$ by erasing its $i$th component, i.e., 
\[ x_{\sim i}\coloneq (x_1,\dots,x_{i-1},\varnothing,x_{i+1},\dots,x_n).\]

\begin{definition}[Approximate Tensorization of Entropy, \cite{marton2013inequality,caputo14}]\label{def:ate}
Let $\mu$ be a probability measure on $\R^n$. We say $\mu$ satisfies Approximate Tensorization of Entropy (ATE) with constant $C$ if for every other measure $\nu$ absolutely continuous w.r.t.\ $\mu$,
\[ \DKL{\nu \river \mu} \leq C \sum_{i = 1}^n \E*_{x\sim \nu}{\DKL*{(\nu\observe{D_{\sim i}} x_{\sim i}) \river (\mu\observe{D_{\sim i}}x_{\sim i})}}, \]
or equivalently, for all nonnegative functions $f$,
\[ \Ent_{\mu}{f} \leq C \sum_{i = 1}^n \E*_{x\sim \mu}{\Ent_{\mu\observe{D_{\sim i}}x_{\sim i}} f}. \]
\end{definition}
This inequality has a third equivalent reformulation. It asserts that
\[ \DKL{\nu D_{n \to n - 1} \river \mu D_{n \to n - 1}} \leq \parens*{1 - \frac{1}{Cn}} \DKL{\nu \river \mu} \]
where $D_{n \to n - 1}$ is the Markov kernel that picks a coordinate $i \in [n]$ uniformly at random and erases it (``down operator''), i.e., the average of $D_{\sim 1},\dots,D_{\sim n}$. This can be seen from the equivalence described in \cref{def:entropy-contraction}.

ATE is a very strong functional inequality, which yields an array of useful consequences ``for free'' once established. We do not go into full details about all of these consequences because we do not explicitly need them in this work, but give a brief overview below. They are fully explored in the references (with additional important references to prior work). First, we recall its implications for concentration of measure and mixing: 
\begin{proposition}[Fact 3.5 of \cite{chen2021optimal}]
If ATE for measure $\mu$ holds with constant $C$ then:
\begin{enumerate}
    \item The MLSI (Modified Log-Sobolev Inequality) for Glauber dynamics holds with constant $1/Cn$.
    \item As a consequence of the MLSI, on a discrete space, defining $\mu_{\min} = \min\set{\mu(x)\given x\in\supp(\mu)}$, the $\epsilon$-mixing time is $O(n(\log\log(1/\mu_{\min}) + \log(1/\epsilon))).$
    \item As a consequence of the MLSI, a $W_1$ entropy-transport inequality holds with constant $C$ and this is equivalent to subgaussian concentration of Lipschitz functions with constant $C$ and with respect to the Hamming metric.
    \item As a consequence of the MLSI, the Poincar\'e inequality for Glauber dynamics holds with constant $1/Cn$.
    \item If $\mu$ is defined on a discrete space and additionally $b$-marginally bounded, i.e., $\P_{X\sim \mu}{X_i=x_i\given X_{\sim i}=x_{\sim i}}\geq b$ for all $x$, then the log-Sobolev inequality for Glauber dynamics holds with constant $O_b(1/Cn)$. This is furthermore equivalent to hypercontractivity of the semigroup.
\end{enumerate}
\end{proposition}
Here are some other useful consequences:
\begin{enumerate}
    \item As a consequence of the MLSI, reverse hypercontractivity for the Glauber semigroup holds with an appropriate constant \cite{mossel2013reverse}.
    \item As a consequence of the Poincar\'e inequality, $\mu^{\hom}$ is $C$-spectrally independent \cite{anari2023universality}. Spectral independence is defined below in \cref{subsec:poly}.
    \item If ATE holds uniformly over all external fields (which is the case for all of the results in this paper), then spectral independence holds under all external fields by the previous point, which is equivalent to fractional log-concavity of the generating polynomial (see \cref{subsec:poly}). In turn, this implies entropic independence/subadditivity of entropy/Brascamp-Lieb type inequalities \cite{carlen2009subadditivity,anari2021entropic}.
    \item Block tensorization of entropy holds with an appropriate constant --- see \cite{lee2023parallelising} for details.
    \item ATE yields non-asymptotic bounds on the statistical performance of pseudolikelihood estimation \cite{besag1975statistical,koehler2022statistical}.
    \item ATE yields strong guarantees for identity testing in an appropriate oracle model \cite{BCSV21}.
    \item For bounded degree and marginally bounded graphical models, ATE yields KKL and Talagrand isoperimetric inequalities \cite{koehler2023influences}.
    \item On the hypercube $\set{\pm 1}^n$, ATE implies the log-Sobolev inequality for a different continuous-time dynamics: the inequality concretely takes on the form $\Ent{f} \lesssim C \sum_i \norm{\partial_i \sqrt{f}}_2^2$ with $\partial_i g(x) = g(x_1,\dots,x_{i-1},+1,x_{i+1},\dots,x_n) - g(x_1,\dots,x_{i-1},-1,x_{i+1},\dots,x_n)$. The corresponding semigroup can be understood as a variant of Glauber with a variable transition rate (that can become much larger than one in some situations). This type of inequality was studied in, e.g.,\cite{martinelli1999lectures,bauerschmidt2019very}; see \cite{eldan2021spectral} for more discussion.  
\end{enumerate}

\paragraph{Riemannian Langevin dynamics.} Just like in Euclidean space, the Langevin diffusion on a Riemannian manifold is a continuous random walk. Given a probability distribution with relative density $d\mu \propto e^{f(x)}dx$ on the manifold, the corresponding Langevin dynamics to sample from $\mu$ can be written as the solution of a manifold-valued stochastic differential equation:
\[ dX_t = \grad f(X_t)dt + \sqrt{2} dW_t \]
where $\grad$ denotes the Riemannian gradient, and $W_t$ is Brownian motion on the manifold. See \cite{li2020riemannian,cheng2022efficient} for more about this equation and discretization schemes, and \cite{hsu2002stochastic} for a textbook treatment of stochastic calculus on manifolds. 

For our results, all we need to know about this process is the form of the log-Sobolev inequality for the Langevin semigroup when the manifold is a product of $n$ many $N - 1$ dimensional spheres.  As a special case of the general definition, we say the log-Sobolev inequality for the manifold Langevin dynamics for distribution $\mu$ holds with constant $C$ if for all nonnegative functions $f$,
\begin{equation}\label{eqn:lsi-sphere}
	\Ent_{\mu}{f} \leq C \sum_{i = 1}^n \E*_{\mu}{\norm*{\grad_i \sqrt{f}}^2}, 
\end{equation}
where $\grad_i$ is the spherical gradient w.r.t.\ site $i$. This is the same inequality studied in the previous work \cite{bauerschmidt2019very} on the $O(N)$ model. Via the Herbst argument, it implies subgaussian concentration of Lipschitz functions.  See \cite{bakry2014analysis,hsu2002stochastic} for more background on the log-Sobolev inequality and diffusion processes. 

\subsection{Stochastic localization}\label{subsec:sl}
Given a measure $\nu_0$ on $\R^n$, a standard $n$-dimensional Brownian motion $W_t$ and an adapted process $C_t$ valued in positive semidefinite $n \times n$ matrices, define the \emph{Stochastic Localization} (SL) process to be the solution to the stochastic differential equation (SDE)
\begin{equation}\label{eqn:sl} dF_t(x) = F_t(x)\cdot \dotprod{ x - \mean{\nu_t}, C_t dW_t }, 
\end{equation}
where $d\nu_t=F_td\nu_0$ and $F_0 \equiv 1$. See, e.g., \cite{eldan2013thin,eldan2020clt,eldan2021spectral} for the proof of the existence and uniqueness of this SDE for a sufficiently regular process $C_t$ --- in this work, we will only need the case where the driving matrix $C_t$ is constant. The stochastic process $\nu_t$ is a probability-measure-valued martingale, which means that $\nu_t(A)$ is a martingale for any measurable event $A$. 

We will use the following Gaussian channel interpretation (see \cite{el2022information,klartag2021spectral}) of stochastic localization with a time-invariant driving matrix.  
\begin{theorem}[Theorem 2 of \cite{el2022information}]\label{thm:gci}
	Let $C$ be a positive definite $n \times n$ matrix and $\nu_0$ a distribution over $\R^n$. Let $X^* \sim \nu_0$, let $B_t$ be an independent standard Brownian motion, and define
	\[ y_t = t X^* + C^{-1} B_t. \]
	Define $d\nu_t = F_t d\nu_0$ where the Radon-Nikodym derivative $F_t$ is defined so that $\E_{\nu_0}{F_t} = 1$ and
	\[ F_t(x) \propto \exp\parens*{\dotprod{ C^2 y_t, x } - \frac{t}{2} \dotprod{ x, C^2 x } }. \]
	Then there exists a Brownian motion $W_t$ adapted to the filtration generated by the $y_t$ process such that $F_t$ is a solution of the SDE \eqref{eqn:sl} with $C_t = C$. 
\end{theorem}
In this setting, the SL process is closely related to the F\"ollmer drift (see, e.g., \cite{klartag2021spectral}) and the Polchinski (renormalization group) equation, see \cite{bauerschmidt2022log,chen2022localization}. We remark that in Ising models, if we look only at time $t = 1$ and choose driving matrix $C = J^{1/2}$, then $y_1$ is just the output of the Hubbard-Stratonovich transform. 

\paragraph{Entropic stability.} 
We also need to use some results concerning the key concept of \emph{entropic stability} which we now recall. As a reminder, entropic stability itself arises as a generalization of the concept of entropic independence \cite{anari2021entropic}, well-suited for simplicial localization schemes, to other linear-tilt localization schemes \cite{chen2022localization}. In particular, it tells us the exact analogue of entropic independence for stochastic localization. 
\begin{definition}[Exponential tilt]\label{def:exponential-tilts}
	Given a probability measure $\mu$ on $\R^n$, we let $\tilt_w \mu$ denote the \emph{exponential tilt} by $w$. This is the probability measure with Radon-Nikodym derivative 
	\[ \frac{d\mathcal T_w \mu}{d\mu}(x) \propto e^{\dotprod{w,x}} \]
	provided that $\E*_{x \sim \mu}{e^{\dotprod{w,x}}} < \infty$. 
\end{definition}
In the literature, an exponential tilt is sometimes called an \emph{external field}.
\begin{definition}[Entropic stability, Definition 29 of \cite{chen2022localization}]
	For $\mu$, a probability measure on $\R^n$, a function $\psi : \R^n \times \R^n \to \R_{\geq 0}$, and $\alpha > 0$ we say that $\mu$ is $\alpha$-\emph{entropically stable} with respect to $\psi$ if for every $w \in \R^n$, the exponential tilt $\tilt_w \mu$ exists and
	\[ \psi\parens*{\mean{\tilt_w \mu}, \mean{\mu}} \leq \alpha \DKL{\tilt_w \mu \river \mu}. \]
\end{definition}
Entropic stability with respect to a quadratic functional induced by the driving matrix implies the following conservation of entropy estimate for stochastic localization.
\begin{proposition}[Proposition 39 of \cite{chen2022localization}]\label{prop:conservation-of-entropy}
	Let $\nu_t$ be the stochastic localization process with driving matrix $C_t$, and suppose that for some $T > 0$ and almost surely for all $t \in [0,T]$ that $\nu_t$ is $\alpha_t$-entropically stable with respect to the function $\psi(x,y) = \frac{1}{2} \norm{C_t (x - y)}_2^2$. Then for all nonnegative functions $f$,
	\[ \E*{\Ent_{\nu_T}f}  \ge \exp\parens*{-\int_0^T \alpha_t dt}\cdot \Ent_{\nu_0}{f}. \]
\end{proposition}
The following connects entropic stability with the covariance matrix of tilted measures. 
\begin{lemma}[Lemma 40 of \cite{chen2022localization}]\label{lem:entropic-stability}
	Let $\mu$ be a probability measure on $\mathbb R^n$, and let $A$ and $C$ be positive definite matrices. If $\cov{\tilt_w \mu} \preceq A$ for every $w \in \mathbb R^n$, then $\mu$ is $\alpha$-entropically stable with respect to $\psi(x,y) = \frac{1}{2} \norm{C (x - y)}_2^2$ for $\alpha = \opnorm{CAC}$.
\end{lemma}
The property of having bounded covariance under all tilts is known as \emph{semi-log-concavity} \cite{eldan2022log}; it is a weakening of the assumption of strong log-concavity of probability measures on $\mathbb R^n$ due to the Brascamp-Lieb inequality \cite{brascamp2002some}. This property is equivalent to the assumption of the previous lemma after an appropriate change of basis, and weaker than fractional log-concavity of the generating polynomial defined in the next section.
\begin{definition}[Semi-log-concavity]
Let $\mu$ be a probability measure on $\mathbb R^n$. We say that $\mu$ is $\beta$-semi-log-concave if
$\cov{\tilt_w \mu} \preceq \beta I$ for every $w \in \mathbb R^n$.
\end{definition}
Related to semi-log-concavity, defining the following semigroup on probability measures will simplify the notation in several of our arguments. It is related to the action of the heat semigroup on the Fourier transform, but it is nonlinear due to the presence of the normalization factor. Note that it commutes with the exponential tilt operation $\mathcal T_w$, and that it acts as the identity operator for measures whose support lies on a sphere of constant radius.  
\begin{definition}
Given a probability measure $\mu$ on $\mathbb R^N$ and $s \in \mathbb R$, define probability measure $\mathcal G_s \mu$ by its relative density
\[ \frac{d\mathcal G_s \mu}{d\mu}(x) \propto \exp(-s \|x\|^2/2) \]
assuming that the integral defining the normalizing constant exists (this is automatically true when $s \ge 0$).
\end{definition}
\paragraph{Supermartingale property of relative entropy.}
Finally, we give an application of \cref{lem:factorization-rhs-preserve} in the context of stochastic localization. We will use this along with entropic stability to prove entropy contraction (i.e., approximate tensorization of entropy) estimates for down-up walks as well as the polarized walk. The very general form of the estimate is crucial for the latter application. For the Glauber dynamics, this was previously observed for the variance functional in \cite{eldan2021spectral}, and a special case of this result with entropy was previously used by \textcite{lee2023parallelising} to prove approximate tensorization of entropy estimates. 
\begin{lemma}\label{lem:supermartingale}
	Let $\nu_t$ be the stochastic localization process with driving matrix $C_t$. For all Markov kernels $\kappa$ and nonnegative functions $f$, the stochastic process (indexed by time $t \ge 0$)
	\[ \E*_{y\sim \nu_t \kappa}{\Ent_{\nu_t \observe{\kappa} y}{f}}   \]
	is a supermartingale. Equivalently, for all $0 \le s < t$ we have
	\[ \E*{\E*_{y \sim \nu_t \kappa}{ \Ent_{\nu_t\observe{\kappa}y}{f}  } \given \F_s} \leq \E*_{y \sim \nu_s \kappa} {\Ent_{\nu_s\observe{\kappa}y}{f}}, \]
	 where $\F_s$ is the filtration generated by the underlying Brownian motion in stochastic localization.
\end{lemma}
\begin{proof}
	By the martingale property of stochastic localization, we have the mixture decomposition
	\[ \nu_s = \E{\nu_t \mid \mathcal{F}_s}, \]
	so the result follows immediately from \cref{lem:factorization-rhs-preserve}.
\end{proof}

\subsection{Geometry of polynomials}\label{subsec:poly}

\begin{definition}[Strongly log-concave polynomials \cite{gurvits2009multivariate}]
	We say that a multivariate polynomial $f\in \R[z_1,\dots,z_n]$ with nonnegative coefficients is strongly log-concave if for each of its derivatives $g=\partial_{i_1}\cdots\partial_{i_k}f$, either $g=0$, or $\log g$ is concave over $\R_{>0}^n$.
\end{definition}
When $f$ is homogeneous, strong log-concavity is equivalent to complete log-concavity \cite[Theorem 3.2]{anari2018log}, also known as Lorentzianity \cite[Theorem 2.30]{branden2020lorentzian}. See, e.g., \cite{anari2018log,anari2019log,gurvits2009multivariate,branden2020lorentzian} for more background.

We now recall a few key facts we will need about strongly log-concave polynomials.  We specialize some facts to the homogeneous case since this is what we will use in this paper. 
\begin{proposition}[Proposition 2.7 of \cite{gurvits2009multivariate}, Proposition 5.1 of \cite{anari2018log}]\label{prop:slc-coeff}
	A homogeneous bivariate polynomial $f(y,z) = \sum_{k = 0}^n c_k y^{n - k} z^k$ is strongly log-concave iff its coefficients form an ultra log-concave sequence, that is $\set{k\given c_k>0}$ is a contiguous interval of integers, and for all $1 \le k \leq n - 1$,
	\[ \parens*{\frac{c_k}{\binom{n}{k}}}^2 \geq \frac{c_{k - 1}}{\binom{n}{k-1}}\cdot \frac{c_{k + 1}}{\binom{n}{k+1}}. \] 
\end{proposition}
Recall that the \emph{elementary symmetric polynomial} $e_k$ in variables $z_1,\ldots,z_n$ is 
\[ e_k(z) = \sum_{S\in \binom{[n]}{k}} \prod_{i\in S} z_i. \]
The elementary symmetric polynomials play a key role in \emph{polarization}, which we define next. Polarization was first studied in the context of stable polynomials, i.e., polynomials that are root-free in a half-plane of the complex plane \cite{Borcea2008TheLA}.
\begin{definition}[Polarization]
	Suppose that $g=\sum_{\alpha\in \Z_{\geq 0}^n} c_{\alpha} \prod_{i = 1}^n z_i^{\alpha_i}$ is a polynomial that has degree at most $\kappa_i$ in variable $z_i$. Define the polarization of $g$ to be the multi-affine polynomial in variables $(z_{ij})$ for $i \in [n]$ and $j \in [\kappa_i]$ defined by
	\[ \polar_{\kappa}(g) = \sum_{\alpha} c_{\alpha} \prod_{i = 1}^n \frac{e_{\alpha_i}(z_{i1},\dots, z_{i\kappa_i})}{\binom{\kappa_i}{a_i}}. \]
\end{definition}
When $\kappa$ is clear from context, we simply drop it and write $\polar$ for polarization.

\begin{proposition}[Proposition 18 of \cite{anari2021log}, \cite{branden2020lorentzian}]\label{prop:polarization}
	If a homogeneous polynomial $g$ is strongly log-concave, then so is its polarization $\polar_{\kappa}(g)$.
\end{proposition}

Strong log-concavity and similar properties of polynomials extend to distributions through the concept of generating polynomials.
\begin{definition}[Generating polynomial]
	If $\mu$ is a probability distribution on $2^{[n]}$, we define its generating polynomial to be
	\[ g_{\mu}(z_1,\ldots,z_n) =\E*_{S\sim \mu}{\prod_{i\in S}z_i} = \sum_{S \subseteq [n]} \mu(S) \prod_{i \in S} z_i. \]
\end{definition}
Next, we define a generalization of strong log-concavity.
\begin{definition}[Fractional log-concavity]
	For $\alpha \in (0,1]$, we say a multi-affine polynomial $g\in \R[z_1,\dots,z_n]$ with nonnegative coefficients is $\alpha$-fractionally-log-concave ($\alpha$-FLC) if $\log g(z_1^{\alpha},\dots,z_n^{\alpha})$ is concave on $\R_{>0}^n$.
\end{definition}
We note that if a multi-affine $g$ is $\alpha$-FLC, so are its derivatives, since they can be derived as limits:
\[ \partial_i g = \lim_{z_i\to \infty} g/z_i. \]
In particular, for multi-affine polynomials, $1$-FLC is the same as strong log-concavity. We say a probability distribution on $2^{[n]}$ is $\alpha$-FLC if its generating polynomial is $\alpha$-FLC.
\begin{definition}[Homogenization]
	For a multi-affine polynomial $g\in \R[z_1,\dots,z_n]$, we define its multi-affine homogenization to be $g^{\hom}\in \R[z_1,\dots,z_n,\bar{z}_1,\dots,\bar{z}_n]$ defined as
	\[ g^{\hom}(z_1,\dots,z_n,\bar{z}_1,\dots,\bar{z}_n)=\bar{z}_1\cdots\bar{z}_n g\parens*{\frac{z_1}{\bar{z}_1},\dots,\frac{z_n}{\bar{z}_n}}. \]
	Similarly, for a distribution $\mu$ on $2^{[n]}$, we define $\mu^{\hom}$ to be the distribution whose generating polynomial is $g_{\mu}^{\hom}$. In other words, $\mu^{\hom}$ will be a distribution supported on $\binom{[n]\sqcup[\bar{n}]}{n}\subset 2^{[n]\sqcup[\bar{n}]}$, where for each $S\subseteq [n]$ we let
	\[\mu^{\hom}\parens*{\set{i\given i\in S}\sqcup \set{\bar{i}\given i\notin S}}=\mu(S).\]
\end{definition}

We will also often deal with distributions $\mu$ on $\set{\pm 1}^n$. Identifying $\set{\pm 1}^n$ with $2^{[n]}$, we define the generating polynomial of such a $\mu$ as
\[g_{\mu} (z_1, \dots, z_n) = \E*_{x\sim \mu}{\prod_{i:x_i=+1} z_i}=\sum_{x\in \set{\pm 1}^n} \mu(x)\prod_{i:x_i=+1}z_i, \]
and its homogenization as
\[g_{\mu}^{\hom} (z_1, \dots, z_n, \bar{z}_1, \dots, \bar{z}_n) = \E*_{x\sim \mu}{\prod_{i:x_i=+1} z_i\cdot \prod_{i:x_i=-1}\bar{z}_i}=\sum_x \mu(x) \prod_{i: x_i=+1} z_i\prod_{i: x_i =-1} \bar{z}_i.\]

We also associate another homogeneous multi-affine polynomial with $\mu$ by first homogenizing with a single additional variable, followed by polarization:
\begin{align*} \polar g_\mu^{\hom}(z_1,\dots,z_n,y,\dots,y)&=\polar \E*_{x\sim \mu}{y^{\card{\set{i\given x_i =-1}}}\cdot \prod_{i:x_i=1} z_i } \\ &= \E*_{x\sim \mu}{\frac{e_{\card{\set{i\given x_i =-1}}}(y_1,\dots,y_n)}{\binom{n}{\card{\set{i\given x_i=-1}}}} \cdot \prod_{i:x_i=1}z_i}. \end{align*}
We define the corresponding distribution as the \emph{polar homogenization} of $\nu$.
\begin{definition}[Polar homogenization]\label{def:polarize}
	For a given $n$, let $\Pi$ denote the Markov kernel from $2^{[n]}$ to $\binom{[n]\sqcup[\bar{n}]}{n}$, which maps $S\subseteq [n]$ to $S\sqcup T$, where $T\in \binom{[\bar{n}]}{n-\card{S}}$ is chosen uniformly at random. For a distribution $\mu$ on $2^{[n]}$ we call $\mu\Pi$ the polar homogenization of $\mu$. This is a distribution supported on sets $S\sqcup T$ where $S\subseteq [n]$ and $T\subseteq [\bar{n}]$ and $\card{S}+\card{T}=n$:
\[\mu\Pi(S \sqcup T) = \frac{\mu(S)}{\binom{n}{n-\card{S}}}=\frac{\mu(S)}{\binom{n}{\card{S}}}.\]
\end{definition}

We extend the definition of exponential tilts, \cref{def:exponential-tilts}, to distributions $\mu$ on $2^{[n]}$ by identifying $2^{[n]}$ with $\set{0,1}^n\subset \R^n$. In other words, 
\[ \tilt_w\mu(S) = \frac{\mu(S)\prod_{i \in S} e^{w_i}}{g_{\mu}(e^{w_1},\dots,e^{w_n})} \]
where $g_{\mu}$, the generating polynomial of $\mu$, gives us the normalizing constant. Note that $\alpha$-fractional-log-concavity is closed under exponential tilts.

\begin{definition}[Correlation matrix] \label{def:corr}
	Let $\mu$ be a probability distribution over $2^{[n]}$. Its correlation matrix $\corMat_{\mu} \in \R^{n\times n}$  is defined by
\[\corMat_{\mu} (i,j) = \P_{S\sim \mu}{j\in S\given i\in S}-\P_{S\sim \mu}{i\in S}. \]
\end{definition}

\begin{definition}[Spectral independence]
For $\eta \geq 0$, a distribution 
$\mu$ on $2^{[n]}$ is said to be $\eta$-spectrally independent if
\[\lambda_{\max}(\corMat_{\mu}) \leq \eta.\]
\end{definition}

\begin{remark}
The original definition of spectral independence in \cite{ALO20} imposes this requirement on $\mu$ as well as all of its ``links.'' Here, we follow the convention in \cite{anari2021entropic} and use the term spectral independence to refer only to a spectral norm bound on the correlation matrix of $\mu$.
\end{remark}

\begin{proposition}[Remark 70 of \cite{alimohammadi2021fractionally}]\label{fact:si-flc}
	A distribution $\mu$ on $2^{[n]}$ is $\eta$-spectrally independent iff 
	\[ \eval{\nabla^2 \log g_{\mu}\parens*{z_1^{1/\eta},\ldots,z_n^{1/\eta}}}_{z = \1} = (1/\eta)^2 D \corMat_{\mu} - (1/\eta) D \preceq 0, \]
	where $D$ is the diagonal matrix with entries $D_{ii} = \P_{S\sim \mu}{i\in S}$. Moreover, $\mu$  is $1/\eta$-FLC iff $\tilt_w \mu$ is $\eta$-spectrally independent for all $w \in \R^n$.
\end{proposition}
With some abuse of notation, for distributions $\mu$ on $2^{[n]}$ we use $\cov{\mu}$ and $\mean{\mu}$ to denote $\cov{\1_S}$ and $\mean{\1_S}$ where $S\sim \mu$. We note that $D \corMat_{\mu}$ is the same as the $\cov{\mu}$.

Spectral independence is frequently used to obtain $\Var{}$ contraction rates for \emph{down kernels}.}
\begin{definition}\label{def:down-kernel}
	For a fixed ground set $[n]$ and $0\leq \l\leq k\leq n$, we let $D_{k\to \l}$ denote the Markov kernel from $\binom{[n]}{k}$ to $\binom{[n]}{\l}$ defined as follows: given a set $S$, we map it to a uniformly random subset $T\in \binom{S}{\l}$.
\end{definition}
\Tag{
For a distribution $\mu$ on $\binom{[n]}{k}$, we have $\eta/k$ contraction of $\Var{}$ under the kernel $D_{k\to 1}$ iff $\mu$ is $\eta$-spectrally independent \cite[see, e.g.,]{alimohammadi2021fractionally}.}

    \section{Trickle-down equation in linear-tilt localization schemes}\label{sec:trickledown}
\subsection{Trickle-down equation}
\paragraph{Linear-tilt localization scheme.}
Let $\nu_0$ be a probability measure. Suppose there exists a sequence of random vectors $Z_{t}$ generating filtrations $\F_t = \sigma(Z_0,\ldots, Z_{t})$ so that $\E{Z_{t + 1} \given \F_t} = 0$, i.e., they form a martingale difference sequence. Define the corresponding measure-valued stochastic process $(\nu_t)_t$ by its Radon-Nikodym derivative
\[ \frac{d\nu_{t + 1}}{d\nu_t}(x) = 1 + \dotprod{x - \mean{\nu_t}, Z_{t + 1} }. \]
Note that 
\[ \E*_{\nu_t}{\frac{d\nu_{t + 1}}{d\nu_t}}  = 1 + \dotprod{ \mean{\nu_t} - \mean{\nu_t}, Z_{t + 1} } =  1, \]
so, provided we check nonnegativity, $\nu_{t + 1}$ is indeed a probability measure (almost surely). To ensure nonnegativity and guarantee $\nu_{t + 1}$ is a probability measure, we make the assumption
\begin{equation}
\P_{x \sim \nu_t}{\dotprod{x-\mean{\nu_t}, Z_{t + 1} }  \ge -1} = 1.
\end{equation}
Also, observe that for any measurable set $S$
\[ \E*{\nu_{t + 1}(S) \given \F_t} = \nu_t(S), \]
because $\E{Z_{t + 1} \given \F_t} = 0$. So $(\nu_t)_t$ is a measure-valued martingale. 

We say this process is a \emph{linear-tilt localization scheme} if additionally, for any measurable event $S$, $\nu_t(S)$ almost surely converges to either $0$ or $1$ as $t \to \infty$. This is consistent with the terminology of Section 2.4 of \cite{chen2022localization}. However, the calculations we perform next do not require this additional assumption.

\paragraph{Trickle-down equation.} Now we formulate a general trickle-down recursion.
Recall that we have the following \emph{law of total variance} for $\nu_t$:
\begin{align*}
\cov{\nu_t}
&= \E_{x\sim \nu_t}{x^{\otimes 2}} - \E_{x\sim \nu_t}{x}^{\otimes 2} =\E_{x\sim \nu_t}{x^{\otimes 2}}-\mean{\nu_t}^{\otimes 2} \\
&= \E*{\E_{x\sim \nu_{t + 1}}{x^{\otimes 2}} - \mean{\nu_{t+1}}^{\otimes 2} \given \F_t} +\Tag<sigconf>{\\ &\quad}\E*{\mean{\nu_{t + 1}}^{\otimes 2} - \mean{\nu_t}^{\otimes 2} \given \F_t} \\
&= \E{\cov{\nu_{t + 1}} \mid \F_t} + \cov{\mean{\nu_{t+1}}\given \F_t}.
\end{align*}
Observe that by definition
\begin{align*} 
\mean{\nu_{t+1}}-\mean{\nu_t}
&=\E_{x \sim \nu_t}{x\cdot \dotprod{x - \mean{\nu_t}, Z_{t + 1}}} \\
&= \E_{x \sim \nu_t}{x\otimes (x-\mean{\nu_t})}Z_{t+1} \Tag<sigconf>{\\ &} = \cov{\nu_t} Z_{t + 1}
\end{align*}
so
\begin{equation}\label{eqn:trickledown-1}
\begin{aligned}
\cov{\nu_t} 
&= \E{\cov{\nu_{t + 1}} \given \F_t} +\Tag<sigconf>{\\ &\quad} \E*{(\mean{\nu_{t+1}}-\mean{\nu_t})^{\otimes 2} \given \F_t} \\
&= \E{\cov{\nu_{t + 1}} \given \F_t} + \cov{\nu_t} \E*{Z_{t + 1}^{\otimes 2} \given \F_t} \cov{\nu_t} \\
&= \E{\cov{\nu_{t + 1}} \given \F_t} + \cov{\nu_t} \cov{Z_{t + 1} \given \F_t} \cov{\nu_t}
\end{aligned}
\end{equation}
or rearranging, we have
\[ \cov{\nu_t} - \cov{\nu_t} \cov{Z_{t + 1} \given \F_t} \cov{\nu_t} = \E{\cov{\nu_{t + 1}} \given \F_t}. 
\]
We can recursively apply \cref{eqn:trickledown-1} up to any time $T \ge t$ to yield the more general integral formula
\Tag{\begin{equation}\label{eqn:trickledown-2}
	\cov{\nu_t} = \E*{\cov{\nu_{T}} \given \F_t} +  \sum_{s = t}^{T-1} \E*{\cov{\nu_s} \cov{Z_{s + 1} \given \F_s} \cov{\nu_s} \given \F_t}
\end{equation}}%
\Tag<sigconf>{\begin{multline}\label{eqn:trickledown-2}
	\cov{\nu_t} = \E*{\cov{\nu_{T}} \given \F_t} +\\  \sum_{s = t}^{T-1} \E*{\cov{\nu_s} \cov{Z_{s + 1} \given \F_s} \cov{\nu_s} \given \F_t}
\end{multline}}%
which will be very useful in the present work.
\subsection{Examples}
\paragraph{Example: stochastic localization.} 
Recall first that stochastic localization is the SDE with driving matrix $C_t$ defined as
\[ dF_t(x) = F_t(x) \dotprod{ x - \mean{\nu_t}, C_t dW_t } \]
where $W_t$ is a brownian motion and $d\nu_t = F_t d\nu_0$. 

For $\Delta > 0$, which we will take to zero, we can parameterize time in multiples of $\Delta$. The Euler-Murayama discretization of the SDE for stochastic localization with driving matrix $C_t$ would be
\[ \frac{d\nu_{t + \Delta}}{d\nu_t}(x) = 1 + \dotprod{x - \mean{\nu_t},  Z_{t + \Delta} } \]
where $Z_{t + \Delta} \sim \Normal{0, \Delta C_t^2}$. Taking the limit of $\Delta \to 0$, \cref{eqn:trickledown-2} becomes
\begin{equation}\label{eqn:sl-trickledown}
	\cov{\nu_t} = \E{\cov{\nu_T} \given \F_t} + \int_t^T \E{\cov{\nu_s} C_s^2 \cov{\nu_s} \given \F_t} ds 
\end{equation}
which is the trickle-down equation for stochastic localization. To be precise, the discretized process with fixed $\Delta > 0$ is not a localization scheme since it can assign events negative probabilities; however, it is a valid localization scheme in the limit $\Delta \to 0$. For completeness, we include a more formal proof of \cref{eqn:trickledown-2} below.
\begin{theorem}
Let $\nu_0$ be a probability measure on $\R^n$ and let $(\nu_t)_{t \ge 0}$ be the stochastic localization process with driving matrix $C_t$. Then \cref{eqn:sl-trickledown} holds for any $0 \le t \le T$.
\end{theorem}
\begin{proof}
This is the same proof as above, except executed rigorously in the limit $\Delta \to 0$ using stochastic calculus. It is also a standard calculation in the literature on stochastic localization, see e.g.\ \cite{eldan2020taming}.
We have
\Tag{\[ d\mean{\nu_t} = \E_{x\sim \nu_t}{x\cdot \dotprod{x-\mean{\nu_t}, C_tdW_t}} = \E_{x\sim \nu_t}{x\otimes (x-\mean{\nu_t})}C_tdW_t=\cov{\nu_t}C_tdW_t. \]}%
\Tag<sigconf>{\begin{multline*} d\mean{\nu_t} = \E_{x\sim \nu_t}{x\cdot \dotprod{x-\mean{\nu_t}, C_tdW_t}} =\\ \E_{x\sim \nu_t}{x\otimes (x-\mean{\nu_t})}C_tdW_t=\cov{\nu_t}C_tdW_t. \end{multline*}}%
Applying the It\^o formula, we get
\begin{align*}
	d(\mean{\nu_t}^{\otimes 2}) &=\cov{\nu_t}C_t^2\cov{\nu_t}dt+\Tag<sigconf>{\\ &\quad}\mean{\nu_t}\otimes (d\mean{\nu_t})+\Tag<sigconf>{\\ &\quad}(d\mean{\nu_t})\otimes \mean{\nu_t}\\
	&= \cov{\nu_t}C_t^2\cov{\nu_t}dt+\text{martingale},
\end{align*}
where martingale denotes a pure diffusion term, i.e., one with no drift. Since $\cov{\nu_t}=\E_{x\sim \nu_t}{x^{\otimes 2}}-\mean{\nu_t}^{\otimes 2}$, and $\E_{x\sim \nu_t}{x^{\otimes 2}}$ is a martingale, we have
\Tag{\[ d\cov{\nu_t}=d\E_{x\sim \nu_t}{x^{\otimes 2}}-d(\mean{\nu_t}^{\otimes 2})=\text{martingale}-\cov{\nu_t}C_t^2\cov{\nu_t}dt. \]}%
\Tag<sigconf>{\begin{multline*} d\cov{\nu_t}=d\E_{x\sim \nu_t}{x^{\otimes 2}}-d(\mean{\nu_t}^{\otimes 2})=\\ \text{martingale}-\cov{\nu_t}C_t^2\cov{\nu_t}dt. \end{multline*}}%
Integrating and taking the conditional expectation, which cancels the martingale terms, it follows that for any $T\geq t$,
\[  \cov{\nu_t} = \E{\cov{\nu_T}  \given \F_t} + \int_{t}^{T} \E{\cov{\nu_s} C_s^2 \cov{\nu_s} \given \F_t} ds\]
as claimed. 
\end{proof}

\paragraph{Example: pinning in set systems (pure simplicial complex).} Suppose that $\nu_0$ is a probability measure on $\binom{[n]}{k}$ which we embed in $\set{0,1}^n\subset \R^n$ by mapping $S\in \binom{[n]}{k}$ to $\1_s$. Following the high-dimensional expanders paradigm, we can naturally define a sequence of measures in the following way: let $S \sim \nu_0$, let $(a_1,\ldots,a_k)$ be the elements of $S$ ordered according to a uniformly random permutation, and for $t \in \set{0,\dots, k}$ define 
\[ \nu_t = \nu_0 \observe{D_{k\to t}}\set{a_1,\dots, a_t}, \]
where $D_{k\to t}$ is the down kernel from \cref{def:down-kernel}. In other words, $\nu_t$ is the posterior distribution of $T\sim \nu_0$ given the observation $\set{a_1,\dots,a_t}\subseteq T$. We can optionally define $\nu_t = \nu_k$ for every $t > k$ to be consistent with the definition of localization schemes. Also, it is worth noting that the ``link'' in high-dimensional expanders terminology will not be $\nu_t$, but the law of $S \setminus \set{a_1,\dots,a_t}$ for $S \sim \nu_t$.

This is a linear-tilt localization scheme\footnote{This observation generalizes one from Section 2.4.1 of \cite{chen2022localization} made in the context of spin systems.} w.r.t.\ the filtration $\F_t = \sigma(\set{a_1,\dots,a_t})$. This is because
\begin{align*} 
\nu_{t + 1}(T) 
&= \frac{\nu_t(T)\cdot \1[a_{t + 1} \in T]}{\P_{R \sim \nu_t}{a_{t + 1} \in R}}\Tag<sigconf>{\\ &} = \nu_t(T) \cdot \parens*{1 + \frac{\1[a_{t + 1} \in T] - \P_{R\sim \nu_t}{a_{t + 1} \in R}}{\P_{R\sim \nu_t}{a_{t + 1} \in R}}} \\
&= \nu_t(T) \cdot \parens*{1 + \dotprod*{ \1_T - \mean{\nu_t}, \frac{\1_{a_{t + 1}}}{\P_{R\sim \nu_t}{a_{t + 1} \in R}}} } \\
&= \nu_t(T) \cdot \parens*{1 + \dotprod*{ \1_T - \mean{\nu_t}, Z_{t + 1} }}
\end{align*}
where we define $Y_{t+1}=\1_{a_{t+1}}/\P_{R\sim \nu_t}{a_{t+1}\in R}$ and $Z_{t + 1} = Y_{t+1}-\E{Y_{t+1}\given \F_t}$. Here we used $\dotprod*{ \1_T - \mean{\nu_t},  \E*{Y_{t+1} \given \F_t} } = 0$, which follows from \cref{eqn:helpful} below and observing that for any $R$ in the support of $\nu_t$ we have $\dotprod{\1_R,\1_{[n]\setminus\set{a_1,\dots,a_t}}}=k-t$, which also means $\dotprod{\mean{\nu_t},\1_{[n]\setminus\set{a_1,\dots,a_t}}}=k-t$.

From the definition of the random vectors $Z_t$, it is clear that they generate the same filtration $\mathcal{F}_t$ as the variables $a_t$ and that they form a martingale difference sequence. We observe that the conditional law of $a_{t + 1}$ given $\F_t$ is the same as picking an element uniformly at random from $T \setminus \set{a_1,\dots,a_t}$ where $T \sim \nu_t$. Using this we have
\Tag{\begin{equation}\label{eqn:helpful} \E*{Y_{t+1} \given \F_t} = \sum_{e\in [n]\setminus{a_1,\dots,a_t}}\frac{\P_{R\sim \nu_t}{e\in R}}{k-t}\cdot \frac{\1_e}{\P_{R\sim \nu_t}{e\in R}} = \frac{\1_{[n]\setminus \set{a_1,\dots,a_t}}}{k - t}. 
\end{equation}}%
\Tag<sigconf>{\begin{multline}\label{eqn:helpful} \E*{Y_{t+1} \given \F_t} = \sum_{e\in [n]\setminus{a_1,\dots,a_t}}\frac{\P_{R\sim \nu_t}{e\in R}}{k-t}\cdot \frac{\1_e}{\P_{R\sim \nu_t}{e\in R}} =\\  \frac{\1_{[n]\setminus \set{a_1,\dots,a_t}}}{k - t}. 
\end{multline}}%
So
\begin{align*}
	\cov{Z_{t+1}\given \F_t} &=\E*{Y_{t+1}^{\otimes 2}\given \F_t}-\E*{Y_{t+1}\given F_t}^{\otimes 2} \\ &=\frac{1}{k-t}\cdot \parens*{\sum_{e\in [n]\setminus\set{a_1,\dots,a_t}}\frac{\1_e^{\otimes 2}}{\P_{R\sim \nu_t}{e\in R}}}-\Tag<sigconf>{\\ &\quad}\parens*{\frac{\1_{[n]\setminus \set{a_1,\dots,a_t}}}{k-t}}^{\otimes 2}.
\end{align*}
Again, the vector $\1_{[n]\setminus\set{a_1,\dots,a_t}}$ is orthogonal to all $\1_T-\mean{\nu_t}$, and thus is in the kernel of $\cov{\nu_t}$. So \cref{eqn:trickledown-2} simplifies a bit, and we obtain the equation
\Tag{\begin{equation}\label{eqn:trickledown-pinning}
\cov{\nu_t} = \E{\cov{\nu_{T}} \given \F_t} +  \sum_{s = t}^{T - 1} \frac{1}{k - s} \E{\cov{\nu_s} N_s^{-1} \cov{\nu_s} \given \F_t}
\end{equation}}%
\Tag<sigconf>{\begin{multline*}\label{eqn:trickledown-pinning}
\cov{\nu_t} = \E{\cov{\nu_{T}} \given \F_t} +\\  \sum_{s = t}^{T - 1} \frac{1}{k - s} \E{\cov{\nu_s} N_s^{-1} \cov{\nu_s} \given \F_t}
\end{multline*}}%
where $N_s$ is a diagonal matrix encoding the marginals of the link at $\set{a_1,\dots,a_s}$; in other words,  $N_s=\diag(\mean{\nu_s}-\1_{\set{a_1,\dots,a_s}})$. When we set $T = t + 1$, this is a familiar equation from which the trickle-down equation of \textcite{Opp18} and matrix trickle-down \cite{abdolazimi2022matrix} can both be derived. \Tag{For the reader's convenience, we illustrate how this works for Oppenheim's trickle-down in \cref{a:oppenheim}.} We also give another example in a similar spirit in \cref{a:basic-trickledown} illustrating trickle-down of semi-log-concavity.
    \Tag{\section{Rapid mixing of Glauber dynamics in Ising models}\label{sec:ising}
In this section, we prove the main result about rapid mixing of Ising models. The supporting lemmas are given after the proof of the main theorem. Some elements of the analysis can be done in a more general setting, but to keep the presentation concrete we hold off on making most generalizations until the next section of the paper. 
\subsection{Covariance bound and approximate tensorization}
\begin{theorem}\label{thm:ising}
Suppose that $\nu$ is a probability measure on $\set{\pm 1}^n$ defined by
	\[ \nu(x) \propto \exp\parens*{\frac{1}{2} \dotprod{ x, J x } + \dotprod{ h, x }} \]
for some external field $h \in \R^n$ and interaction matrix $J$ satisfying $J \succeq 0$. Without loss of generality, suppose that the diagonal of $J$ is constant, so there exists $\alpha$ such that $J_{ii} = \alpha \ge 0$. Let $\eta = \alpha/\opnorm{J} \in [0,1]$. 
Then
\[ \opnorm{\cov{\nu}} \leq q_{\eta}(\opnorm{J}) \]
where $q_{\eta} : \R_{\geq 0} \to \R_{\geq 0} \cup \set{\infty}$ solves the Volterra integral equation
\[ q_{\eta}(z) = \parens*{1 - r\parens*{\eta z}} + \int_0^{z} q(y)^2 dy \]
where
\[ r(a) = \E*_{\sigma \sim \Berpm{\tanh(0)}, g \sim \Normal{0,1}}{\tanh^2(a\sigma + \sqrt{a} g)}.\]
Furthermore, $\nu$ satisfies Approximate Tensorization of Entropy with constant at most
\[ \exp\parens*{\int_0^{\opnorm{J}} q_{\eta}(z) dz }. \]
\end{theorem}
\begin{proof}
Let $\nu_t$ be the stochastic localization process with driving matrix $C_t = J^{1/2}$ and $\nu_0 = \nu$, and define $\Sigma_t = \cov{\nu_t}$. 
Recall from \cref{eqn:sl-trickledown} that
\[ 
\Sigma_{t} = \E{\Sigma_{1} \given \mathcal{F}_t} + \int_t^1 \E{\Sigma_{s} J \Sigma_{s} \given \mathcal{F}_{t}} ds 
\]
so applying \cref{lem:tbound} we have
\[ \Sigma_{t} \preceq \diag(1 - r(\alpha(1 - t))) + \int_t^1 \E{\Sigma_{s} J \Sigma_{s} \given \mathcal{F}_{t}} ds. \]
Define $\beta(t) = \sup_{h} \opnorm{\cov{\mu_{(1 - t)J, h}}}$,
where $\mu_{(1 - t) J, h}$ denotes an Ising model with interaction matrix $(1 - t) J$ and external field $h$. Then, by considering the stochastic localization process started from this model, we have that $\beta(1) = 1$ and
\[ \beta(t) \le [1 - r(\alpha(1 - t))] + \opnorm{J} \int_t^1 \beta(s)^2 ds \]
so we can apply \cref{lem:integral-equation}. Therefore \[ \opnorm{\Sigma_0} \le \gamma(1) \]
where $\gamma(t)$ is the solution to \eqref{eqn:gamma} with $K = \opnorm{J}$ and $f(t) = 1 - r(\alpha t)$,
so
\[ \gamma(t) = [1 - r(\alpha t)] + \opnorm{J} \int_0^t \gamma(s)^2 ds. \]
Making the change of variables $q(z) = \gamma(z /\opnorm{J})$ (i.e., $z = \opnorm{J}t$) yields the integral equation
\[ q_{\eta}(z) = [1 - r(\alpha z/\opnorm{J})] + \opnorm{J} \int_0^{z/\opnorm{J}} \gamma(s)^2 ds = [1 - r(\alpha z/\opnorm{J})] + \int_0^{z} q_{\eta}(y)^2 ds\]
where we made the corresponding change of variables $y = \opnorm{J}s$ inside the integral. Hence
\[ \opnorm{\Sigma_0} \le q_{\eta}(\opnorm{J}) \]
which proves the bound on the covariance matrix.

By \cref{lem:entropic-stability}, this implies that every such $\nu$ is $\opnorm{J} q_{\eta}(\opnorm{J})$-entropically stable with respect to the function $\psi(x,y) = \frac{1}{2} \norm{J^{1/2} (x - y)}_2^2$. The measure $\nu_1$ is a product measure and therefore satisfies ATE with constant $1$. Therefore, for any nonnegative function $f$
\begin{align*} \Ent_{\nu_0}{f} 
&\le \exp\parens*{\int_0^1 q_{\eta}(t \opnorm{J}) \opnorm{J} dt} \E {\Ent_{\nu_1}{f}} \\
&\le \exp\parens*{\int_0^1 q_{\eta}(t \opnorm{J}) \opnorm{J} dt} \sum_i \E {\E_{X_{\sim i} \sim \nu_1}{\Ent_{\nu_1(\cdot \mid X_{\sim i})}{f}}} \\
&\le \exp\parens*{\int_0^1 q_{\eta}(t \opnorm{J}) \opnorm{J} dt} \sum_i \E_{X_{\sim i} \sim \nu_0}{\Ent_{\nu_0(\cdot \mid X_{\sim i})}{f}}
\end{align*}
where in the first inequality we used \cref{prop:conservation-of-entropy}, the second inequality is tensorization for the product measure, and the last inequality is the supermartingale property (\cref{lem:supermartingale}).
Making the change of variables $z = t \opnorm{J}$ proves the result. 

\end{proof}

\begin{lemma}\label{lem:integral-equation}
Suppose $f(t) \ge 0$ for $t \ge 0$ and  $K \ge 0$. 
Either there exists a unique solution for all time, or
there exists $T > 0$ and a solution on $[0,T)$ to the integral equation 
\begin{equation}\label{eqn:gamma}
\gamma(t) = f(t) + K \int_0^t \gamma(s)^2 ds,
\end{equation}
such that the solution is unique and satisfies $\gamma(t) \to \infty$ as $t \to T$. 
Suppose $\alpha(t)$ is a function 
satisfying 
\[ \alpha(t) \le f(t) + K \int_0^t \alpha(s)^2 ds. \]
Then $\alpha(t) \le \gamma(t)$. 
\end{lemma}
\begin{proof}
See Chapter 5 of \cite{lakshmikantham1969differential} --- Theorem 5.3.1 establishes the comparison inequality, Theorem 5.2.1 establishes existence, and Theorem 5.4.4 establishes uniqueness. 
\end{proof}
We recall the following basic fact.
\begin{lemma}\label{lem:bernoulli}
Suppose that $\mu$ is a probability measure on $\set{\pm 1}$ such that $\mu(x) \propto \exp(hx)$. Then $\E_{\mu} X = \tanh(h)$ and $\Var_{\mu} X = 1 - \tanh^2(h)$. 
\end{lemma}
\begin{proof}
This follows from the definitions using that 
$\tanh(h) = \frac{e^h - e^{-h}}{e^h + e^{-h}}$. 
\end{proof}
\begin{lemma}\label{lem:tbound}
Under the assumptions of \cref{thm:ising}, let $\nu_t$ be the stochastic localization process with driving matrix $C_t = J^{1/2}$. Then for any $t \in [0,1]$,
\[ \E{\Sigma_{1} \given \mathcal{F}_t} \preceq \diag([1 - r(\alpha \cdot (1 - t))]_{i \in [n]}) \]
where $r$ is defined in \cref{lem:elementary-lb} and  $\Sigma_1 = \cov{\nu_1}$.
\end{lemma}
\begin{proof}
We first prove the result when $t = 0$ and then generalize. 
Recall from \cref{thm:gci} that the measure $\nu_1$ is equal in law to the product measure
\[ \mu(x) \propto \exp\parens*{\dotprod{ J Y_1 + h, x }}\]
where $Y_1 = X^* + J^{-1/2} G$ is generated by sampling $X^* \sim \nu_0$ and independently $G \sim \Normal{0,I}$, so that
\[ J Y_1 + h = J X^* + h + J^{1/2} G. \]
Hence $\Sigma_1$ is diagonal and furthermore by applying \cref{lem:bernoulli} we have
\[ \E{(\Sigma_1)_{ii}} = 1 - \E_{X^* \sim \nu_0,g \sim \Normal{0,1}}{\tanh^2(\dotprod{ J_i, X^* } + h_i + \sqrt{J_{ii}} g)} \le 1 - r(J_{ii}) \]
where $r$ is defined in \cref{lem:elementary-lb} which we applied to get the final bound. This proves the result in the special case $t = 0$.

We now consider the case where $t > 0$. We know, either by explicit calculation or by \cref{thm:gci}, that $\nu_t$ is an Ising model with interaction matrix $(1 - t) J$ and external field $h_t$. Furthermore, conditional on $\mathcal F_t$ the measure $\nu_1$ is equal in law to a product measure
\[ \mu(x) \propto \exp(\dotprod{ J Y'_1 + h_t, x })\]
where $X' \sim \nu_t$
\[ Y'_1 = (1 - t) X^* + \sqrt{1 - t}\, J^{-1/2} G' \]
with $G' \sim \Normal{0,I}$ so that
\[ (1 - t) J Y'_1 + h_t = (1 - t) J X' + h_t + \sqrt{1 - t} J^{1/2} G. \]
Hence applying \cref{lem:elementary-lb} in the same way proves the result in general. 
\end{proof}
\begin{lemma}\label{lem:elementary-lb}
Suppose that $\nu$ is an Ising model on $n$ sites with interaction matrix $J$ and external field $h$. Then for any $i \in [n]$
\[ \E_{X^* \sim \nu_0,g \sim \Normal{0,1}}{\tanh^2(\dotprod{ J_i, X^* } + h_i + \sqrt{J_{ii}} g)} \ge \inf_{b} r(J_{ii})   \] 
where
\[ r(a) = \E_{\sigma \sim \operatorname{Ber}_{\pm}(0), g \sim \Normal{0,1}}{\tanh^2(a\sigma + \sqrt{a} g)}\]
and $\operatorname{Ber}_{\pm}(z)$ is the law of a random variable valued in $\set{\pm 1}$ with mean $z$. 
\end{lemma}
\begin{proof}
Using that
\[ \E{X^*_i \given X^*_{\sim i}} = \tanh(J_{i,\sim i} \cdot X^*_{\sim i} + h_i) \]
and considering $b = \dotprod{ J_{i,\sim i}, X^*_{\sim i} } + h_i$,
we can lower bound
\begin{align*}  
\E_{X^* \sim \nu_0,g \sim \Normal{0,1}}{\tanh^2(\dotprod{ J_i, X^* } + h_i + \sqrt{J_{ii}} g)} 
&= \E_{X^* \sim \nu_0,g \sim \Normal{0,1}}{\E{\tanh^2(\dotprod{ J_i, X^* } + h_i + \sqrt{J_{ii}} g) \mid b}} \\
&\ge \inf_{b \in \R} \E_{\sigma \sim \operatorname{Ber}_{\pm}(\tanh(b)), g \sim \Normal{0,1}}{\tanh^2(b + J_{ii} \sigma + \sqrt{J_{ii}} g)}.
\end{align*}
By Lemma~\ref{lem:ising-optimal-b}, the infimum on the right hand side is attained at $b = 0$.
\end{proof}

\subsection{A monotonicity inequality via coupling stochastic localizations}
In this section we prove the following lemma as a consequence of a more general construction of a coupling of two stochastic localization processes, which we elaborate upon below.
\begin{lemma}\label{lem:ising-optimal-b}
For any $a \ge 0$, the global minimum of the function $t_a : \mathbb{R} \to \mathbb R$ defined by
\[ t_a(b) = \E_{\sigma \sim \operatorname{Ber}_{\pm}(\tanh(b)), g \sim \Normal{0,1}}{\tanh^2(b + a \sigma + \sqrt{a} g)} \]
is attained at $b = 0$.
\end{lemma}
\begin{proof}
It is equivalent to show that the global maximum of $1 - t_a$ is attained at $b = 0$. Using the Gaussian channel interpretation (Theorem~\ref{thm:gci}), $1 - t_a(b)$ can be interpreted as the operator norm of the covariance matrix of stochastic localization initialized at the measure $\operatorname{Ber}_{\pm}(\tanh(b))$ and run for time $a$. Since $\tanh^2$ is monotonically increasing on $\mathbb{R}_{\ge 0}$ and symmetric, the assumption of \cref{lem:rotation-coupling-cons} is satisfied, and so the desired conclusion follows.
\end{proof}
It remains to establish the desired fact about stochastic localization.
Note that the stochastic localization process used here is different from the one in the previous section (the process here is operating on a single ``spin'' of the system, without interacting with the other spins).
Informally, the general lemma established below allows us to lift certain monotonicity inequalities from the final measure arrived at in stochastic localization back to the initial measure.

The argument takes advantage of symmetry in a key way. 
\begin{definition}
We say a probability measure $\mu$ on $\mathbb R^N$ is spherically symmetric if for every orthogonal matrix $R$ and measurable set $S$, $\mu(R S) = \mu(S)$.
\end{definition}
For example, when $N = 1$ spherical symmetry means that the measure is unchanged under the reflection map $x \mapsto -x$. We can now state a more general monotonicity result.

\begin{lemma}\label{lem:rotation-coupling-cons}
Suppose that $\mu$ is a probability measure in $\mathbb R^N$ which is spherically symmetric. Let $T \ge 0$.
Suppose that the following reverse monotonicity inequality holds: for any $v,w\in \mathbb R^N$ such that $\|v\| \le \|w\|$,
\[ \|\cov{\mathcal T_v \mathcal G_T \mu}\|_{OP} \ge \|\cov{\mathcal T_w \mathcal G_T \mu}\|_{OP}. \]
Then the following reverse monotonicity inequality holds. Let  $v,w \in \mathbb R^N$ be such that $\|v\| \le \|w\|$.
Suppose that for $h \in \{v,w\}$, the law of stochastic process $\mu_t(h)$ for $t \ge 0$ is stochastic localization with identity driving matrix initialized at $\mu_0(h) = \mathcal T_h \mu$, i.e. initialized at
\[ \frac{d\mu_0(h)}{d\mu}(x) \propto e^{\dotprod{h,x}}. \]
Then
\[ \E{\|\cov{\mu_T(v)}\|_{OP}} \ge \E{\|\cov{\mu_T(w)}\|_{OP}}. \]
\end{lemma}
\begin{proof}
This follows directly from \cref{lem:rotation-coupling-sl}, since it constructs a coupling where $\|\cov{\mu_T(v)}\|_{OP} \ge \|\cov{\mu_T(w)}\|_{OP}$ holds almost surely (in other words, one stochastically dominates the other). Taking expectation on both sides proves the desired result.
\end{proof}

Now we state and prove the crucial coupling result. 

\begin{lemma}\label{lem:rotation-coupling-sl}
Suppose that $\mu$ is a probability measure in $\mathbb R^N$ which is rotationally symmetric. For any $v,w$ with $\|v\| \le \|w\|$ there exists a joint coupling of stochastic processes $\mu_t(v), \mu_t(w)$ indexed by $t \ge 0$ such that:
\begin{enumerate}
    \item The marginal law of the stochastic process $(\mu_t(v))_t$ is stochastic localization initialized at $\mu_0(v) = \mathcal T_v \mu$ with driving matrix $C_t = I$.
    \item Likewise, the marginal law of $(\mu_t(w))_t$ is stochastic localization initialized at $\mu_0(w) = \mathcal T_w \mu$ with driving matrix $C_t = I$.
    \item There exist adapted $\mathbb R^N$-valued stochastic processes $v_t,w_t$ such that $\mu_t(v) = \mathcal{T}_{v_t} \mathcal{G}_t \mu$, $\mu_t(w) = \mathcal{T}_{w_t} \mathcal{G}_t \mu$, and $\|v_t\| \le \|w_t\|$ for all times $t$ almost surely.
\end{enumerate}
\end{lemma}
\begin{proof}
    For simplicity, we give the proof below in the case that $\mu$ is a measure on a discrete set, but the same argument works in general with only minor changes.
     
    The construction is a multi-step process. First we construct $\mu_t(v)$ for all times as stochastic localization driven by an $N$-dimensional Brownian motion $W_t$.
    Recall that the definition of stochastic localization driven by Brownian motion $W_t$ is that
    \[ d\mu_t(v)(x) = \mu_t(v)(x) \langle x - \mean{\mu_t(v)}, dW_t \rangle \]
 so that by Ito's formula
    \begin{align*} d\log \mu_t(v)(x) 
    &= \langle x - \mean{\mu_t(v)}, dW_t \rangle - \frac{1}{2} \|x - \mean{\mu_t(v)}\|^2 dt \\
    &= \langle x, dW_t \rangle + \langle x, \mean{\mu_t(v)} \rangle dt - \frac{1}{2} \|x\|^2 dt - \frac{1}{2} \|\mean{\mu_t(v)}\|^2 dt -\langle \mean{\mu_t(v)}, dW_t\rangle .  
    \end{align*}
    where we note that the last two terms are independent of $x.$
    Hence if we define $v_t$ by $v_0 = v$ and $dv_t = dW_t + \mean{\mu_t(v)} dt$ then we have $\mu_t(v) = \mathcal{T}_{v_t} \mathcal{G}_t \mu$.

    Now to construct $\mu_t(w)$, let $B_t$ be an independent $N$-dimensional Brownian motion and construct $\mu_t$ and $w_t$ in the same way as stochastic localization driven by Brownian motion $B_t$, for all times $t$ up to stopping time 
    \[ \tau = \inf \{ t \ge 0 : \|v_t\| \ge \|w_t\| \}. \]
    Observe by continuity that $\|v_{\tau}\| = \|w_{\tau}\|$ almost surely since initially $\|v\| \le \|w\|$.
    From time $\tau$ onward, we define $\mu_t(w)$ in the following way. Let $R$ be an orthogonal operator such that $R v_{\tau} = w_{\tau}$. For $t > \tau$, define
    \[ B'_t = R(W_t - W_{\tau})\1[t > \tau]  + B_{\min(t,\tau)} \]
    and observe that the law of the process $B'_t$ is that of a standard $N$-dimensional Brownian motion. Now for all times $t$ we can define
    \[ d\mu_t(w)(x) = \mu_t(w)(x) \langle x - \mean{\mu_t(w)}, dB'_t \rangle \]
    to be the stochastic localization process driven by $B'_t$, which extends the definition we gave before past the stopping time $\tau$. Similarly, we can now define for all times that $w_t$ is the corresponding tilt given by the SDE $dw_t = dB'_t + \mean{\mu_t(w)}dt$.

    On the other hand, for $t > \tau$ observe that
    \[ d (R v_t) = R dv_t = R dW_t + R \mean{\mathcal T_{v_t} \mathcal G_t \mu} = dB'_t + \mean{\mathcal T_{R v_t} \mathcal G_t \mu} \]
    by \cref{lem:mean-under-orthogonal}. This is exactly the SDE defining $w_t$, so by standard uniqueness theory \cite{karatzas2014brownian} we have that $w_t = R v_t$ for $t > \tau$. 

    In conclusion, from the definition of the stopping time we have that $\|v_t\| \le \|w_t\|$ for $t < \tau$, and by our construction above we have that $\|v_t\| = \|w_t\|$ for $t \ge \tau$. This proves the final claim.
\end{proof}
\begin{lemma}\label{lem:mean-under-orthogonal}
Suppose that $\mu$ is a spherically symmetric distribution in $\mathbb R^N$. Then for any orthogonal matrix $R$ and vector $v \in \mathbb{R}^N$ such that $\mathcal T_v \mu$ exists, $R \mean{\mathcal T_v \mu} = \mean{\mathcal T_{R v} \nu}$.
\end{lemma}
\begin{proof}
Expanding the definition, we have that
\begin{align*} 
R \mean{\mathcal T_v \mu} 
&= \frac{1}{Z_v} \int R x e^{\langle v, x \rangle} d\mu(x)  \\
&= \frac{1}{Z_v} \int y e^{\langle v, R^{-1} y \rangle} d\mu(y) \\
&= \frac{1}{Z_v} \int y e^{\langle R v, y \rangle} d\mu(y) \\
&= \mean{\mathcal T_{R v} \mu}
\end{align*}
where $Z_v = \int e^{\langle v, x \rangle} d\mu(x)$ is the (rotationally invariant) moment generating function, in the second step
we made a change of variables $y = R x$ and used rotational invariance of $\mu$ i.e. $ d\mu(Rx) = d\mu (x)$, and in the third step we again used that $R$ is an orthogonal transformation i.e. $ \langle v, R^{-1} y\rangle = v^\intercal R^{-1} y = v^\intercal R^\intercal y = \langle Rv, y\rangle$.
\end{proof}
}
    \Tag{\section{Rapid mixing of Glauber dynamics in spin systems over semi-log-concave base measures}\label{sec:semi}
In this section, we establish a version of our analysis for spin systems where the base measure is semi-log-concave. This greatly generalizes the setting of the Ising models. The general result is generally less tight than specializing the technique to the particular measure of interest, but as we will show the result is easy to apply, and interestingly the resulting integral equation often admits a closed-formula solution.

We will need to use the following standard notation in the argument: $A \otimes B$ denotes the Kronecker product of matrices $A$ and $B$. Explicitly, if $A$ is a matrix of dimension $m \times n$ then 
\[ A \otimes B = \begin{bmatrix} 
A_{11} B & \cdots & A_{1n} B \\
\vdots & \ddots & \vdots \\ A_{m1} B & \cdots & A_{mn} B\end{bmatrix}. \]

\begin{theorem}\label{thm:semilogconcave}
Suppose that $\mu = \bigotimes_{i = 1}^n \mu^{(i)}$ is a product measure where each $\mu^{(i)}$ is a probability measure on $\mathbb R^N$. Let $J \succeq 0$ with $J_{ii} = \alpha$ for all $i \in [n]$, and define the measure $\nu$ by its density
\[ \frac{d\nu}{d\mu}(x) \propto \exp\left(\frac{1}{2} \sum_{i,j} J_{ij} \langle x_i, x_j \rangle \right). \]
Let $\eta = \alpha/\|J\|_{OP}$.
Suppose that for all $i \in [n]$ and $s \in [0,\alpha]$ the measure $\mathcal{G}_{-s} \mu^{(i)}$ defined by
\begin{equation}\label{eqn:strong-slc}
\frac{d\mathcal{G}_{-s} \mu^{(i)}}{d\mu^{(i)}}(x) \propto \exp(s\|x\|^2/2) 
\end{equation}
exists and is $\rho(s) > 0$ semi-log-concave. Then $\|\cov{\nu}\|_{OP} \le q_{\eta,\rho}(\|J\|_{OP})$ where $q_{\eta,\rho}$ solves the integral equation
\begin{equation}\label{eqn:qsemilogconcave}
q_{\eta,\rho}(z) = \frac{1}{\eta z + [\rho(\eta z)]^{-1}} + \int_0^z q_{\eta,\rho}(s)^2 ds
\end{equation}
and ATE holds with constant at most
\[ \exp\left(\int_0^{\|J\|_{OP}} q_{\eta,\rho}(z) dz\right).  \]
\end{theorem}
\begin{remark}
The assumption that the reweighted measure \eqref{eqn:strong-slc} is $\rho(s)$-semi-log-concave lets us obtain much better results than if we only assume $\rho(0)$-semi-log-concavity for $\mu^{(i)}$. For example, if $\mu^{(i)}$ is $\gamma$-strongly log concave, then this assumption is satisfied with $\rho(s) = \frac{1}{\gamma - s}$ by the Brascamp-Lieb theorem \cite{brascamp2002some}. In fact, the same bound holds for all $\gamma$-semi-log-concave measures, see \cref{thm:trickle-down-semi-logconcave}. In our main applications, $\rho$ is constant, which is much better than the generic guarantee, and also in this case we will show that the integral equation is analytically solvable.
\end{remark}
In the next sections, we prove this theorem and also show how to derived the aforementioned analytical solution.
\subsection{Decoupling argument for stochastic localization}
Our analysis of the Ising model involved a key reduction from arguing about a joint stochastic localization process to a more tractable \emph{single-site} stochastic localization process. As a first step in our argument, we revisit this idea and formalize it as a general decoupling principle.
\begin{lemma}\label{lem:decoupling-sl}
Suppose that $\mu = \bigotimes_{i = 1}^n \mu^{(i)}$ is a product measure where each $\mu^{(i)}$ is a probability measure on $\mathbb R^N$. Let $J \succeq 0$ with $J_{ii} = \alpha$ for all $i$, and define the measure $\nu_0$ by its density
\[ \frac{d\nu_0}{d\mu}(x) \propto \exp\left(\frac{1}{2} \sum_{i,j} J_{ij} \langle x_i, x_j \rangle \right). \]
Let $\nu_t$ for $t \ge 0$ be the stochastic localization process with driving matrix $C_t = (J \otimes I_N)^{1/2}$, and for $i \in [n]$ let $\nu_{t}^{(i)}  = \mathcal{L}_{X \sim \nu_t}(X_i)$ denote the (random) marginal law of $X_i$ under $\nu_t$. 
For every $i \in [n]$ there exists a probability measure $\lambda_i$ on $\mathbb{R}^N$ such that the marginal law of the random measure $\nu_1^{(i)}$ satisfies
\[ \mathcal{L}\left(\nu_1^{(i)}\right) = \mathcal{L}\left(\tilde{\nu}^{(i)}_1\right) \]
where the random measure $\tilde{\nu}^{(i)}_1$ is constructed as follows:
\begin{enumerate}
\item Let $H_i \sim \lambda_i$ and $\tilde W_i = (\tilde W_{i,t})_{t}$ is a standard Brownian motion in $\mathbb R^N$ independent of $H_i$.
\item Let \[ \frac{d\tilde{\nu}^{(i)}_0}{d\mu^{(i)}}(x^*_i) \propto \exp(J_{ii} \|x^*_i\|^2 + \dotprod{ H_i, x^*_i }). \]
\item Then, $(\tilde{\nu}^{(i)}_t)_{t \ge 0}$ is the stochastic localization process generated by $\tilde W$ with driving matrix $C_t = J_{ii} I_N$ starting from $\tilde{\nu}^{(i)}_0$.
\end{enumerate}
Furthermore, we have the following block-diagonal decomposition of the average covariance:
\[ \E{\cov{\nu_1}} = \begin{bmatrix} \E*{\cov*{\tilde{\nu}^{(1)}_1}} \\ 
& \ddots \\ & & \E*{\cov*{\tilde{\nu}^{(n)}_1}} \end{bmatrix}. \]
\end{lemma}
\begin{proof}
The first step is to apply \cref{thm:gci} --- let $B_t$ and $W_t$ be the Brownian motions defined there.
By \cref{thm:gci}, at time one stochastic localization cancels out the quadratic term in the log-density yielding
\[ \frac{d\nu_1}{d\mu}(x) \propto \exp\parens*{\dotprod{ (J \otimes I_N) y_1, x }} \]
where $y_1 = X^* + (J \otimes I_N)^{-1/2} B_1 \in \mathbb{R}^{Nn}$ and $X^* \sim \nu_0$. Defining $h^1 = (J \otimes I_N) y_1$ and writing $h^1 = (h^1_i)_{i = 1}^n$ we have that for any $i \in [n]$
\[ h^1_i = \sum_{j = 1}^n J_{ij} X^*_j + \sqrt{J_{ii}} G_i \]
where $G_i = \frac{(B_1)_i}{\sqrt{J_{ii}}} \sim \Normal{0, I_N}$ is independent of $X^*$. Furthermore,
\[ \frac{d\nu_{1}^{(i)}}{\mu^{(i)}}(x_1) \propto \exp(h^1_i x_1). \]
Define $H_i = \sum_{j \ne i} J_{ij} X^*_j$ so that 
\[ h^1_i = H_i + J_{ii} X^*_i + \sqrt{J_{ii}} G_i. \]
Let $\tilde{\nu}^{(i)}_0$ denote the conditional law of $X^*_i$ given $H_i$ 
and observe that its relative density with respect to the base measure is
\[ \frac{d\tilde{\nu}^{(i)}_0}{d\mu^{(i)}}(x^*_i) \propto \exp\left(J_{ii} \|x^*_i\|^2 + \dotprod{ H_i, x^*_i }\right). \]
Let $\tilde{\nu}^{(i)}_t$ be defined as the stochastic localization process with driving matrix $\tilde{C}_t = J_{ii} I_N$ generated by a fresh $N$-dimensional Brownian motion $\tilde W_i$. By appealing to \cref{thm:gci} again to determine the law of $\tilde{\nu}^{(i)}_1$ conditional on $H_i$, we find that the two constructions of random measures on spin $i$ via stochastic localization have the same conditional law:
\[ \mathcal{L}(\nu_1^{(i)} \mid H_i) = \mathcal{L}(\tilde{\nu}^{(i)}_1 \mid H_i). \]
Taking the expectation over $H_i$ yields the stated equality in law, where $\lambda_i$ is the law of $H_i$. From this, the equality in covariances follows since $\nu_1$ is a product measure.
\end{proof}

\subsection{Proof of covariance bound and approximate tensorization}
The decoupled stochastic localization process can be controlled using the following lemma.
\begin{lemma}[Eqn. (16) of \cite{eldan2020taming}]\label{lem:cov-decay}
Let $\mu_t$ be the stochastic localization process initialized at $\mu_0$ with time-homogeneous driving matrix $C_t = Q$. Then
\[ \E{Q\cov{\mu_t}Q} \preceq \left(tI + (Q\cov{\mu_0}Q)^{-1}\right)^{-1} \]
\end{lemma}
Note that this result upper bounds the average covariance matrix at time $t$ by a function of its covariance at time $0$ --- so in a sense, if the trickledown concept is based upon backwards induction, this inequality is a complementary application of forward induction. The lemma has a short proof by combining Gronwall's inequality with the SDE for the covariance matrix. See also \cite{el2022information} for a related proof of this lemma without direct reference to stochastic localization.

\begin{proof}[Proof of \cref{thm:semilogconcave}]
Some steps in this proof are similar to the analysis for the Ising model, so we elaborate in the most detail on the parts which differ.
We view the spin system as a distribution on $\mathbb R^{Nn}$, and let $\nu_t$ be the stochastic localization process with the time-invariant driving matrix $C_t = (J \otimes I_N)^{1/2}$. By \cref{thm:gci}, or by a direct computation, this yields
\[ \frac{d\nu_t}{d\mu}(x) \propto \exp\parens*{\langle h_t, x_i \rangle + \frac{1 - t}{2}\sum_{i,j} J_{ij} \langle x_i, x_j \rangle} \]
for some adapted process $h_t$ in $\mathbb{R}^{Nn}$.
Recall the trickle-down equation
\[ 	\cov{\nu_t} = \E{\cov{\nu_1} \given \F_t} + \int_t^1 \E{\cov{\nu_s} J \cov{\nu_s} \given \F_t} ds 
\]
and applying the triangle inequality yields
\[ \|\cov{\nu_t}\|_{OP} \le \left\|\E{\cov{\nu_1} \given \F_t}\right\|_{OP} + \int_t^1 \|\E{\cov{\nu_s} J \cov{\nu_s} \given \F_t}\|_{OP} ds. \]
By the decoupling argument from Lemma~\ref{lem:decoupling-sl}, the covariance matrix of $\nu_1$ is block-diagonal, and each block can be reinterpreted as the covariance resulting from running stochastic localization with identity covariance in $\mathbb R^N$ for time $J_{ii}(1 - t)$ starting from the measure $\mathcal{G}_{-J_{ii}(1 - t)} \mu^{(i)}$.
By \cref{lem:cov-decay} and the assumption that $J_{ii} = \alpha$ we get that for each $i$, the corresponding block of the covariance satisfies
\[ \E{\cov{\nu_1(X_i = \cdot)} \mid \mathcal F_t} \preceq \left[\alpha(1 - t) I + \cov{\nu_t(X_i = \cdot)}^{-1}\right]^{-1} \preceq \frac{1}{\alpha(1 - t) + 1/\rho(\alpha (1 - t))} I\]
where in the last step we used the semi-logconcavity assumption. We therefore have the inequality
\[ \|\cov{\nu_t}\|_{OP} \le \frac{1}{\alpha(1 - t) + 1/\rho(\alpha (1 - t))} + \|J\|_{OP} \int_t^1 \E{\|\cov{\nu_s}\|_{OP}^2 \given \F_t} ds. \]
As in the argument for the Ising model, by considering at every time the worst-case value of $\|\cov{\nu_t}\|_{OP}$ and making a comparison argument, this yields the inequality $\|\cov{\nu_t}\|_{OP} \le f(t)$ where $f$ solves the integral equation
\[ f(t) = \frac{1}{\alpha(1 - t) + 1/\rho(\alpha (1 - t))} + \|J\|_{OP} \int_t^1 f(s)^2 ds. \]
Recalling that $\eta = \alpha/\|J\|_{OP}$, so $\alpha = \eta \|J\|_{OP}$, we can rewrite this as
\[ f(t) = \frac{1}{\eta \|J\|_{OP}(1 - t) + 1/\rho(\alpha (1 - t))} + \int_{0}^{\|J\|_{OP}(1 - t)} f(t + s/\|J\|_{OP})^2 ds. \]
Letting $z = \|J\|_{OP}(1 - t)$ and $q_{\eta,\rho}(z) = f(1 - z/\|J\|_{OP})$ yields
\begin{align*} 
q_{\eta,\rho}(z) 
&= f(1 - z/\|J\|_{OP}) \\
&= \frac{1}{\eta z + [\rho(\eta z)]^{-1}} + \int_0^z f(1 - (z - s)/\|J\|_{OP})^2 ds \\
&=  \frac{1}{\eta z + [\rho(\eta z)]^{-1}} + \int_0^z q_{\eta, \rho}(z - s)^2 ds \\
&=   \frac{1}{\eta z + [\rho(\eta z)]^{-1}} + \int_0^z q_{\eta, \rho}(s)^2 ds.
\end{align*}
Approximate Tensorization of Entropy (ATE) follows from the covariance bound in the same way as in the Ising case (using entropic stability and the super-martingale property). 
\end{proof}
\subsubsection{Aside: A general distribution-specific bound}
In some cases we expect to obtain tighter results by avoiding the generic bound from \cref{lem:cov-decay}, e.g. as we did in the Ising case. 
In this spirit and for completeness, we record the following result which follows from the ideas we have already explained. Although computing the bound reduces to evaluating quantities which are independent of the number of spins $n$, in practice it may not be as easy to apply as the bound based upon semi-log-concavity. We do not build upon this result in the remainder of this paper.
\begin{theorem}\label{thm:hard-to-use}
Suppose $N \ge 2, n \ge 1$ and that $\nu$ is the probability distribution on $\mathbb R^{Nn}$ with  density
\[ \frac{d\nu}{d\mu}(x) \propto \exp\parens*{\frac{1}{2} \sum_{1 \le i,j \le n} J_{ij} \dotprod{ x_i, x_j } + \sum_{i = 1}^n \dotprod{ h_i, x_i } } \]
with respect to the probability measure $\mu=\otimes_{i = 1}^n \mu_0$ where $\mu_0$ is a probability measure on $\mathbb R^N$.
Here $J$ and $h$ are parameters and we assume the interaction matrix $J$ satisfies $J \succeq 0$. Suppose, without loss of generality, that the diagonal of $J$ is constant, so there exists $\alpha$ such that $J_{ii} = \alpha \ge 0$. Let $\eta = \alpha/\opnorm{J} \in [0,1]$. 
Then
\[ \opnorm{\cov{\nu}} \le q_{\eta,\mu_0}(\opnorm{J}) \]
where $q_{\eta,\mu_0} : \R_{\ge 0} \to \R_{\ge 0} \cup \set{\infty}$ solves the Volterra integral equation
\[ q_{\eta,\mu_0}(z)=r_{\mu_0}(\eta z)+\int_{0}^z q_{\eta,\mu_0}(y)^2dy, \]
where
\[ r_{\mu_0}(t)=\sup_{b \in \mathbb R^N} \opnorm*{\E*_{x\sim \mu_0, g\sim \Normal{0, I}}{\cov*{\mathcal G_{-t}\tilt_{b + tx+\sqrt{t}g}\mu_0}}}. \]
Furthermore, Approximate Tensorization of Entropy (ATE) is satisfied with constant at most
\[ \exp\parens*{\int_0^{\opnorm{J}} q_{\eta, N}(z) dz }. \]
\end{theorem}
\begin{proof}
The proof is the same as that of \cref{thm:semilogconcave}, except that we do not apply \cref{lem:cov-decay} in which case the quantity $r_{\mu_0}$ remains in the final bound.
\end{proof}
\subsection{Analytical solution using the Riccati equation}
The following argument allows us to explicitly solve the integral equation in many cases of interest, where the semi-logconcavity bound is constant.
\begin{theorem}\label{thm:fixed-semilogconcave}
Suppose that the assumption of \cref{thm:semilogconcave} is satisfied for a constant function $\rho(s) = \rho > 0$. Then the solution of \eqref{eqn:qsemilogconcave} exists, and is finite and unique, for $z \in [0, s(\eta)/\rho)$ and is given by
\[ q_{\eta,\rho}(z) = \rho Q_{\eta}(\rho z) \]
where $Q_{\eta} : [0,s(\eta)) \to \mathbb{R}_{\ge 0}$ is given by
\[ Q_{\eta}(z) = \eta \frac{-\lambda_1 \lambda_2^2 (\eta z + 1)^{\lambda_1 - \lambda_2} + \lambda_1^2 \lambda_2}{\lambda_2^2(\eta z + 1)^{\lambda_1 - \lambda_2 + 1} - \lambda_1^2 (\eta z + 1)}, \]
with $\lambda_1 > 0 > \lambda_2$ the two roots of the quadratic equation 
$\eta \lambda^2 - \eta \lambda - 1 = 0$, explicitly
\[ \lambda_1 = \lambda_1(\eta) = \frac{\eta + \sqrt{\eta^2 + 4\eta}}{2 \eta}, \qquad \lambda_2 = \lambda_2(\eta) = \frac{\eta - \sqrt{\eta^2 + 4\eta}}{2 \eta}, \]
and where 
\[ s(\eta) = \frac{1}{\eta} \left[\left|\frac{\lambda_1}{\lambda_2}\right|^{2/(\lambda_1 - \lambda_2)} - 1\right]. \]
\end{theorem}
\begin{figure}
    \centering
    \includegraphics[width=0.6\linewidth]{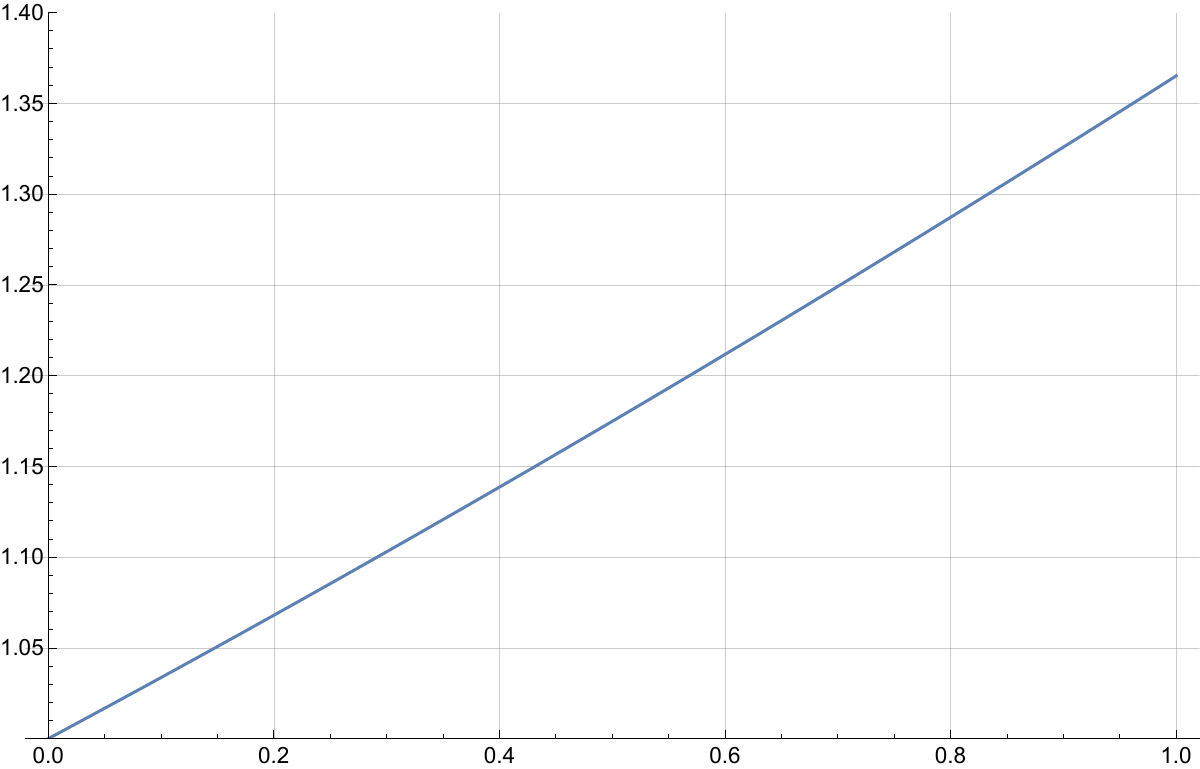}
    \caption{Plot of $s(\eta)$  as a function of $\eta$ varies from $0$ to $1$. The function is very close to, but not exactly, affine on this domain.}
    \label{fig:zplot}
\end{figure}
\begin{proof}
In this case we have
\begin{equation}\label{eqn:integral-form-semilogconcave-fixed}
q_{\eta,\rho}(z) = \frac{1}{\eta z + \rho^{-1}} + \int_0^z q_{\eta,\rho}(s)^2 ds. 
\end{equation}
We now show how to exhibit a solution of this equation (we already know it will be unique, see \cref{lem:integral-equation}). 
Differentiating with respect to $z$ yields
\[ q'_{\eta,\rho}(z) = \frac{-\eta}{(\eta z + \rho^{-1})^2} + q_{\eta,\rho}(z)^2. \]
This is a Riccati equation so we can solve it by making an appropriate substitution (see e.g.\ \cite{reid1972riccati}).
Explicitly, making the substitution $q_{\eta, \rho} = -u'/u$ yields a differential equation
\[ \frac{d^2 u}{dz^2} - \frac{\eta u}{(\eta z + \rho^{-1})^2} = 0. \]
Making the substitution $x = \log(\eta \rho z + 1)$ and simplifying yields the differential equation
\[ \eta u'' - \eta u' - u = 0. \]
We can guess particular solutions by observing that $u = e^{\lambda x}$ yields the quadratic equation
\[ \eta \lambda^2 - \eta \lambda - 1 = 0 \]
which has solutions
\[ \lambda = \frac{\eta \pm \sqrt{\eta^2 + 4 \eta}}{2\eta} \]
and so in general for some $c_1,c_2 \in \mathbb R$ the solution must be of the form
\[ u = c_1 e^{\lambda_1 x} + c_2 e^{\lambda_2 x} = c_1 (\eta \rho z + 1)^{\lambda_1} + c_2 (\eta \rho z + 1)^{\lambda_2} \]
where $\lambda_1 > 0,\lambda_2 < 0$ are the two roots. Substituting back, this gives 
\begin{align*} 
q_{\eta,\rho}
&= \eta \rho\  \frac{-c_1 \lambda_1 (\eta \rho z + 1)^{\lambda_1 - 1} - c_2 \lambda_2 (\eta \rho z + 1)^{\lambda_2 - 1}}{c_1 (\eta \rho z + 1)^{\lambda_1} + c_2 (\eta \rho z + 1)^{\lambda_2}} \\
&=  \eta \rho\ \frac{-c_1 \lambda_1 (\eta \rho z + 1)^{\lambda_1 - \lambda_2} - c_2 \lambda_2}{c_1 (\eta \rho z + 1)^{\lambda_1 - \lambda_2 + 1} + c_2 (\eta \rho z + 1)}.
\end{align*}
Assume for the purpose of contradiction that $c_1 = 0$, then $q_{\eta,\rho}(z) = -\eta \rho \lambda_2/(\eta \rho z + 1)$ so $q_{\eta,\rho}$ is a solution for all $z \ge 0$ and $q_{\eta,\rho} \to 0$ as $z \to \infty$. However, from \eqref{eqn:integral-form-semilogconcave-fixed} we can see $q_{\eta,\rho}(0) = \rho > 0$, so the integral on the right hand side of \cref{eqn:integral-form-semilogconcave-fixed} and thereby $q_{\eta,\rho}$ are bounded below by a positive constant for all $z$. By contradiction, $c_1 \ne 0$.

Therefore, if we define $C = -c_2/c_1$ we can write
\begin{align}\label{eqn:q-C}
q_{\eta,\rho}(z)
&=  \eta \rho \frac{-\lambda_1 (\eta \rho z + 1)^{\lambda_1 - \lambda_2} + C \lambda_2}{(\eta \rho z + 1)^{\lambda_1 - \lambda_2 + 1} - C (\eta \rho z + 1)}. 
\end{align}
In the case $z = 0$, we therefore have
\[ \rho = q_{\eta}(0) = \eta \rho \frac{-\lambda_1 + C \lambda_2}{1 - C} \]
so $1 - C = -\eta \lambda_1 + C \eta \lambda_2$, i.e.
\[ C = \frac{1 + \eta \lambda_1}{1 + \eta \lambda_2} = \frac{\lambda_1^2}{\lambda_2^2} \]
where we used that $\lambda_1,\lambda_2$ are roots of $\eta \lambda^2 = 1 + \eta \lambda$.
So multiplying through by $\lambda_2^2$, we have
\[ q_{\eta,\rho}(z) =  \eta \rho \frac{-\lambda_1 \lambda_2^2 (\eta \rho z + 1)^{\lambda_1 - \lambda_2} + \lambda_1^2 \lambda_2}{\lambda_2^2(\eta \rho z + 1)^{\lambda_1 - \lambda_2 + 1} - \lambda_1^2 (\eta \rho z + 1)} \]
This formula is only valid for $z$ from $0$ to the first positive root $z_1$ of the denominator (since e.g. the right hand side is typically infinite and therefore no longer a solution of the ODE at time $z_1$). Since $\eta \rho z + 1 > 0$ for $z \ge 0$, the positive roots must satisfy
\[ (\eta\rho z + 1)^{\lambda_1 - \lambda_2} = C. \]
So in fact there is exactly one positive root and it is given by
\[ \frac{1}{\eta \rho} \left[C^{1/(\lambda_1 - \lambda_2)} - 1\right] = s(\eta)/\rho. \]
\end{proof}
}    
    \Tag{\section{Rapid mixing of Langevin and Glauber dynamics in \texorpdfstring{$O(N)$}{O(N)} model}\label{sec:ON}
For any $N \ge 1$, let $S_{N - 1}$ be the unit sphere in $\R^N$. Recall that for $N \ge 2$, this is a connected smooth submanifold of $\R^N$ with dimension $N - 1$. 
\subsection{Covariance bounds and rapid mixing of Langevin}
We start with the following result which gives the semi-log-concavity estimate.
\begin{theorem}[Theorem D.2 of \cite{dyson1978phase}]\label{thm:dyson-sphere}
For $N \ge 1$, define $f_N(h)$ to be the cumulant generating function for the $N - 1$ dimensional unit sphere, so
\[ f_N(h) = \log \int_{s \in S_{N - 1}} \exp(h \cdot s) \mu(dA)  \]
where $\mu$ is the uniform measure on the sphere.
Then:
\begin{enumerate}
    \item $f$ is spherically symmetric, so $f_N(h) = F_N(\norm{h}_2)$ where $F_N : \R_{\ge 0} \to \R$. 
    \item The function $g_N(h) = F'_N(h)$ solves the differential equation
    \[ g'_N(h) + g_N(h)^2 + (N - 1) g_N(h)/h = 1 \]
    on $\R_{> 0}$ with boundary conditions $g_N(0) = 0$ and $g'_N(0) = 1/N$. 
    \item We have $g''_N(0) = 0$ and $g''_N(h) < 0$ for all $h > 0$. In particular, $g_N$ is concave on $\R_{\ge 0}$.
    \item For $h \ne 0$, the eigenvalues of $\nabla^2 f_N(h)$ are $g'_N(\norm{h}_2)$ with multiplicity one and $g_N(\norm{h}_2)/\norm{h}_2$ with multiplicity $N - 1$. In particular, for all $h$ we have $\opnorm{\nabla^2 f_N} \le \max_{h \ge 0} g'_N(h) = 1/N$.
\end{enumerate}
\end{theorem}
\begin{remark}
The differential equation defining $g_N$ is a Riccati equation which actually admits an explicit solution in terms of Bessel functions. The solution can be found using the same techniques from the proof of \cref{thm:fixed-semilogconcave}. 
\end{remark}
\begin{theorem}\label{thm:ON}
Suppose $N \ge 2, n \ge 1$ and that $\nu$ is the probability distribution on $S_{N - 1}^{\otimes n}$ with probability density
\[ \frac{d\nu}{d\mu}(x) \propto \exp\parens*{\frac{1}{2} \sum_{1 \le i,j \le n} J_{ij} \dotprod{ x_i, x_j } + \sum_{i = 1}^n \dotprod{ h_i, x_i } } \]
with respect to the uniform measure $\mu$ on $S_{N - 1}^{\otimes n}$. Here $J,h$ are parameters and we assume the interaction matrix $J$ satisfying $J \succeq 0$. Suppose (without loss of generality) that the diagonal of $J$ is constant, so there exists $\alpha$ such that $J_{ii} = \alpha \ge 0$. Let $\eta = \alpha/\opnorm{J} \in [0,1]$. 

Then
\[ \opnorm{\cov{\nu}} \le q_{\eta,1/N}(\opnorm{J}) \]
where $q_{\eta,1/N}$ is as defined in \cref{thm:fixed-semilogconcave}.
Furthermore, the log-Sobolev constant of the Langevin dynamics is at least
\[ C_N \exp\parens*{-\int_0^{\opnorm{J}} q_{\eta,1/N}(z) dz } \]
where $C_N > 0$ is a constant depending only on $N$, inherited from \cite{zhang2011uniform}. 
\end{theorem}
\begin{proof}
The covariance bound follows from \cref{thm:semilogconcave}, \cref{thm:fixed-semilogconcave}, and the fact that the spherical measure is $1/N$-semi-log-concave from \cref{thm:dyson-sphere}.

The log-Sobolev bound also follows from an completely analogous argument (using the same stochastic localization process as in \cref{thm:semilogconcave}, entropic stability, and the fact that the right hand side of \eqref{eqn:lsi-sphere} is a martingale under stochastic localization), except that to cover the base case we appeal to the uniform log-sobolev inequality from \cite{zhang2011uniform}.
\end{proof}
\subsection{Nearly linear time sampling via Glauber dynamics}
\begin{theorem}\label{thm:ON glauber}
In the same setting as \cref{thm:ON}, the probability measure $\nu$ satisfies ATE with constant at most
\[ \exp\parens*{\int_0^{\opnorm{J}} q_{\eta,1/N}(z) dz }. \]
Furthermore, a single step of the Glauber dynamics can be $\epsilon$-approximately implemented using $O(\log(1/\epsilon)(\log n + \log \log (1+B))$ arithmetic operations, where 
$B = \max_i \sum_j |J_{ij}| + \|h_i\|$.
In other words, we can sample from a distribution that is $\epsilon$-close in total variation distance to the conditional distribution at the spin chosen to be updated. Therefore, to sample from a distribution $\varepsilon$-close to $\nu$ in total variation distance, we run the Glauber dynamics for $T=n \log(n/\varepsilon)  $ steps and set $\epsilon =\frac{\varepsilon}{T }$. The total runtime is $ O(n \log(n/\varepsilon) (\log N + \log \log (1+B)) ).$
\end{theorem}
\begin{proof}
ATE follows from \cref{thm:fixed-semilogconcave} given the semi-log-concavity bound previously established.
What remains is to discuss the implementation of Glauber dynamics.

Each step of the Glauber dynamics sample the spin $ x_i\in S^{N-1}$ conditional on the remaining spins $x_j$ being set at $u_j.$  The density $\pi(x_i)$ at $x_i$ is $\exp( \langle h_i+ \sum J_{ij} u_j, x_i\rangle ).$ By applying a suitable rotation matrix, we can assume without loss of generality that 
\[ h_i+ \sum J_{ij} u_j = b e_1 \]
where $b\in \R_{\geq 0}$ and $e_1 $ is the basis vector with the first coordinate being $1$ and the rest of the coordinates being $0$. By the triangle inequality, we have $b \le B$ where $B$ is as in the theorem statement. 

If $b=0$ then $x_i$ is uniformly distributed on the sphere $S^{N-1}$, thus we can sample $x_i$ in $O(n)$ time. Below, suppose $b> 0.$
If $N=1$ then $\pi$ is a Bernoulli distribution, so sampling can be done in $O(1)$ time. (This is the case of the Ising model.) Below, we assume $n\geq 2.$

Let $ y = \langle x_i, e_1 \rangle ,$ then $y\in [-1,1].$ Note that if we can sample $y$ exactly, then to sample $x_i,$ we only need to uniformly sample from the set $\set{u\in S^{N-1} \langle u, e_1\rangle = y}$, which is isomorphic to the sphere $S^{n-2}$, and this can be done in $O(N)$ time.

By \cref{fact:sphere first coord},   the density of $ y$ induced by $\pi$ is 
\[ \pi(y) \propto \rho(y) := \exp(b y) (1-y^2)^{(N-3)/2}. \]

\paragraph{Case $N = 2$.}
If $N = 2$, then we can do a change of variable by writing $y= \cos(\theta)$ for $\theta \in [-\pi, \pi).$ Then $\rho$ induces a density $\varphi$ supported on $[-\pi, \pi)$ defined by $\varphi(\theta) = \exp(b\cos(\theta) ).$ We consider the restriction $\tilde{\varphi}$ of $\varphi$ to $[-\pi/4, \pi/4),$ which satisfies the preconditions of \cref{lem:log concave sample 1 dim} with $V(\theta) = b(1-\cos(\theta)) $ and $V''(\theta)= b \cos (\theta)\in [b/\sqrt{2}, b] $ for $\theta \in [-\pi/4,\pi/4).$ 
Next, we sample from $\varphi$ by rejection sampling from $ \tilde{\varphi}:$
\begin{itemize}
    \item Sample $u$ uniformly from $\set{0,1,2,3}$
    \item Sample $\tilde{\theta}  \in [-\pi/4,\pi/4)$ from $\tilde{\varphi}$ then output $\theta =\tilde{\theta} +u\pi/2$ with probability $\exp(b(\cos(\theta) - \cos(\tilde{\theta}))\leq 1 .$ 
\end{itemize}
This rejection sampling algorithm succeeds with probability $ \frac{1}{4} \sum_{u=0}^3\int_{-\pi/4}^{\pi/4} \exp(b(\cos(\tilde{\theta} + u\pi/2) - \cos(\tilde{\theta})) d\tilde{\theta}\geq 1/4.$ Combined this with \cref{lem:log concave sample 1 dim}, we have an algorithm that runs in $O(\log(1/\epsilon))$ time and output a sample from $ \hat{\varphi}$ s.t. $d_{TV} (\hat{\varphi}, \varphi)\leq\epsilon.$

\paragraph{Case $N = 3$.} If $N = 3$, then the cumulative distribution of $\rho$ can be exactly computed: 
\[F(y) = \int_{-1}^y \rho (u) du = b^{-1} (\exp(by) - \exp(-b))\] 
thus one can sample from $\rho$ in $O(1)$ time by inverse-CDF sampling. Explicitly, we uniformly sample $r$ from $[0, F(1)]$ and then output the unique $y$ s.t. $r=F(y).$

\paragraph{Remaining cases ($N \ge 4$).}
We now deal with the four and higher-dimensional cases. 
Note that $\rho(y) = \exp(-V(y))$ with 
\begin{align*}
V(y) &= -by - \frac{N-3}{2} \log(1-y^2)
\end{align*}
and observe that
\begin{align*}
\nabla V(y) = - b + (N-3) \frac{y}{1-y^2}, \qquad
\nabla^2 V(y) =  \frac{(N-3)(1+y^2)}{(1-y^2)^2} \geq (N-3) > 0.
\end{align*}
Let $c =  \frac{N-3}{b} .$ Observe that the unique solution to $ \nabla V(y) =0$ in $ [-1,1]$ is $ y_0 = \frac{ \sqrt{c^2 + 4}-c}{2}\in (0,1].$ 
Let 
\[ \tilde{V}(y) = V(\frac{y}{\sqrt{N-3}} +y_0) - V(y_0) \] 
and observe that $\tilde{V}$ is supported on $[\sqrt{N-3} (-1-y_0) , \sqrt{N-3} (1-y_0)], $  $\tilde{V}(0)=\tilde{V}'(0)=0$  and $\tilde{V}''(0) \geq 1.$ Note that if $ y^2 \geq 1 - \exp(-(\frac{1}{2} +b+V(y_0)) )  =1 - (1-y_0^2)^{ b (1-y_0) +  1/2} $ then
\[V(y) - V(y_0) \geq -\log(1-y^2) - b - V(y_0) \geq \frac{1}{2}\]
and if $ y^2 \leq 1 - \exp(-(\frac{1}{2} +b+V(y_0)) )$ then
$V''(y) \leq 2(N-3) \exp(2 (\frac{1}{2} +b+V(y_0)) ) .$

Thus, \cref{lem:log concave sample 1 dim} applies for $\tilde{V}$ with $a = \sqrt{N-3} (-1-y_0), b = \sqrt{N-3} (1-y_0),$ and 
\begin{align*}
\frac{a'}{\sqrt{N-3}} 
&=-\sqrt{1- (1-y_0^2)^{ b (1-y_0) +  1/2} }-y_0 , \\
\frac{b'}{\sqrt{N-3}} 
&=\sqrt{1- (1-y_0^2)^{ b (1-y_0) +  1/2} }-y_0, \\
\kappa &= 2\exp(2 (\frac{1}{2} +b+V(y_0)) ) .
\end{align*}
Note that  
\[ 0\leq 2b(1-y_0) = 2(b - \frac{\sqrt{(N-3)^2 + 4b^2} - (N-3) }{2}) \leq (N-3) \] and 
\[ \frac{1}{1-y_0^2} = \frac{\sqrt{c^2 + 4}}{2c}+ \frac{1}{2} \leq \frac{1}{c} + 1 = \frac{b}{N-3} +1, \]
thus
\begin{align*}
  \log\log (\max \set{\abs{a}, \abs{b}} \kappa )
  & \leq \log \log (N \exp(1 + b + V(y_0)))\\ 
  &= \log\log \left(N \left(\frac{1}{1-y_0^2}\right)^{2b (1-y_0)+1}\right)\\
  &=\log\left(\log N + N \log\frac{1}{1-y_0^2}\right)
  = O(\log N + \log\log (1+b))  
\end{align*}
where the second line follows from the definition of $V.$
\end{proof}

\begin{fact}[Well-known, see e.g.\ \cite{zhang2011uniform}]\label{fact:sphere first coord}
Let $p(x_1)$ be the probability density of the first coordinate $X_1$ of a random vector $X$ sampled uniformly at random from the $N - 1$ dimensional unit sphere $S^{N-1}$ in $\mathbb R^N$ centered at the origin. Then $p$ is supported on $[-1,1]$ with density 
\[ p(x_1) \propto (1 - x_1^2)^{(N - 3)/2}. \]
\end{fact}

In the implementation of each Glauber dynamics step, we need to sample from a distribution supported on the sphere $S^{n-1},$ which can be reduced to sampling from a log-concave distribution supported on a real interval. The following shows how to do the latter task using a modification of the technique in \cite{Chewi2021TheQC}. 
While \cite{Chewi2021TheQC} requires the potential function(s) to be supported on the entire real line and to have uniformly bounded second derivative, our result only requires bounded second derivative in a large enough interval of the support.
\begin{lemma}\label{lem:log concave sample 1 dim}
Consider $V : [a,b] \to \R\cup \set{\pm \infty}$ with $0\in [a,b]. $
    Suppose $V(0) =\nabla V(0) = 0$ and $ V''(x)\geq \alpha> 0$ for $x\in (a,b).$ There exists $a',b'$ s.t. $a\leq a'\leq 0 \leq b'\leq b$ s.t. $ V(x) \geq 1/2 $ if $x\in [a,a'] \cup [b',b]$ and $ V''(x) \leq \beta $ if $x\in [a',b'].$ Let $\pi$ be a distribution $[a,b]$ s.t. $p(x) \propto\exp(-V(x)).$ For $\kappa = \beta/\alpha$, we have that: 
    \begin{itemize}
        \item There exists an algorithm which runs in $O(\log\log(\max \set{\abs{a}, \abs{b}} \kappa) )$ time and with probability $\Omega(1)$, outputs a sample from $ p.$ 
        \item Given any $\epsilon > 0 $, there exists an algorithm which samples from $\hat{p}$ s.t. $d_{TV}(p, \hat{p})\leq \epsilon$ in $O(\log (\frac{1}{\epsilon}) \cdot \log\log(\max \set{\abs{a}, \abs{b}} \kappa) ) $ time. 
    \end{itemize}
    
\end{lemma}
\begin{proof}
 Without loss of generality, we can assume $\alpha =1$ so that $V''(x) \geq 1\forall x\in [a,b]$ and  $ V''(x)\leq \kappa$ for $x\in [a',b'].$ Otherwise, we can replace $V$ with $\tilde{V}$ where $ \tilde{V}(x) = V(x/\sqrt{\alpha}).$ 

We use a rejection sampling algorithm similar to in \cite{Chewi2021TheQC}. To account for the fact that $V"(x)$ is only bounded above by $\kappa$ in the interval $ [a',b'],$ we only need to slightly modify the proof so that  $x_-, x_+\in [a',b'].$

First, we build the envelope $q$, which is a distribution over $[a,b]$, then we draw a sample from $q$ and accept with probability $\frac{p(x)}{q(x)}.$ We only need to show that the acceptance probability of this algorithm is $\Omega (1).$ For the second claim, we repeat this algorithm $O(\log (1/\epsilon))$ times and output the result of the first successful run. If all runs fail then output a uniform sample from $[a,b].$ We define $q$ as follows.
\begin{enumerate}
    \item Find the first index $i_+\in \set{0,1, \dots,\lfloor   \log_2 (b'\sqrt{\kappa}) \rfloor } $ s.t. $V(2^{i_+}/\sqrt{\kappa}) \geq 1/2.$ If $i_+$ exists, let $x_+ =2^i_+/\sqrt{\kappa}$ else let $x_+ = b'$
     \item Find the first index $i_-\in \set{0,1, \dots, \lceil \log_2 (-a'\sqrt{\kappa}) \rceil } $ s.t. $V(-2^{i_-}/\sqrt{\kappa}) \geq 1/2.$ If $i_-$ exists, let $x_- =-2^i_-/\sqrt{\kappa}$ else let $x_- = a'$
     \item  
     Let $q(x)\propto\tilde{q}(x)$ where
     \[\tilde{q}(x) = \begin{cases} \exp(-\frac{x-x_-}{2x_-} - \frac{(x-x_-)^2}{2}) 
 &\text{if } x \in [a, x_-] \\  1 &\text{if } x \in [x_-,x_+]\\  \exp(-\frac{x-x_+}{2x_+} - \frac{(x-x_+)^2}{2}) 
 &\text{if } x \in [x_+, b] \\\end{cases}\]
\end{enumerate}
    First, observe that $ V(x) \geq x^2/2$ for all $x \in \mathbb R$ and $ V (x) \leq \kappa x^2/2$ for all $x\in [a',b']$ thus $1/2 \leq V(b')\leq \kappa (b')^2/2$ and $1/2 \leq V(a') \leq \kappa (a')^2/2.$ Thus $ \min \set{\log (b' \sqrt{\kappa}) , \log (-a'\sqrt{\kappa}) } \geq  0$ and the first two steps are well-defined. Sampling from $q$ can be done in $O(1)$ time, since cumulative distribution of $q$ at $x$ is $\P_{Z\sim \mathcal{N}(0,1)}{u<Z<u'}$ where $u,u'$ can be computed in $O(1)$ time given $ x_-, x_+, a,b, x.$

Let $\tilde{p}(x)=\exp(-V(x)).$ 
    We only need to check that 
    $\tilde{p}(x)\leq \tilde{q}(x) $ and that there exists absolute constant $C > 0$ s.t. 
    \[ C  Z_{\tilde{p}} =\int_a^b \exp(-V(x)) dx \geq Z_{\tilde{q}}=\int_a^b \tilde{q}(x) dx. \]
    To prove this inequality, by using the argument in \cite{Chewi2021TheQC},
    we only need to check that $x_+ = O\left( \int_0^{x_+} \tilde{p}(x) dx\right)$ and $-x_- = O\left(\int_{x_-}^0 \tilde{p}(x) dx\right).$ We will show the first inequality; the second is entirely analogous. If $i_+$ doesn't exist, then since $V$ is increasing in $[0, b],$ \[V(x_+/2) = V(b'/2) \leq V(2^{\lfloor   \log_2 (b'\sqrt{\kappa}) \rfloor}/\kappa)< 1/2.  \]
If $ i_+$ exists and $i_+> 0$ then by definition of $i_+,$ $V(x_+/2) < 1/2.$ In both cases,
\[ \int_0^{x_+} \tilde{p}(x) d x \geq \int_0^{x_+/2}\exp(-1/2) dx = \Omega(x_+). \]
The remaining case is $i_+ = 0.$ Note that $ x_+ =1/\sqrt{\kappa} \in (0,b']$ thus
\[\int_0^{x_+} \tilde{p}(x) dx \geq \int_0^{1/\kappa } \exp(-\kappa x^2/2) dx \geq \frac{1}{3\sqrt{\kappa} } =\frac{x_+}{3} .\]

The rest of the proof follows from \cite[Proof of Theorem 3]{Chewi2021TheQC}, since the upper bound $V''(x)\leq \kappa$ is only used for proving the above claims. The rest of the proof in \cite{Chewi2021TheQC} only uses $V''(x) \geq 1,$ which does hold for the entire domain of $V.$ 
    
\end{proof}}
    \Tag{\section{Sampling Ising models with the polarized walk}\label{sec:expander}
In this section, we establish guarantees for sampling antiferromagnetic Ising models on spectral expanders and related results. Major differences from the previous sections is the use of the tools from geometry polynomials and the closely related \emph{polarized walk} rather than the Glauber dynamics.
\subsection{Estimates via geometry of polynomials}
    
\begin{lemma}\label{lem:half-flc}
Suppose that $\nu$ is a measure on the hypercube $\set{\pm 1}^n$ of the form
\[ \nu(x) \propto \exp\parens*{-\frac{\gamma}{2n} \parens*{\sum_i x_i}^2 } \]
for any $\gamma \ge 0$.
The homogenized generating polynomial of $\nu$ 
\[f(z_1,\dots, z_n,\bar{z}_1, \dots, \bar{z}_n) = \sum_x \nu(x) \prod_{i: x_i = 1} z_i \prod_{i: x_i =-1} \bar{z}_i\]
is $1/2$-fractionally log concave. 
\end{lemma}
\begin{remark}
Recall that the corresponding homogenized distribution $\nu^{\hom}$ is a version of $\nu$ supported on $\binom{[n]\cup [\bar{n}]}{n}.$ 
We will not need this fact, but it is worthwhile to note that
    \cref{lem:half-flc} implies the $2$-step down-up walk on $\nu^{\hom}$ mixes in $\widetilde{O}(n^2)$ steps and entropy contraction of the down operator $D_{n\to (n-2)}$ w.r.t.\ $\nu^{\hom}$, i.e., for all $\pi:\binom{[n]\cup [\bar{n}]}{n} \to \R_{\geq 0} $
    \[\DKL {\pi D_{n\to (n-2)} \river \nu D_{n \to (n-2)} }\leq \left(1- \frac{2}{n (n-1)} \right) \DKL{\pi \river \nu}.\]

    $D_{n\to (n-2)}$ is the Markov kernel that maps $ S\in \binom{[n]\cup [\bar{n}]}{n}$ to a uniformly random subset of $S $ of size $n-2.$
    The $2$-step down-up walk 
    corresponds to the $2$-block Glauber dynamics for $\nu$, which operates by choosing uniformly random $\set{i,j}\in \binom{[n]}{2},$ then resampling the spin at $i,j$ conditioned on the remaining assignments. 
\end{remark}
\begin{remark}
This result is sharp. 
Consider $n =2, h = \vec{0}.$ As $\gamma \to \infty,$ $\nu$ converges to the uniform distribution over $\begin{bmatrix}  1 \\ -1 \end{bmatrix}$ and $\begin{bmatrix}  -1 \\ 1 \end{bmatrix}.$ The homogenized generating polynomial converges to  $ z_1 \bar{z}_2 + z_2 \bar{z}_1,$ which is  $1/2$-fractionally log-concave; that it is exactly $1/2$-spectrally independent can be checked by explicitly computing its influence matrix. 
\end{remark}
\begin{remark}
It is interesting to compare \cref{lem:half-flc} to the NP Hardness result from Appendix H of \cite{koehler2022sampling}.
There they show that, replacing $ (\sum_i x_i)^2 = \dotprod{\vec{1}, x}^2$ with $\dotprod{a, x}^2$ in the definition of $\nu(x)$, for an arbitrary vector $a\in \Z^n$, results in a hard problem.
More precisely, for $a$ coming from the number partitioning problem, no polynomial time algorithm can approximately sample from such a model within a TV distance of $1/2$ unless RP = NP. 
Thus, \cref{lem:half-flc} cannot be recovered from a generic fact about negatively-spiked rank-one Ising models --- its validity has to do with the arithmetic structure of the interactions.
\end{remark}

One key to the proof of the result will be some fundamental facts relating the spectral independence and fractional log-concavity of different variants of a probability measure $\mu$, which we establish now. 
\begin{lemma}\label{lem:complement vs homogenization}
Let $\mu: 2^{[n]}\to \R_{\geq 0} $ be a probability distribution and $\mu^{\text{comp}}:2^{[n]} \to \R_{\geq 0}$ be the complement of $\mu$, i.e., $\mu^{\text{com}} ([n] \setminus S) = \mu(S).$ Let $\mu^{\hom}:\binom{[n]\cup [\bar{n}]}{n} \to \R_{\geq 0}$ be the homogenization of $\mu$.
Then:
\begin{enumerate}
    \item If $\mu$ and $\mu^{\text{comp}}$ are $\frac{1}{\alpha}$-spectrally independent then $\mu^{\hom}$ is $\frac{2}{\alpha}$-spectrally independent.
    \item If $\mu$ and $\mu^{\text{comp}} $ are $\alpha$-FLC then $\mu^{\hom}$ is $\alpha/2$-FLC.
\end{enumerate}
\end{lemma}
\begin{proof}
 For a probability distribution $\rho: 2^{X}\to \R_{\geq 0}, $ let $ \text{mean}(\rho)\in \R_{\geq 0}^X$ be the vector with $\text{mean}(\rho)_i = \P_{\rho}{i}.$

 We consider $\mu, \mu^{\text{com}}, \mu^{\hom}$ as distributions over $2^{[n]}, 2^{[\bar{n}]}, 2^{[n] \cup[\bar{n}]}$ respectively.
Thus $ \cov{\mu}$ is indexed by $[n],$ $\cov{\mu^{\text{com}}}$ is indexed by $[\bar{n}]$ and $ \cov{\mu^{\hom}}$ is indexed by $[n] \cup [\bar{n}].$

Because being FLC is equivalent to being spectrally independent under arbitrary tilts, the second claim follows from the first. 
Indeed, for $\lambda \in \R^{[n]\cup [\bar{n}]},$ $\lambda \ast \mu^{\hom} = (\tilde{\lambda}\ast \mu)^{\hom}$ with $\tilde{\lambda_{i}} = \lambda_i/\lambda_{\bar{i}}$ and $(\tilde{\lambda}\ast \mu)^{\text{comp}} = \lambda' \ast \mu^{\text{comp}}$ with $\lambda'_{\bar{i}} = \frac{1}{\tilde{\lambda}_i}.$  By the assumption on $\mu$ and $\mu^{\text{comp}},$ $\tilde{\lambda}\ast \mu $
 and $(\tilde{\lambda}\ast \mu)$ are $\frac{1}{\alpha}$-spectrally independent thus the first claim implies $ \lambda \ast \mu^{\hom} = (\tilde{\lambda}\ast \mu)^{\hom}$ is $\frac{2}{\alpha}$-spectrally independent. Since this is true for all $\lambda,$ $\mu^{\hom}$ is $\alpha/2$-fractionally log-concave.

We now proceed to prove the claim about spectral independence. 
First, note that $\cov{\mu} = \cov{\mu^{\text{comp}}}$ and $\text{mean}(\mu^{\hom}) = (\text{mean}(\mu), \text{mean}(\mu^{\text{comp}}))$. Furthermore, we claim that 
\[ \cov{{\mu^{\hom}}}= \begin{bmatrix}A & -A \\ -A & A\end{bmatrix} \]
with $A = \cov{\mu}$. 
To see this, observe that $ \P_{{\mu^{\hom}}}{i} = \P_{\mu}{i}$, $ \P_{{\mu^{\hom}}}{\bar{i}} = \P_{\mu^{\text{comp}}} {\bar{i}} = 1- \P_{\mu}{i}$, and 
\begin{align*}
  (\cov{{\mu^{\hom}}})_{\bar{i}, \bar{j}} &= \P_{{\mu^{\hom}}} {\bar{i}, \bar{j}} - \P_{{\mu^{\hom}}}{\bar{i}} \P_{{\mu^{\hom}}} {\bar{j}} \\
  &=\P_{S\sim \mu}{i, j\not\in S} -\P_{S\sim \mu}{i\not\in S}\P_{S\sim \mu}{j\not \in S}   \\
  &= \P_{S\sim \mu} {i\not\in S} - \P_{S\sim \mu} {i\not\in S, j\in S} -   \P_{S\sim \mu}{i\not\in S} (1- \P_{S\sim \mu}{j \in S})\\
  &= -(\P_{S\sim \mu}{i \not\in S,j\in S} -\P_{S\sim \mu}{i\not\in S}\P_{S\sim \mu}{j\in S}\\
  &=-(\cov{{\mu^{\hom}}})_{\bar{i}, j},
\end{align*}
and furthermore
\begin{align*}
-(\cov{{\mu^{\hom}}})_{\bar{i}, j}
  &= -(\P_{S\sim \mu}{j\in S}- \P_{S\sim \mu}{i\in S, j\in S})+ (1- \P_{S\sim \mu}{i\in S}) \P_{S\sim \mu}{j \in S}\\
  &=\P_{S\sim \mu}{i\in S, j\in S} - \P_{S\sim \mu}{i\in S}) \P_{S\sim \mu}{j \in S}\\
  &= (\cov{{\mu^{\hom}}})_{i, j}.
\end{align*}
Next, for any vector $\vec{x}\in \R^{[n]\cup [\bar{n}]},$ we can write $\vec{x} = \vec{x}_0 || \vec{x}_1$ with $\vec{x}_0\in \R^{[n]}$ and $ \vec{x}_1 \in \R^{[\bar{n}]}.$ Then
\begin{align*}
    \MoveEqLeft \vec{x}^\intercal \left(\frac{2}{\alpha} \diag(\text{mean}({\mu^{\hom}})) - \cov{{\mu^{\hom}}}  \right)\vec{x} \\
    &= \frac{2}{\alpha}(\vec{x}_0^\intercal  \diag(\text{mean}(\mu)) \vec{x}_0 +  \vec{x}_1^\intercal  \diag(\text{mean}(\mu^{\text{comp}})) \vec{x}_1) - (\vec{x}_0 - \vec{x}_1)^\intercal A (\vec{x}_0 -\vec{x}_1)\\
    &= 2 \vec{x}_0^\intercal  \left(\frac{1}{\alpha}\diag(\text{mean}(\mu) - A\right) \vec{x}_0 +  \vec{x}_1^\intercal  \left(\frac{1}{\alpha}\diag(\text{mean}(\mu^{\text{comp}}))-A\right) \vec{x}_1  \\
    &\quad + 2 (\vec{x}_0^\intercal A \vec{x}_0 +\vec{x}_1^\intercal A \vec{x}_1) - (\vec{x}_0 - \vec{x}_1)^\intercal A (\vec{x}_0 - \vec{x}_1)\\
    &= 2 \vec{x}_0^\intercal  \left(\frac{1}{\alpha}\diag(\text{mean}(\mu)) - A\right) \vec{x}_0 +  \vec{x}_1^\intercal  \left(\frac{1}{\alpha}\diag(\text{mean}(\mu^{\text{comp}}))-A\right) \vec{x}_1  + (\vec{x}_0 +\vec{x}_1)^\intercal A (\vec{x}_0 + \vec{x}_1)\geq 0
\end{align*}
where the last inequality follows from the following inequalities: 
\begin{align*} 
\frac{1}{\alpha}\diag(\text{mean}(\mu)) &\succeq \cov{\mu} =A,
\\
\frac{1}{\alpha}\diag(\text{mean}(\mu^{\text{comp} } )) &\succeq \text{Cov}(\mu^{\text{comp}}) =A, \\
A= \cov{\mu} =\E_{\mu}{(x-\text{mean}(\mu)) (x -\text{mean}(\mu))^\intercal }&\succeq 0. 
\end{align*}
Since the above inequality is true for all $\vec{x},$ we conclude that $\frac{2}{\alpha} \diag(\text{mean}({\mu^{\hom}})) - \cov{{\mu^{\hom}}}  \succeq 0$ and $ \mu^{\hom}$ is $\frac{\alpha}{2}$-FLC.
\end{proof}

\begin{lemma}\label{lem:polarize}
Let $\mu : 2^{[n]} \to \R_{\ge 0}$ be a probability distribution. Then:
\begin{enumerate}
    \item If $\Pi(\mu)$ is $\frac{1}{\alpha}$-spectrally independent then so is $\mu.$ 
    \item If $ \Pi(\mu)$ is $\alpha$-fractionally log-concave then so is $\mu.$
\end{enumerate}
\end{lemma}
\begin{proof}
To prove the second claim, we need to show $\lambda \ast \mu$ is $\frac{1}{\alpha}$-spectrally independent for all $\lambda \in \R_{\geq 0}^n.$ Note that $ \Pi(\lambda \ast \mu) = (\lambda || \textbf{1}) \ast \Pi(\mu)$ is $\frac{1}{\alpha}$-spectrally independent as a scaling of $\Pi(\mu)$, so the second conclusion follows from the first one. We now prove the first claim about spectral independence. 

We examine the covariance matrix of $ \Pi(\mu).$ We show that for $i, j\in X,$
\[(\cov{\Pi(\mu)})_{i,j} = (\cov{\mu})_{i,j} \text{\quad and\quad} \text{mean}(\Pi(\mu))_i = \text{mean}(\mu)_i\]
Indeed, 
\[\text{mean}(\Pi(\mu))_i =\sum_{S\subseteq X: i\in S, T \cup Y: \abs{T} = n-\abs{S}} \mu(S \sqcup T) =\sum_{S\subseteq X: i\in S, T \cup Y: \abs{T} = n-\abs{S}}  \mu(S) \frac{1}{\binom{n}{\card{S}}} = \sum_{S\subseteq X: i\in S} \mu(S) = \text{mean}(\mu)_i  \]
Similarly, 
\begin{align*}
    \P_{\Pi(\mu)}{i,j} &= \sum_{S\subseteq X: i, j\in S, T \cup Y: \abs{T} = n-\abs{S}} \mu(S \sqcup T)  = \sum_{S\subseteq X: i, j\in S} \mu(S) = \P_{\mu}{i,j}
\end{align*}
and
\begin{align*}
    (\cov{\Pi(\mu)})_{i,j} &= \P_{\Pi(\mu)}{i,j} -\P_{\Pi(\mu)}{i}\P_{\Pi(\mu)}{j} = \P_{\mu}{i,j}- \P_{\mu}{i}\P_{\mu}{j} = (\cov{\mu})_{i,j}.
\end{align*}
$1/\alpha$-spectral independence of $\Pi(\mu)$ means that $0\preceq (\frac{1}{\alpha}\diag(\text{mean}(\Pi(\mu))) - \cov{\Pi(\mu)})$ thus 
\[ 0\preceq \parens*{\frac{1}{\alpha}\diag(\text{mean}(\Pi(\mu))) - \cov{\Pi(\mu)}}_{X,X} = \frac{1}{\alpha}\diag(\text{mean} (\mu)) - \cov{\mu}. \]
\end{proof}
\begin{proof}[Proof of \cref{lem:half-flc}]
Let $\mu: 2^{[n]}\to \R_{\geq 0}$ be defined by $\mu(S) = \nu(\chi_S)$ where
\[ (\chi_S)_i =\begin{cases} 1 &\forall i\in S  \\-1 &\text{ else}\end{cases}. \] Note that $\nu = \mu^{\hom}.$
Let $\alpha =1.$
    By \cref{lem:bivariate,lem:polarize}, $\mu$ is ${\alpha}$-fractionally log concave. Since the generating polynomial of $\mu^{\text{comp}}$ is again $\Pi_g,$ $ \mu^{\text{comp}}$ is $\alpha$-fractionally log concave. Thus by \cref{lem:complement vs homogenization}, $ \nu=\mu^{\hom}$ is $\frac{\alpha}{2}$-fractionally log-concave.
\end{proof}

\begin{lemma}\label{lem:bivariate}
For any $n \ge 1$ and $c \ge 0$ the bivariate polynomial
\[ g(x,y) = \sum_{k = 0}^n \binom{n}{k} e^{-c(k - n/2)^2} x^k y^{n - k}  \]
is strongly log-concave. Furthermore, its polarization
\[ \Pi_n g(x_1,\ldots,x_n,y_1,\ldots,y_n) = \sum_{k = 0}^n \sum_{S\in\binom{[n]}{k}, T\in \binom{[n]}{n-k}} e^{-c (k - n/2)^2/2} x_S y_{T}\frac{1}{\binom{n}{n-k}} ,  \]
which is a multilinear polynomial,
is strongly log-concave. 
\end{lemma}
\begin{proof}
The first conclusion is is true because the sequence $(e^{-c(k - n/2)^2})_{k = 0}^n$ is log-concave, i.e.
\[ e^{-c(k - n/2)^2} \ge \sqrt{e^{-c(k- 1 - n/2)^2} e^{-c(k + 1 - n/2)^2}} \]
which follows from the the more general fact that the function $e^{-cx^2}$ on the real-line is log-concave. Given this, 
strong log-concavity of the bivariate polynomial follows from \cref{prop:slc-coeff}. Strong log-concavity of the multilinear polynomial follows from polarization (\cref{prop:polarization}). 
\end{proof}
\subsection{Trickle-down}
\begin{lemma}\label{lem:trickle down d regular}
Suppose that $\nu$ is a measure on the hypercube $\set{\pm 1}^n$ of the form
\[ \nu(x) \propto \exp\parens*{\frac{1}{2} \dotprod{ x, J x } + \dotprod{ h, x } - \frac{\gamma}{2n}\parens*{\sum_i x_i}^2} \]
for some $J \succeq 0,\gamma \ge 0$ and $h \in \R^n$. 
If $\opnorm{J} < 1/2$ then
\[ \opnorm{\cov{\nu}} \le \frac{2}{1 - 2\opnorm{J}}. \]
Moreover, if we define $\nu_t$ to be the stochastic localization process with driving matrix $C_t = J^{1/2}$, then
\[\opnorm{\cov{\nu_t}}\le \frac{2}{ 1- 2(1-t) \opnorm{J} }\]
\end{lemma}
\begin{proof} 
By \cref{thm:gci} (or by explicit calculation), we have that 
\[ \nu_1(x) \propto \exp\parens*{ \dotprod{ h + H_1, x } - \frac{\gamma}{2n}\parens*{\sum_i x_i}^2}  \]
for some random vector $H_1$. By \cref{lem:half-flc} we know that the measure $\nu_1$ is $1/2$-fractionally log-concave, and by \cref{fact:si-flc} this tells us that the covariance matrix of $\nu_1$ has operator norm at most $2$. Hence by \eqref{eqn:sl-trickledown} and submultiplicativity of the operator norm we have that
\[ \opnorm{\cov{\nu_t}} \le 2 I + \int_t^1 \E{\norm{\cov{\nu_s}}^2 \opnorm{J}}. \]
Observe that the solution to the differential equation $dy/dt = \opnorm{J} y^2$ with initial condition $y(0) = 2$ is $y(t) = \frac{2}{1 - 2 \opnorm{J} t}$. So by performing a comparison argument in the same way as in the proof of \cref{thm:ising}, we obtain the result. 

The second statement follows from the first one and the definition of $\nu_t(x)$, since (by e.g.\ \cref{thm:gci})
\[\nu_t (x) \propto \exp\parens*{ \frac{1}{2}\dotprod{ x , (1-t)  J x} +\dotprod{ h + H_1, x } - \frac{\gamma}{2n}\parens*{\sum_i x_i}^2}  \]
and $\opnorm{(1-t) J} = (1-t) \opnorm{J}$;
\end{proof}
\subsection{Dynamics}
\begin{definition}[polarized operators]\label{def:polarized}
    Suppose $\nu$ is a probability measure on the hypercube $\set{\pm 1}^n.$ Let $\tilde{\nu}: 2^{[n]}\to \R_{\geq 0}$ be defined by $\tilde{\nu}(S) = \nu(\chi_S) $ where $\chi_S \in \set{\pm 1}^n$ is defined so that for each $i \in [n]$,
    \begin{equation} \label{eqn:chi-notation}
    (\chi_S)_i =\begin{cases} 1 &\text{if }i\in S,  \\-1 &\text{otherwise.}\end{cases} 
    \end{equation}
    Consider the polarization $\Pi(\tilde{\nu})$ of $\tilde{\nu}.$ For $x\in \set{\pm 1}^n$ let $x_+ =\set{i:x_i=1}.$ The down operator $D_{n\to (n-1)}$ on $\Pi(\nu) $ corresponds to the following Markov kernel $D^{\text{pol}}$ mapping\footnote{We can equivalently view it as a Markov operator on $\{\pm 1\}^n$ if we identify a vector $x \in \{\pm 1\}^n$ and the set $x_+$.} $x \in \set{\pm 1}^n$ to $T \in 2^{[n]}$ with the following transition probabilities: 
\begin{equation}
   D^{\text{pol}} (x, T ) = \begin{cases} \frac{1}{n} &\text{if } T = x_+\setminus \set{i} \text{ for some } i\in x_+,  \\ 
   \frac{n-\abs{x_+}}{n}   &\text{if } T = x_+, \\
   0 & \text{otherwise.}
\end{cases}
\end{equation}
The polarized up-operator, define with respect to $\nu,$ maps $T \in 2^{[n]}$ to $x \in \set{\pm 1}^n$ with transition probabilities as follows:
\begin{equation}
   U^{\text{pol}} ( T , x ) = \begin{cases} \frac{1}{n} \times \frac{\nu(x)  }{ \nu D^{\text{pol}} (T) } &\text{if } T = x_+\setminus \set{i} \text{ for some } i\in x_+,  \\ 
   \frac{n-\abs{T}}{n} \times \frac{\nu(x)  }{ \nu D^{\text{pol}} (T) }  &\text{if } T = x_+, \\
   0 & \text{otherwise.} 
\end{cases}
\end{equation}
The polarized walk on $\nu$ is the Markov operator $D^{\text{pol}} U^{\text{pol}}.$ Explicitly, its transitions are as follows for $x,x' \in \set{\pm 1}^n$:
\begin{equation}
    (D^{\text{pol}} U^{\text{pol}}) ( x, x' ) =\begin{cases} 
    \frac{n-\abs{x_+}}{n^2} \times \frac{\nu(x')  }{ \nu D^{\text{pol}} (x_+) } &\text{if } x'_+ =  x_+\cup \set{i}, i\not\in x_+, \\
     \frac{n - \abs{x'_+}}{n^2}\times \frac{\nu(x')  }{ \nu D^{\text{pol}} (x'_+) } &\text{if } x'_+=x_+\setminus \set{i}, i\in x_+, \\
    \frac{1}{n^2} \times \frac{\nu(x') }{\nu D^{\text{pol}} (x_+\cap x'_+) }&\text{if } x'_+=x_+\setminus \set{i}\cup \set{j}, i\in x_+, j\not\in x_+, \\
    \parens*{\frac{n-\abs{x_+}}{n}}^2 \times  \frac{\nu(x)  }{ \nu D^{\text{pol}} (T) }  &\text{if } x' = x, \\
    0 &\text{otherwise.}
    \end{cases}
\end{equation}

 
\end{definition}

\begin{proposition}\label{prop:entropy contraction at T=1}
    Suppose that $\nu$ is a measure on the hypercube $\set{\pm 1}^n$ of the form
\[ \nu(x) \propto \exp\parens*{-\frac{\gamma}{2n} \parens*{\sum_i x_i}^2 }. \]
The polarized down operator $D^{\text{pol}}$ has $(1-1/n)$-entropy contraction with respect to $\nu$.
\end{proposition}
\begin{proof}
    Recall that by \cref{prop:slc-coeff} the polarization $\Pi(\nu)$ is log-concave, thus $ D_{n\to (n-1)}$ has $(1-\frac{1}{n})$-entropy contraction wrt $\Pi(\tilde{\nu}).$ For any probability distribution $\nu'$ on the hypercube $\{\pm 1\}^n$, we therefore have that
    \begin{align*}
      \DKL {\nu'D^{\text{pol}}\river\nu D^{\text{pol}} } &= \DKL {\Pi(\nu') D_{n\to (n-1)}\river \Pi(\nu) D_{n\to (n-1)}} \\
      &\leq\parens*{1- \frac{1}{n}}\DKL{\Pi(\nu')\river \Pi(\nu)} \\
      &= \parens*{1- \frac{1}{n}}\DKL {\nu'\river \nu}.  
    \end{align*}
\end{proof}

\begin{theorem} \label{thm:ising projected downup}
Suppose that $\nu$ is a measure on the hypercube $\set{\pm 1}^n$ of the form
\[ \nu(x) \propto \exp\parens*{\frac{1}{2} \dotprod{ x, J x } + \dotprod{ h, x } - \frac{\gamma}{2n}\parens*{\sum_i x_i}^2} \]
for some $J \succeq 0$, $\gamma \ge 0$, and $h \in \R^n$, and
suppose that $\opnorm{J} < 1/2.$ Then:
\begin{enumerate}
    \item The polarized down-operator $D^{\text{pol}} $ has $\left(1-\frac{1-2\opnorm{J}}{n}\right)$-entropy contraction with respect to $\nu.$ 
    \item The polarized walk on $ \nu$ mixes within $\epsilon$-TV distance in $O(\frac{1}{1-2\opnorm{J}} n\log(n/\epsilon))$ steps. 
    \item Let 
\[ Q = J - \frac{\gamma}{n} \vec{1} \vec{1}^T \] 
and let $d$ be the maximum number of non-diagonal nonzero entries in a row of $ Q.$ Each step of the polarized walk can be implemented in $ O(d\log n)$ time, so the polarized walk outputs a sample within $\epsilon$-TV distance of $\nu$ in total runtime $O\parens*{\frac{1}{1-2\opnorm{J}} nd\log(n)\log(n/\epsilon)}$
\end{enumerate}
\end{theorem}
\begin{proof}
First we prove the entropy contraction claim.
As in \cref{lem:trickle down d regular}, define $\nu_t$ to be the stochastic localization process with driving matrix $C_t = J^{1/2}$, then 
\[ \nu_1(x) := \nu_1^{(H_1)}(x)\propto \exp\parens*{ \dotprod{ h + H_1, x } - \frac{\gamma}{2n}\parens*{\sum_i x_i}^2}  \]
for some random vector $H_1$. Let $\pi$ be the distribution of $ H_1$, we can write
\[\nu = \int \nu_1^{(H_1)} d\pi(H_1). \]

Fix a test function $\varphi:\set{\pm 1}^n \to\R_{> 0}.$
By \cref{lem:trickle down d regular} and \cref{lem:entropic-stability}
$\nu_t$ is $\alpha_t$-entropically stable with respect to $ \psi(x,y) = \frac{1}{2} \norm{J^{1/2} (x-y)}^2 = \frac{1}{2}\norm{C_t (x-y)}^2$ with $\alpha_t = \frac{2 \opnorm{J}  }{1 - 2(1-t) \opnorm{J} }$ thus
by \cref{prop:conservation-of-entropy},
\[\exp \parens*{-\int_0^1\alpha_t dt } \Ent_{\nu }{\varphi}\leq  \E{\Ent_{x\sim \nu_1}{\varphi(x)}} = \E_{H_1 \sim \pi}{\Ent_{x\sim \nu_1^{(H_1)}}{\varphi(x)} } \]
By \cref{prop:entropy contraction at T=1}, $D^{\text{pol}}$ has $(1-\kappa)$-entropy contraction wrt $\nu_1$ with $\kappa =1/n.$ Using the equivalent definition of entropy contraction in \cref{def:entropy-contraction} with $F = D^{\text{pol}}$ and the supermartingale property (\cref{lem:supermartingale}), we have 
\[  \kappa \E_{H_1 \sim \pi}{\Ent_{x\sim \nu_1^{(H_1)}}{\varphi(x)} } \leq  \E_{H_1 \sim \pi}{  \E_{z \sim \nu_1^{(H_1)} D^{\text{pol}}  }{ \Ent_{X\sim \nu_1^{(H_1)} (\cdot   
 \vartriangleright
 z ) }{\varphi (X) }  }} \leq \E_{z\sim \mu D^{\text{pol}}  } {\Ent_{X\sim \mu (\cdot   
 \vartriangleright
 z ) }{\varphi (X) } } \]
 thus
 \[\frac{1}{n} (1- 2\opnorm{J}) \Ent_{\nu }{\varphi} = \kappa  \exp \parens*{-\int_0^1 \frac{2}{1/\opnorm{J} - 2 (1-t) } dt }\Ent_{\nu }{\varphi} \leq \E_{z\sim \mu D^{\text{pol}}  } {\Ent_{X\sim \mu (\cdot   
 \vartriangleright
 z ) }{\varphi (X) } }. \]
So we have shown that $ D^{\text{pol}} $ has $(1- \frac{ 1- 2\opnorm{J}}{n})$-entropy contraction with respect to $\mu.$

Approximate tensorization of entropy implies a mixing time bound which depends on the probability of the smallest atom in the probability measure. We now eliminate this dependence to state a slightly sharper bound. 

\paragraph{Exchange inequality.} We show $\Pi(\nu)$ satisfy the exchange inequality as defined in \cite[Lemma 23]{ALOVV21}, which implies that the down-up walk on $\Pi(\nu)$ reaches a $\poly(n)$-warm start after $O(n \log n)$ steps and thus the same holds for the polarized walk on $\nu.$ 

The exchange inequality says that for any sets $R, R'$ in the support of $\Pi(\nu)$ and $u\in R \setminus R',$ there exists $v \in R'\setminus R $ such that
\begin{equation}\label{eq:exchange}
  \Pi(\nu) (R)\, \Pi(\nu) (R') \leq 2^{O(n)}\, \Pi(\nu)(R\setminus\set{u} \cup\set{v})\, \Pi(\nu)(R \setminus \set{v} \cup \set{u}).   
\end{equation}
We now show how to prove \cref{eq:exchange}. 
Let $\tilde{\nu}_1$ be a probability measure on $2^{[n]}$ such that $\tilde{\nu}_1(S) = \nu_1(\chi_S)$, where $\chi_S$ is as defined in \eqref{eqn:chi-notation}, and where
\[ \nu_1 (x) \propto \exp\parens*{-\frac{\gamma}{2n} \left(\sum_i x_i\right)^2 + \dotprod{ h , x}}. \]
Recall from \cref{prop:slc-coeff} that $\Pi(\tilde{\nu}_1)$ is log-concave, thus \cref{eq:exchange} is satisfied if we replace $ \nu$ with $\tilde{\nu}_1.$ 
Moreover,  we consider $\Pi(\rho)$ as a probability distribution over $2^{X\sqcup Y} $ where as in \cref{def:polarize} ,$X$ is the set of vertices, that is $[n]$, and $Y$ is the set of dummy variables, also of size $n.$ For $R\subseteq X\cup Y$, $\Pi(\rho)(R)$ is defined by 
\[ \Pi(\rho)(R) = \frac{\rho(R\cap X)}{\binom{X}{\abs{R \cap X} } } 
, \] 
we have that
\[ \Pi(\nu) (R) \propto \Pi(\tilde{\nu}_1) (R) \exp\left(\frac{1}{2}\dotprod{ r , J  r}  \right)  \]
with $r = \chi_{R \cap X}.$ So we only need to check that 
\begin{equation}\label{ineq:exchange 2}
   \exp\left(\frac{1}{2}\dotprod{ r , J  r}\right)  \exp\left(\frac{1}{2}\dotprod{ r' , J  r'} \right) \leq 2^{O(n)} \exp\left(\frac{1}{2}\dotprod{ s , J  s }\right)\exp\left(\frac{1}{2}\dotprod{ s' , J  s' }\right) 
\end{equation}
where we define $r' = \chi_{R' \cap X}$, $s = \chi_{(R\setminus \set{u} \cup \set{v}) \cap X} $, and $s' = \chi_{(R'\setminus \set{v} \cup \set{u}) \cap X}.$ There are two cases:
\begin{itemize}
    \item $s$ ($s'$ resp.) is $r$ ($r'$ resp.) with coordinate $i $ flipped. Then 
    \[\exp\left(\frac{1}{2}\dotprod{ r , J  r} \right) \exp\left(- \frac{1}{2}\dotprod{ s , J  s }\right) = \exp\left(\sum_{j\neq i} J_{ij } r_i r_j\right) \leq \exp \left(\sum_{j\neq i} \abs{J_{ij}}\right) = e^{O(\sqrt{n})} \]
    since $\sum_{j\neq i} \abs{J_{ij}}\leq \norm{J}_{\infty} \leq \sqrt{n} \opnorm{J}$ and $\opnorm{J} < 1/2$ by assumption. 
    \item $s$ ($s'$ resp.) is $r$ ($r'$ resp.) with coordinates $i$ and $j$ flipped.  Then
    \[\exp\left(\frac{1}{2}\dotprod{ r , J  r} \right) \exp\left(- \frac{1}{2}\dotprod{ s , J  s }\right) = \exp\left(\sum_{k\neq i, j} r_k(J_{ki } r_i + J_{kj} r_j)  \right) \leq \exp \left(\sum_{k\neq i, j} (\abs{J_{ki}}+ \abs{J_{kj}}) \right) = e^{O(\sqrt{n})}.    \]
\end{itemize} 
The same inequality holds if we replace $ r$ with $r'$ and $s$ with $s'.$ Thus \cref{ineq:exchange 2} holds in both cases and we finished showing the exchange inequality for $\Pi(\nu).$

 For any distribution $\nu'$ on $2^{[n]},$ by the construction of the polarized walk we have that 
 \[ \Pi(\nu' D^{\text{pol}} U^{\text{pol}}) = \Pi(\nu') D_{n\to (n-1)} U_{(n-1)\to n}. \]
 Thus, for $t_1 = O(n \log n)$, by the result of \cite{ALOVV21} we have
 \begin{align*}
  \DKL {\nu (D^{\text{pol}} 
U^{\text{pol}})^{t_1}\river \nu } &= \DKL {\Pi(\nu (D^{\text{pol}} U^{\text{pol}})^{t_1})\river \Pi(\nu) } \\
&= \DKL { \Pi(\nu')  (D_{n\to (n-1)} U_{(n-1)\to n})^{t_1}\river \Pi(\nu)}   \\
&\leq \text{poly}(n)
 \end{align*}
This combines with the entropy contraction implies that after $O(\frac{1}{1- 2 \opnorm{J}} n \log \frac{n}{\epsilon} ) $ steps of the polarized walk, we obtain a distribution $\hat{\nu}$ s.t. $d_{TV}(\hat{\nu}, \nu) \leq \epsilon.$ 

\paragraph{Implementing a step of the walk.} 
Now we show each step of the polarized walk can be implemented in $O(d \log n)$ time.
For each vertex $i,$ we define the quantities $B_i(x) = \sum_j Q_{ij} x_j + h_i$ and $R_i(x) = \exp( B_i(x))$ which represent the effect of the other sites on site $i$. Note that we can compute $B_i(x)$ and $ R_i(x)$ in $O(d)$ time given the list of neighbors of $i.$  Throughout the algorithm, we store
\begin{itemize}
    \item The current state $x$, which is a vector on the hypercube $\{\pm 1\}^n$, and the set $x_+=\set{i\in [n]:x_i=1}.$
    \item The current value $V.$ We maintain the invariant that $V:= V(x) =\exp\left(\frac{1}{2}\dotprod{ x, Q x} + \dotprod{ h, x} \right).$ 
\item A data structure $\mathcal{D}$ storing tuples $(i,R_i)$ at the vertices of a binary search tree keyed by $i$, which additionally allows the following operations:
    \begin{enumerate}
        \item
        Sum(): output $\sum_i R_i$ for all $(i, R_i)$ currently stored in the data structure.
        \item Range-Search($v$, $\ell$): given $\ell \geq 0,$ output the minimum $i$ in the subtree rooted at the node $v$ such that
        \[\sum_{j< i} R_j \geq \ell. \]
        We omit $v$ and simply write Range-Search($\ell$) if $v$ is the root of the tree.
        \item 
        $\text{Update} (i, R)$: sets  $R_i$ to $R.$ If $ i$ doesn't already exist in the data structure then it inserts key $i$ with value $R_i = R$ into the tree. 
        \item $\text{Delete}(i)$: delete the pair $(i,R_i)$ from the binary search tree. 
    \end{enumerate}
\end{itemize}
We maintain the invariant that $i\in \mathcal{D}$ iff $ x_i = -1$ and $R_j = \exp(B_j(x)) > 0$ for all $j\in \mathcal{D}.$
 With this invariant, $\mathcal{D}$ contains at most $n$ nodes at any given time, so all operations on this data structure take $O(\log n)$ time by using the implementation specified in \cref{prop:data structure}.
  It takes $O(nd \log n)$ time to initialize the algorithm at an arbitrary $x \in\set{\pm }^n$ by inserting all vertices assigned to $-$ into $\mathcal{D}. $
  
Now, we show that each step of the polarized walk can be implemented in $O(d\log n)$ time.   

For the down step, in $O(\log n)$ time, we sample a coordinate $ i$ in $x_+$ with probability $1/n$ or $\perp$ with probability $ \frac{n - \abs{x_+}}{n}.$ Let $T=x_+.$ If the sample is a coordinate $i\in x_+,$ in a total of $O(d\log n)$ time we perform the following updates: 
\begin{itemize}
    \item Update $T = T \setminus \set{i}, $ compute $R_i = \exp(B_i(x))$ and insert $(i, R_i)$ into $ \mathcal{D}.$ Update $V \leftarrow V/R_i^2.$
    \item Update $B_{j} (x) \leftarrow B_j(x) -2 Q_{ij}$ and $R_j \leftarrow R_j /\exp(2 Q_{ij}) $ for all $j \in \mathcal{N}(i) $ such that $x_j = -1$.
    \item Update $x = \chi_T.$
\end{itemize}
Note that we maintain the invariant 
$V=\exp(\frac{1}{2}\dotprod{ x, Q x} + \dotprod{ h, x} )$ 
with $x = \chi_T$, $R_j = \exp(B_{j}(x))$ and $j\in \mathcal{D} $ if $x_j= -1,$ or equivalently, $j\not\in T.$
 
For the up step, let $L$ be the output of Sum(). Then:
\begin{enumerate}
    \item With probability $\frac{n -\abs{T} }{n} ( \frac{L}{n}  +\frac{n -\abs{T} }{n} )^{-1}, $  set the new state $x'$ s.t. $x'_+=T$ and finish the step.
    \item Otherwise (so with probability $ \frac{L}{n} (\frac{L}{n}  +\frac{n -\abs{T} }{n} )^{-1}$),  do the following
\begin{itemize}
    \item Sample $\ell$ uniformly at random in the interval $[0,L].$ Let $j$ be the output of Range-Search($\ell$). Note that $j$ is sampled with probability 
    \[ \frac{R_j }{L} =\frac{\exp(B_j (\chi_T) ) }{\sum_{j'\not\in T}  \exp(B_{j'} (\chi_T) ) }= \frac{ \nu(\chi_{T\cup \set{j}}  ) }{\sum_{j'\not\in T } \nu(\chi_{T\cup \set{j'}})}. \]
    \item Remove $j$ from $ \mathcal{D},$ and update $V \gets V R_j^2 $.
    \item Update $B_k(x) \leftarrow B_k(x) + 2 Q_{jk} $ and $R_k \leftarrow R_k \exp(2 Q_{jk})$ for all $k\in \mathcal{N}(j) $ such that $x_k = -1.$
    \item  Update $ x =\chi_{T
    \cup \set{j}},$
 \end{itemize}
\end{enumerate} 
The total time for the up step is $ O(\abs{\mathcal{N}(j)}\log n) =O(d \log n).$
\end{proof}
\begin{proposition}\label{prop:data structure}
    The data structure $\mathcal{D}$ as described in the proof of \cref{thm:ising projected downup} can be implemented so that each operation takes $O(\log n)$ time, where $n$ is an upper bound on the number of nodes stored in $\mathcal{D}$ at any given time.
\end{proposition}
\begin{proof}
    To implement $\mathcal{D},$ we will use a self-balancing binary search tree (BST) --- for example, an AVL tree or red-black tree (see e.g.\ \cite{cormen2022introduction}). Each node of the tree stores the tuple $(i,R_i)$ and is keyed by $ i$, i.e., a transversal of the tree should return a list sorted in increasing order $i$ of all tuples $(i,R_i)$ in the tree.  Each node also stores the sum of $ R_j$ for all $j$ in its left subtree
\[S_i = \sum_{j \in \text{left-subtree} (i)} {R_j }.\] 
Insertion/deletion/update for BST involves following the path from the root to a certain node, and we only need to update $S_j$ along this path, which is of length $O(\log n)$ in a self-balancing BST. Thus, insertion/deletion/update takes worst case $O(\log n)$ time. 

To implement Sum(), output the sum of $S_j$ along each node in the path from the root to the element with the largest key $i$ in $\mathcal{D}.$
The algorithm only does $O(1)$ work at each node along a path from the root to a certain node, again taking $O(\log n)$ time.

To implement Range-Search($v$, $\ell$), 
repeat the following:
\begin{itemize}
    \item If $\ell  = S_v,$ output $v$.
    \item If $\ell < S_v:$ If the left-child $v.\text{left}$ of $v$  exists, recursively call Range-Search($v.\text{left}$, $\ell$). If not, output $v.$
    \item If $\ell > S_v:$  If the right-child $v.\text{right}$ of $v$  exists, recursively call Range-Search($v.\text{right}$, $\ell-S_v$). If not, output $v.$ 
\end{itemize}

The algorithm only needs to do $O(1)$ work at each node along a path from the root to a certain node, so it runs in time $O(\log n).$

Let $i$ be the correct output of Range-Search($v, \ell$) according to its specification, i.e., the minimum $i$ in the subtree rooted at the node $v$ such that
\[\sum_{j< i} R_j \geq \ell. \]
We prove by induction that this $i$ is indeed the output of our implementation. 
Let $T(v)$ be the tree rooted at $v$ and consider three cases:
\begin{itemize}
    \item If $ \ell = S_v = \sum_{j < v, j\in T(v)} R_j$ then $i=v,$ thus our algorithm behaves correctly.
    \item If $\ell< S_v = \sum_{j < v, j\in T(v) } R_j$ then $v > i$, since $i$ is the minimum key satisfying the above inequality, i.e., $i = \min \set{i' \in T(v) \given \ell \leq \sum_{j < i', j\in T(v)} R_j} , $ thus $ i\in T(v.\text{left}).$
For any $i' \in T(v.\text{left})$ 
\[ \sum_{j < i', j \in T(v.\text{left})} R_j = \sum_{j < i', j <v,j\in T(v)} R_j = \sum_{j< i', j\in T(v) } R_j \]
thus $i = \min \set{i' \in T(v.\text{left}) \given \ell \leq \sum_{j < i', j\in T(v.\text{left}) } R_j},$ and thus by the inductive hypothesis it will be the output of Range-Search($v.\text{left}, \ell$).
    \item If $\ell> S_v = \sum_{j < v, j\in T(v) } R_j$ then $v < i$ and $i\in T(v.\text{right})$, since $\sum_{j< i, j\in T(v)} R_j \geq \ell > \sum_{j < v , j \in T(v)} R_j.$ 
Analogously, $\sum_{j < i', j\in T(v) } R_j\geq \ell $ iff $i'  \in T(v.\text{right}).$
For any $i' \in T(v.\text{right})$ 
\[ S_v+ \sum_{j < i', j \in T(v.\text{right})} R_j = \sum_{j<v, j \in T(v)} R_j + \sum_{j< i' } R_j  = \sum_{j< i' , j\in T(v) } R_j  \]
thus \[i = \min \set*{i'\in T(v) \given \ell \leq \sum_{j < i', j\in T(v) } R_j} = \min \set*{i' \in T(v.\text{right}) \given \ell -S_r \leq \sum_{j < i', j\in T(v.\text{right}) } R_j},\] 
thus by the inductive hypothesis it will be the output of Range-Search($v.\text{right}, \ell-S_v $).

\end{itemize}
\end{proof}
\subsection{Concentration of measure and entropy-transportation inequality}\label{sec:concentration}

We say a Markov operator $P$ with state space $\set{\pm 1}^n$ is $\rho$-local if $ P(x,y) \neq 0 $ only if $\norm{x-y}^2 \leq \rho.$   By a standard Herbst argument \cite{hermon2019modified}, if a $ \rho$-local Markov chain with stationary distribution $\nu:\set{\pm 1}^n \to \R_{\geq 0}$ has modified log-Sobolev constant $\alpha$ 
then $\nu$ exhibits sub-Gaussian concentration of Lipschitz function with constant $O(\alpha).$   
\begin{definition}\label{def:lipschitz concentration}
We say $\nu:\set{\pm 1}^n \to \R_{\geq 0}$ satisfies sub-Gaussian concentration of Lipschitz functions with constant $K$ if   
\[\P*_{S \sim \nu} {f(S) \geq \E_{\nu}{f(S)} + a } \leq \exp\parens*{-\frac{Ka^2 }{
2 c ^2}}.\]  
where $f : \set{\pm 1}^n \to \R$ is an arbitrary a $c$-Lipschitz functional with respect to the Hamming metric. This is equivalent to a $W_1$-entropy transport inequality  \cite[Theorem 4.8]{van2014probability}.
\end{definition}
\begin{corollary}\label{corollary:concentration}
Under the assumptions of \cref{thm:ising projected downup}, the probability measure $\nu$ satisfies sub-Gaussian concentration of Lipschitz functions with constant $K = O\left(\frac{(1 - 2\opnorm{J})}{n}\right)$.
\end{corollary}

\subsection{Antiferromagnetic Ising model}
\begin{proposition} \label{prop:ising graph}
    Let $\nu:\set{\pm 1 }^n \to \R_{\geq 0}$ be the antiferromagnetic Ising model on graph $G = G([n],E)$, i.e., 
    \[\nu(x) \propto \exp \left(-\frac{\beta}{2} \dotprod{ x, A x} + \dotprod{ h ,x}\right)\]
    where $A$ is the adjacency matrix of $G,$ and $\beta\geq 0 $ and $h\in \R^n$ are given parameters.

Let $ d_{\max}, d_{\min}$ be maximum and mimimum degrees of $G$ respectively. Consider the normalized Laplacian matrix $L$ of $G$, i.e., $L = I - D^{-1/2} A D^{-1/2}$ where $D$ is the diagonal matrix with $D_{i,i}$ be the degree of vertex $i.$ The eigenvalues of $L$ are $0=\lambda_1\leq \lambda_2 \leq \cdots  \lambda_{n} \leq 2.$ Let $\lambda(G) =\max_{i\neq 1} \abs{1- \lambda_i} .$

If $0\leq \beta \leq \frac{1-\delta}{4}  (d_{\max} -d_{\min} + d_{\max}\lambda(G))^{-1}$ 
 for $\delta >0$, then \cref{thm:ising projected downup} and \cref{corollary:concentration} applies i.e.
\begin{enumerate}
    \item We can sample from $\hat{\nu}$ s.t. $d_{TV}(\hat{\nu}, \nu) \leq \epsilon$ by running the polarized walk for $O(\delta^{-1} n\log \frac{n}{\epsilon})$ steps. The total runtime is $O( \delta^{-1} n d_{\max} \log \frac{n}{\epsilon} \log n) .$
    \item $\nu$ has sub-Gaussian concentration of Lipschitz functions.
\end{enumerate}
\end{proposition}
\begin{proof}
Let the eigenvalues of $D^{-1/2} A D^{-1/2} = I-L $ be $\beta_i = 1-\lambda_i.$
    Note that the unit eigenvector associated with $ \beta_1 = 1$ is $ v_1 = \frac{1}{\sqrt{n}}\vec{1},$ thus we can write
    \[ D^{-1/2} A D^{-1/2} = \frac{1}{n} \vec{1} \vec{1}^\intercal + B \]
    where, since $B$ is symmetric, $\opnorm{B} = \max_{i\neq 1} \abs{\beta_i} = \max_{i\neq 1}\abs{1-\lambda_i}  =\lambda(G).$ Thus
    \begin{align*}
      -\beta A &=  -\frac{\beta}{n} (D^{1/2}\vec{1}) (D^{1/2}\vec{1})^\intercal -\beta D^{1/2} B D^{1/2} \\
      & = -  \frac{\beta d_{\max} }{n } \vec{1} \vec{1}^\intercal +  \frac{\beta}{n} (d_{\max} \vec{1} \vec{1}^\intercal-  (D^{1/2}\vec{1}) (D^{1/2}\vec{1})^\intercal ) - \beta D^{1/2} B D^{1/2} 
    \end{align*}
    Let 
    \[ U = \frac{1}{n} (d_{\max} \vec{1} \vec{1}^\intercal-  (D^{1/2}\vec{1}) (D^{1/2}\vec{1})^\intercal ), \qquad V = -  D^{1/2} B D^{1/2}, \qquad J = \frac{1-\delta }{4} I+ \beta(U+V) . \]
    Note that 
    \[ u_{ij} = \frac{1}{n} (d_{\max} -D_i^{1/2} D_j^{1/2}) \leq \frac{1}{n} (d_{\max} - d_{\min}) \]
    so $ \opnorm{U} \leq\norm{U}_{\infty}= \max_i \sum_j \abs{U_{ij}} \leq  (d_{\max} - d_{\min}).$ Also,
    \[\opnorm{V} \leq \opnorm{D^{1/2}} \opnorm{B} \opnorm{D^{1/2}} \leq d_{\max} \lambda(G)\]
    thus \[\opnorm{\beta(U+V)}\leq \beta (\opnorm{U}+\opnorm{V}) \leq \abs{\beta} (d_{\max}-d_{\max} + d_{\max} \lambda(G)) \leq \frac{1-\delta }{4} \]
    and $0\preceq  J \preceq \frac{1-\delta}{2}.$ Since adding multiples $I$ to $A$ doesn't change the distribution $\nu,$ we can write
    \[\nu(x) \propto  \exp\parens*{\frac{1}{2} \dotprod{ x, J x } + \dotprod{ h, x } - \frac{\beta d_{\max} }{2n}\parens*{\sum_i x_i}^2}\]
    with $J$ as defined above, thus \cref{thm:ising projected downup} applies. Also, since all non-diagonal entries of $Q$ is same as that of $-A,$ the maximum number of non-diagonal nonzero entries in each row of $Q$ is bounded above by $d_{\max}.$
   \end{proof}
We thus obtain the following results for the antiferromagnetic Ising model on random graphs.
\begin{corollary}\label{cor:random d regular}
Let $\nu$ be the antiferromagnetic Ising model with parameter $\beta $ on $G$ where $G$ is a random $d$-regular graph. Suppose $0\leq \beta \leq \frac{1-\delta }{8 \sqrt{d-1}}$ for a  constant $\delta > 0.$  With probability $1-o(1)$ over the random instance $G$, we can sample from $\hat{\nu}$ s.t. $d_{TV}(\hat{\nu}, \nu) \leq \epsilon$ in $O(\delta^{-1} n d\log \frac{n}{\epsilon}\log n)$ time and $\nu$ has sub-Gaussian concentration of Lipschitz function.
\end{corollary}
\begin{proof}
By Friedman's theorem \cite{Friedman2003APO}, with probability $1-o(1),$ we have that $d_{\max} \lambda(G)=d\lambda(G) \leq 2 \sqrt{d-1} + o(1).$
\end{proof}
   \begin{corollary}\label{cor:erdos renyi}
       Let $\nu$ be the antiferromagnetic Ising model with parameter $\beta $ on $G$ where
       $G = G(n, p)$ is a Erdos-Renyi random graph. Let $d=(n-1)p$ be the expected degree. Suppose $p > \frac{\log n}{n}.$ There exists constant $C > 0$ s.t. if  $0\leq \beta \leq \frac{1}{C  \sqrt{d \log n}}, $ then with probability $ 1 - o(1)$ over the random instance $G,$ we can sample from $\hat{\nu}$ s.t. $d_{TV}(\hat{\nu}, \nu) \leq \epsilon$ in $O(n d\log \frac{n}{\epsilon}\log n)$ time and $\nu$ has sub-Gaussian concentration of Lipschitz function.
   \end{corollary}
   \begin{proof}
       By a standard Chernoff bound, with probability $ 1-o(1)$, $ d_{\max} \leq d+ O(\sqrt{d \log n}) $ and $d_{\min} \geq d-O(\sqrt{d \log n}) .$ With probability $1-o(1)$, $G$ is connected. By \cite{CojaOghlan2007OnTL} and \cite[Theorem 1.1]{Hoffman2012SpectralGO}, $\lambda(G) \leq O(\frac{1}{\sqrt{d}}).$ Applying \cref{prop:ising graph} gives the desired result. 
   \end{proof}

   \subsection{Application to fixed magnetization models}
   As a special case, our general theory applies to the fixed magnetization Ising model, i.e., the restriction of the Ising measure to the slice $\sum_i x_i = k$ of the hypercube $\{\pm 1\}^n$ for any $k$. See e.g.\  \cite{carlson2022computational}.
   This is because we can recover fixed magnetization as the limit $\gamma \to \infty$ of our more general model. These results will be applicable in particular to models on both ferromagnetic and antiferromagnetic expander graphs.
   \begin{proposition}\label{prop:simple-limit}
   Suppose that integers $k,n$ satisfy $-n \le k \le n$ and $k \equiv n \mod 2$. Suppose $J \succeq 0$ and $h \in \mathbb{R}^n$ are arbitrary and let $\mu$ be a probability measure on the slice $\{x \in \{\pm 1\}^n : \sum_i x_i = k \}$ with probability mass function
    \[ \mu(x) \propto \exp\left(\frac{1}{2} \langle x, J x \rangle + \langle h, x \rangle \right). \]    
    
    For every $\lambda \ge 0$, define the Ising model $\nu_{\lambda}$ on the hypercube $\{\pm 1\}^n$ by its probability mass function
    \[ \nu_{\lambda}(x) \propto \exp\left(\frac{1}{2} \langle x, J x \rangle + \langle h, x \rangle - \frac{\lambda}{2}\left(\sum_i x_i - k\right)^2 \right). \]
    Then $\nu_{\lambda} \to \mu$ as $\lambda \to \infty$. 
   \end{proposition}
   \begin{proof}
       The parity constraint is exactly the condition which determines whether $\sum_i x_i = k$ has a solution for $x \in \{\pm 1\}^n$.
       So the result follows because $\nu_{\lambda}(x) \to 0$ as $\lambda \to \infty$ for any $x \in \{\pm 1\}^n$ not satisfying the constraint $\sum_i x_i = k$, and because the conditional law of $x$ given $\sum_i x_i = k$ under $\nu_{\lambda}$ is exactly $\mu$. 
   \end{proof}

    As the following corollary shows, the result we prove has consequences not just for the polarized walk but also the ordinary down-up walk. To be more precise, we can identify a vector $x \in \{\pm 1\}^n$ with the set $x_+ = \{i : x_i = +1\}$ of plus-valued entries, so the usual $D_{k \to k - 1}$ down operator on sets corresponds to picking a random $+1$ entry of $x$ and setting it to $-1$, and as usual the up operator $U_{k - 1 \to k}$ samples from the posterior on $x$ given such an observation. As the proof of the following result shows, the down-up walk is simply a less lazy version of the polarized walk, and this results in it mixing faster when the number of $+$ entries is small. 

   \begin{corollary} \label{cor:down up walk slice}
   Let $\mu$ be as defined in \cref{prop:simple-limit}. If $\opnorm{J} < 1/2$ then
   \begin{enumerate}
    \item The polarized down-operator $D^{\text{pol}} $ has $\left(1-\frac{1-2\opnorm{J}}{n}\right)$-entropy contraction with respect to $\mu.$ 
    \item The polarized walk on $ \nu$ mixes within $\epsilon$-TV distance in $O\left(\frac{1}{1-2\opnorm{J}} n\log(n/\epsilon)\right)$ steps. 
    \item Let $m = (n + k)/2 \in [0,n]$ be the total number of $+$ spins under the fixed magnetization model.
By identifying a spin assignment $x\in \set{\pm 1}^n$ with the set of vertices assigned to $+$, we can view $\mu$ as a distribution over $\binom{[n]}{m}.$    The (ordinary) down-operator $D_{m \to (m - 1)}$ has $\left(1 - \frac{1 - 2\opnorm{J}}{m}\right)$-entropy contraction with respect to $\mu$.
    \item The (ordinary) down-up walk on $\nu$  mixes within $\epsilon$-TV distance in $O\left(\frac{1}{1-2\opnorm{J}} m\log(m/\epsilon)\right)$ steps. 
    \item The probability measure $\mu$ satisfies sub-Gaussian concentration of Lipschitz functions with constant $K = O\left(\frac{1 - 2\opnorm{J}}{n - \abs{k}}\right)$.
   \end{enumerate}
   \begin{proof}
    The first two points follow by combining \cref{prop:simple-limit} with \cref{thm:ising projected downup} taking $\gamma = \lambda$. For the third point, first observe that from the output of the channel $D^{\text{pol}}$ we can determine whether a $+$-spin was dropped, since the magnetization is fixed. Hence, by the chain rule for KL divergence we have
    \[ \DKL{\mu' D^{\text{pol}} \river \mu D^{\text{pol}}} = \frac{n + k}{2n} \DKL{\mu' D \river \mu D} + \frac{n - k}{2n} \DKL{\mu' \river \mu}, \]
    so using the contraction of the polarized down operator, we have
    \begin{align*} 
    \DKL{\mu' D \river \mu D} 
    &= \frac{2n}{n + k} \DKL{\mu' D^{\text{pol}} \river \mu D^{\text{pol}}} - \frac{n - k}{n + k}\DKL{\mu' \river \mu} \\
    &\le  \left(\frac{2n}{n + k}\left(1 - \frac{1 - 2\opnorm{J}}{n}\right) - \frac{n - k}{n + k}\right) \DKL{\mu' \river \mu } \\
    &\le \left(1 - \frac{2(1 - 2\opnorm{J})}{n + k}\right) \DKL{\mu' \river \mu }.
    \end{align*}
    The mixing time bound follows in the same way as before (using the exchange inequality). 

For the last point, let $m$ and $m'$ be the number of $+$ and $-$ spins respectively, then $ \min \set{m, m'} =\frac{n-\abs{k}}{2}.$ The down-up walk on the set of vertices assigned to $+$ ($-$ resp.) has modified log-Sobolev constant  $ \Omega\left(\frac{1 - 2\opnorm{J}}{ m}\right)$ ($ \Omega\left(\frac{1 - 2\opnorm{J}}{ m'}\right)$ resp.). Both of these walks are $2$-local walks, so the conclusion follows by the Herbst argument, same as in \cref{corollary:concentration}.  
   \end{proof}

   \paragraph{Application: ferromagnetic Ising with fixed magnetization on expanders.}

   \begin{proposition} \label{prop:fixed magnetization expander general}
     Let $\mu:\set{\pm 1 }^n \to \R_{\geq 0}$ be the \emph{canonical} Ising model on graph $G = G([n],E)$ at fixed magnetization $k.$ More precisely, $\mu$ is a probability measure on the slice $\{x \in \{\pm 1\}^n : \sum_i x_i = k \}$ with probability mass function   
    \[\mu(x) \propto \exp \left(-\frac{\beta}{2} \dotprod{ x, A x} + \dotprod{ h ,x}\right)\]
where $A$ is the adjacency matrix of $G,$ and $\beta$ and $h\in \R^n$ are given parameters

Let $ d_{\max}, d_{\min}$ be maximum and mimimum degrees of $G$ respectively. Consider the normalized Laplacian matrix $L$ of $G$, i.e., $L = I - D^{-1/2} A D^{-1/2}$ where $D$ is the diagonal matrix with $D_{i,i}$ be the degree of vertex $i.$ The eigenvalues of $L$ are $0=\lambda_1\leq \lambda_2 \leq \cdots  \lambda_{n} \leq 2.$ Let $\lambda(G) =\max_{i\neq 1} \abs{1- \lambda_i} .$

If $\abs{\beta} \leq \frac{1-\delta}{4}  (d_{\max} -d_{\min} + d_{\max}\lambda(G))^{-1}$ 
 for $\delta >0$, then \cref{cor:down up walk slice} applies i.e., 
\begin{enumerate}
    \item  Let $m = (n + k)/2 \in [0,n]$ be the total number of $+$ spins under the fixed magnetization model. We can sample from $\hat{\mu}$ s.t. $d_{TV}(\hat{\mu}, \mu) \leq \epsilon$ by running the down-up walk for $O(\delta^{-1} m\log \frac{n}{\epsilon})$ steps.

 The down-up walk can implemented using $ O(nd_{\max} \log n)$ time for preprocessing and $O(d_{\max}\log n)$ time for each step, hence the total runtime is $O( n d_{\max}\log n+ \delta^{-1} m d_{\max} \log n \log \frac{m}{\epsilon} ) .$

 When the external field is uniform, i.e., $ h_i  = \bar{h} \forall i$ and the graph is $d$-regular ($d_{\max} = d_{\min } =d$), then the preprocessing time and the time to implement each step can be reduced to $O(m d \log(md))  $ and $ O(d \log (md) +\log \frac{1}{\epsilon})$ respectively, thus the total runtime is $ O(\delta^{-1} m\log \frac{m}{\epsilon}  (d \log (m d)  \log \frac{1}{\epsilon}) ). $
    \item $\mu$ has sub-Gaussian concentration of Lipschitz functions with constant $K = O\left(\frac{\delta }{n - \abs{k}} \right)$.
\end{enumerate}
\end{proposition}
   \end{corollary}
   \begin{proof}
       As in proof of \cref{prop:ising graph}, we can rewrite the probability mass function of $\mu$ as
       \[\mu(x) \propto \exp \left(\frac{1}{2}\langle x, J x\rangle + \langle h, x\rangle - \frac{\beta d_{\max}}{2n} (\sum_i x_i)^2 \right)  \]
       where $0\preceq J \preceq \frac{1-\delta}{2}.$
       Since $\mu$ is supported on the slice $\sum_i x_i = k,$ we can further rewrite $\mu$ as
       \[\mu(x) \propto \exp \left(\frac{1}{2}\langle x, J x\rangle + \langle h, x\rangle \right) \]
       thus \cref{cor:down up walk slice} applies.

       For implementing the random walk, we use the algorithm in \cref{thm:ising projected downup}, which resulted in the stated runtime.

        Below, we discuss how to improve the runtime for $d$-regular graphs with uniform external fields. We use the same notation as in the proof of \cref{thm:ising projected downup}. For a given assignment $x,$ $B_i(x)$ only depends on the number of neighbors of $i$ with $+$ spin. In particular, if $ i$ has no such neighbors, then $B_i(x) = - d + \bar{h}.$ Thus, instead of keeping track of all vertices with $-$ spin as the algorithm in \cref{thm:ising projected downup} does, we only need to keep track of those which have at least one neighbor with $+$ spin. More precisely, we insert $i$ with $ x_i= -1$ and $\mathcal{N}(i) \cap T \neq \emptyset$ into $\mathcal{D},$ and maintain the set $\text{Free}_-  = [n] \setminus T \setminus \mathcal{D}.$ We can compute $F = \sum_{i\in \text{Free}_-} \exp(B_i(x)) =  \exp(-d+\bar{h}) \abs{\text{Free}_-}$ and $L = \sum_{i\in \mathcal{D}} \exp(B_i(x))$ by calling Sum() on $\mathcal{D}.$ Thus, we can implement the up step by uniformly sampling from $\text{Free}_-$ with probability $\frac{F}{L+F}$ and (weighted) sampling from $ \mathcal{D}$ with probability $\frac{L}{L+F}$ where the weighted sampling is implemented in the same way as in \cref{thm:ising projected downup}. The down step can be implemented same as in \cref{thm:ising projected downup}, except that we insert the sampled coordinate $i$ into $\mathcal{D}$ or $\text{Free}_-$ depending on the assignment of its neighbor. 

        \paragraph{Maintaining $\mathcal{D}.$}
        
        The number of vertices in $\mathcal{D}$ is bounded by $ m d ,$ since the $m$ vertices with $+$ spin have at most $m d $ neighbors.
At initialization, it takes $O(md)$ time to compute the number $n^+_i(x)$ of $+$ neighbors of each $-$ vertex $i$ with at least one $+$ neighbor. To do so, we loop through all $+$ vertices $j$ and their neighbors $i$ and increase $ n^+_i(x)$ by $1.$ Since $R_i(x)$ is a function of $ n^+_i(x)$, i.e., $R_i(x) =\exp(B_i(x)) = \exp(2 n^+_i - d + \bar{h})$, we can build $\mathcal{D}$ by inserting $(i, R_i(x))$ with $n_i^+(x) > 0$\footnote{There are $O(md)$ such $i.$} in total $O(md \log (md))$ time. Maintaining $\mathcal{D}$ in each down-up step takes $O(d \log(md))$ time.

Note also that computing the initial value $ V (x) = -\exp(\frac{\beta}{2}\langle x, Ax\rangle + \langle h, x\rangle )$ can also be done in $O(md)$ time given $ n^+_i(x)$ for $i$ s.t. $n^+_i(x) > 0.$ 
Indeed, let $ S_1 = \set{i:  n^+_i(x) > 0}$ and $ S_2 = \set{i: n^+_i(x)=0 \text{ and } x_i = 1}$, we can write 
\[\log V(x) = -\frac{\beta}{2} \sum_i x_i B_i(x)  =  -\frac{\beta}{2} \left ( \sum_{i\in S_1} x_i B_i(x)+  (\bar{h} -d)(\abs{S_2} - (n-\abs{S_1} - \abs{S_2}))  \right)\]
Thus, computing $V(x)$ takes $\abs{S_1} +\abs{S_2} = O(md)$ time.
\paragraph{Maintaining $\text{Free}_-.$}
        Since $\abs{\text{Free}_-} \geq n - m(d+1),$ if $md\log (md) < n/4$ then uniformly sampling from $\abs{\text{Free}_-}$ can be done by uniform sampling from $ [n]$\footnote{Each sample can be produced in constant time by hashing.} then accept the sample if it is not in $ T\cup \mathcal{D}.$ To ensure that the failure probability stays below $\epsilon /2 $ in $\poly(md)$ steps of the down-up walk, we only need to do rejection sampling $\log \frac{md}{\epsilon}$ times. In this case, there is no need to keep an explicit data structure for $\text{Free}_-, $ and the preprocessing time is $O(md \log (md)).$ Each step of the down-up walk takes $ O(d \log (md)+ \log (\frac{md}{\epsilon})) = O(d \log (md) +\log(\frac{1}{\epsilon}))$ time to maintain $\mathcal{D}$ and $\text{Free}_-.$

        If $m d \log (md) > n/4$ then we can keep a data structure (i.e., a binary search tree) for  $\text{Free}_-$ that enables uniform sampling for preprocessing time of $ O(\abs{\text{Free}_-}) = O(n) = O(md \log(md) )$, and the sampling time is $O(\log \abs{\text{Free}_-}) = O(\log (md))$. Each step of the down-up walk takes $O(d \log (md))$ time.

   \end{proof}

   \begin{corollary}[Fixed-magnetization on random-$d$-regular graphs and expanders] 
   \label{cor:d regular expander fixed magnetization}
       Let $\mu$ be the canonical (ferromagnetic or antiferromagnetic) Ising model with fixed magnetization $k$ and parameter $\beta$ on graph $G.$ 
       If $G$ is a random-$d$ regular graph and $ \beta < \frac{1}{8 \sqrt{d-1}}$, with probability $1-o(1)$ over the random instance $G$, the down-up walk on $\mu$ mixes in $O_{\beta}(m \log \frac{m}{\epsilon})$ steps. 
   \end{corollary}

   \begin{corollary}[Fixed-magnetization on Erdos-Renyi random graphs]\label{cor:fixed magnetization erdos renyi}
        Let $\mu$ be the canonical (ferromagnetic or antiferromagnetic) Ising model with fixed magnetization $k$ and parameter $\beta$ on $G$ where $G= G(n,p)$ is an Erdos-Renyi random graph.  Let $d=(n-1)p$ be the expected degree. Suppose $p > \frac{\log n}{n}.$ Let $m=\frac{n+k}{2}.$ There exists constant $C > 0$ s.t. if  $0\leq \beta \leq \frac{1}{C  \sqrt{d \log n}}, $ then with probability $ 1 - o(1)$ over the random instance $G,$ the down-up walk on $\mu$ mixes in $O(m \log \frac{m}{\epsilon})$ steps.
   \end{corollary}

   }
    
    \Tag<sigconf>{\begin{acks}
	This work was supported by NSF CAREER Award CCF-2045354. Thuy-Duong Vuong was supported by a Microsoft Research Ph.D.\ Fellowship. Frederic Koehler was supported in part by NSF award CCF-1704417, NSF award IIS-1908774, and N. Anari’s Sloan Research Fellowship, and thanks the participants of the Simons Institute program on Probability, Geometry, and Computation in High Dimensions for related discussions.
	
\end{acks}
}
    \Tag{\PrintBibliography}
    
    \Tag{\appendix
    \section{Illustrative examples of trickle-down}
\subsection{Illustration: Oppenheim's trickle-down}\label{a:oppenheim}
For illustrative purposes, we show how Oppenheim's celebrated trickle-down result \cite{Opp18} indeed follows by carefully bounding the trickle-down equation \cref{eqn:trickledown-pinning} applied with $T = 1$ and $t = 0$, i.e., from the equation
\begin{equation}\label{eqn:oppenheim-start}
\cov{\nu_0} = \E{\cov{\nu_1}} + \frac{1}{k} \cov{\nu_0} N_0^{-1} \cov{\nu_0}.
\end{equation}
where $N_0 = \diag(\mean{\nu_0})$ is a diagonal matrix encoding the marginals of the measure $\nu_0$. Also, recall that $N_1$ denotes a diagonal matrix encoding the marginals of the link, so it has entries $(N_1)_{ii} = \nu_1(i)$ for $i \ne a_1$ and $(N_1)_{a_1 a_1} = 0$. 

The proof will be equivalent to the usual one, but we will use slightly different terminology and notation consistent with the rest of this paper. The following fact is not needed for the proof, but is a helpful reminder of the link between spectral independence and the $1 \to k \to 1$ up-down walk:
\begin{lemma}[\cite{ALO20}]\label{lem:si-local}
Suppose that $\nu$ is a probability measure on $\binom{[n]}{k}$. Then $\nu$ is $C$-spectrally independent iff $\lambda_2(P) \le C/k$, where $P = U_{1 \to k} D_{k \to 1}$ is the (lazy version of the) $1 \to k \to 1$ up-down walk.
\end{lemma}
See e.g.\ the preliminaries of \cite{anari2023universality} for a self-contained proof of the previous lemma with identical notation.

Recall from \cref{fact:si-flc} that $C$-spectral independence is equivalent to the statement $\cov{\nu} \preceq Ck \Pi$ where $\cov{\nu}$ is the covariance matrix of the random vector $1_S$ for $S \sim \nu$, and $\Pi = \diag(\pi)$ where $\pi_{i} = \nu(i)/k$. The following lemma gives a slightly tighter PSD inequalilty which is yet another equivalent formulation of spectral independence. 
\begin{lemma}\label{lem:si-variant}
Suppose that $\nu$ is a probability measure on ${[n] \choose k}$
which is $C$-spectrally independent, equivalently $\cov{\nu} \preceq Ck \Pi$.
Then in fact
\[ \cov{\nu} \preceq Ck(\Pi - \pi \pi^{\intercal}). \]
\end{lemma}
\begin{proof}
First recall (by expanding the definitions) that if $P = U_{1 \to k} D_{k \to 1}$ is the (lazy version of the) $1 \to k \to 1$ up-down walk, then
\[ \cov{\nu_0} = k^2 (\Pi P - \pi \pi^{\intercal}) \]
where $\pi_i = (1/k) \nu_0(i)$ and $\Pi = \diag(\pi)$ so $N_0 = k\Pi$. 

Define $\lambda = C/k$, which by \cref{lem:si-local} is an upper bound on the second eigenvalue of $P$. Recall that $P$ is self-adjoint with respect to the inner product $u,v \mapsto u^{\intercal} \Pi v$. 
Let $u_1 = \vec{1},u_2,\ldots$ be the right-eigenbasis of $P$ guaranteed by the spectral theorem (orthogonal under that inner product) with eigenvalues $\lambda_1 = 1, \lambda_2,\ldots$, and let $f$ be arbitrary with components $f_i = (f^{\intercal} \Pi u_i) u_i$ so by Plancherel's theorem $f^{\intercal} \Pi f = \sum_i f_i^{\intercal} \Pi f_i$. Then for any $f$,
\[ f^{\intercal} \Pi P f = \sum_i \lambda_i f^{\intercal} \Pi f_i = \sum_i \lambda_i f_i^{\intercal} \Pi f_i \le (1 - \lambda) f_1^{\intercal} \Pi f_1 + \lambda \sum_i f_i^{\intercal} \Pi f_i = (1 - \lambda) (f^{\intercal} \pi)^2 + \lambda f^{\intercal}\Pi f  \]
which proves that $\Pi P \preceq (1 - \lambda) \pi \pi^{\intercal} + \lambda \Pi$. Multiplying by $k^2$ proves the result.  
\end{proof}
\begin{lemma}\label{lem:sigma-u2}
Suppose that $\nu$ is a probability measure on ${[n] \choose k}$ and $\lambda_2(P)$ is the second-largest eigenvalue of $P = U_{1 \to k}D_{k \to 1}$, the (lazy) $1 \to k \to 1$ up-down walk. Let $u_2$ be the corresponding right eigenvector of $P$. Then
\[ \cov{\nu} u_2 =  k^2 \lambda_2(P) \Pi  u_2. \]
\end{lemma}
\begin{proof}
Recall, as in the proof of \cref{lem:si-variant}, that
$\cov{\nu} = k^2 (\Pi P - \pi \pi^{\intercal})$
where $\pi_i = (1/k) \nu_0(i)$ and $\Pi = \diag(\pi)$ so $N_0 = k\Pi$. 
Recall that the largest right eigenvector of $P$ is the all-ones vector, and let $u_2$ be the second largest right eigenvector of $P$, which has eigenvalue $\lambda_2(P) < 1$. Recall that $P$ is self-adjoint with respect to the inner product $u,v \mapsto u^\intercal \Pi v$ and so by spectral theorem we have $u_2^\intercal \pi = u_2^\intercal \Pi \vec{1} = 0$.
From the above facts, we have that
\[ \cov{\nu} u_2 = k^2 (\Pi P - \pi \pi^{\intercal}) u_2  =  k^2 \lambda_2(P) \Pi  u_2. \]
\end{proof}
The following statement is Oppenheim's trickledown theorem. It may look a bit different from the usual statement, because we have stated it for a lazy version of the $1 \to k \to 1$ up-down walk, but it is directly seen to be equivalent to the ``active'' version --- see Remark 13 of \cite{anari2023universality}. In particular, note that when $C = 1$, trickle-down preserves the property of being 1-spectrally independent inherited from the link, so trickle-down can easily be applied recursively in this case.
\begin{theorem}[\cite{Opp18}]
With the notation as above,
suppose that almost surely
\[ \Sigma_{1} \preceq C N_1, \] 
and $\lambda_2(P) < 1$ where $P = U_{1 \to k} D_{k \to 1}$ is the (lazy version of the) $1 \to k \to 1$ up-down walk. Then
\[ \Sigma_{0} \preceq \frac{C(k- 2)}{k - 1 - C} N_0. \]
\end{theorem}
\begin{proof}
First observe from \cref{eqn:oppenheim-start} and \cref{lem:si-variant}
\begin{align} 
\Sigma_{0} 
&= \E{\Sigma_{1}} + \frac{1}{k} \Sigma_{0} N_0^{-1} \Sigma_{0} \\
&\preceq C(k - 1)\E{\Pi_1 - \pi_1 \pi_1^{\intercal}} + \frac{1}{k} \Sigma_{0} N_0^{-1} \Sigma_{0} \\
&= C(k - 1)\E{\Pi_1 - \pi \pi^{\intercal}} - C(k - 1)(\E{\pi_1 \pi_1^{\intercal}} - \pi \pi^{\intercal}) + \frac{1}{k} \Sigma_{0} N_0^{-1} \Sigma_{0} \label{eqn:oppenheim-psd}
\end{align}
where $\Pi_1 = \diag(\pi_1)$ and $(\pi_1)_i = \nu_1(i)/(k - 1)$ for $i \ne a_1$, $(\pi_1)_{a_1} = 0$. (In other words $\pi_1$ is the marginal of the ``link''.) Note that for $S \sim \nu_0$
\[ (k - 1)\pi_1 = \E{1_{S \setminus a_1} \mid a_1} = \E{1_S \mid a_1} - 1_{a_1} \]
so by the law of total expectation
\[ (k - 1)\E{\pi_1} = \E{1_S} - \E{1_{a_1}} = (k - 1)\pi, \]
i.e., $\E{\pi_1} = \pi$. Note that
\[ \E*{1_{S \setminus a_1} 1_{S \setminus a_1}^{\intercal}}_{ij} = \begin{cases}
    (1 - 2/k) \E*{1_{S} 1_{S}^{\intercal}}_{ij} & \text{if $i \ne j$}\\
    (1 - 1/k) \E*{1_{S}}_i & \text{if $i = j$}
\end{cases}\]
so
\begin{align*}
\cov{1_{S \setminus a_1}} 
&= \E*{1_{S \setminus a_1} 1_{S \setminus a_1}^{\intercal}} - \E*{1_{S \setminus a_1}} \E{1_{S \setminus a_1}^{\intercal}} \\
&= (1 - 2/k) \E{1_S 1_S^{\intercal}} +  \Pi - (1 - 1/k)^2 \E{1_S} \E{1_S^{\intercal}} \\
&= (1 - 2/k) \Sigma_{\nu_0} + \Pi - \pi\pi^{\intercal}
\end{align*}
so by the law of total variance and \cref{eqn:oppenheim-start}
\begin{align*} 
(k - 1)^2[\E{\pi_1 \pi_1^{\intercal}} - \pi \pi^{\intercal}] 
&= \cov{\E{1_{S \setminus a_1} \mid a_1}} \\
&= \cov{1_{S \setminus a_1}} - \E{\cov{1_{S \setminus a_1} \mid a_1}} \\
&= (1 - 2/k) \Sigma_{0} + \Pi - \pi \pi^{\intercal} - \E{\Sigma_{1}} \\
&= \frac{1}{k} \Sigma_{0} N_0^{-1} \Sigma_{0} - \frac{2}{k} \Sigma_{0} + \Pi - \pi \pi^{\intercal}.
\end{align*}
Substituting into \cref{eqn:oppenheim-psd} yields
\begin{align*}
\Sigma_{0} 
&\preceq C(k - 1)[\Pi - \pi \pi^{\intercal}] - C(k - 1)(\E{\pi_1 \pi_1^{\intercal}} - \pi \pi^{\intercal}) + \frac{1}{k} \Sigma_{0} N_0^{-1} \Sigma_{0}  \\
&= C(k - 1)[\Pi - \pi \pi^{\intercal}] - \frac{C}{k - 1}\left(\frac{1}{k} \Sigma_{0} N_0^{-1} \Sigma_{0} - \frac{2}{k} \Sigma_{0} + \Pi - \pi \pi^{\intercal}\right) + \frac{1}{k} \Sigma_{0} N_0^{-1} \Sigma_{0}   
\end{align*}
and rearranging gives
\begin{equation}\label{eqn:oppenheim-mid} \left(1 - \frac{2C}{k(k - 1)}\right) \Sigma_{0} \preceq C(k - 1 - 1/(k - 1))[\Pi - \pi\pi^{\intercal}] + \frac{1}{k} (1 - C/(k - 1)) \Sigma_{0} N_0^{-1} \Sigma_{0}   
\end{equation}


Therefore by \eqref{eqn:oppenheim-mid} and \cref{lem:sigma-u2}, we have for $u_2$ the right eigenvector of $P$ corresponding to its second eigenvalue $\lambda_2$ that
\begin{align*}
\left(1 - \frac{2C}{k(k - 1)}\right)k^2 \lambda_2 u_2^\intercal \Pi u_2 
&= 
\left(1 - \frac{2C}{k(k - 1)}\right) u_2^\intercal \Sigma_{0} u_2  \\
&\le C(k - 1 - 1/(k - 1)) u_2^{\intercal} \Pi u_2 + \left(\frac{1}{k} - \frac{C}{k(k - 1)}\right) u_2^{\intercal} \Sigma_{0} N_0^{-1} \Sigma_{0} u_2 \\
&= C(k - 1 - 1/(k - 1)) u_2^{\intercal} \Pi u_2 + \left(1 - \frac{C}{(k - 1)}\right) k^2 \lambda_2^2  u_2^{\intercal} \Pi u_2 \\
\end{align*}
i.e., $\lambda_2$ satisfies the quadratic inequality
\[ 0 \le C(k - 1 - 1/(k - 1)) - \left(1 - \frac{2C}{k(k - 1)}\right) k^2 \lambda_2  + \left(1 - \frac{C}{(k - 1)}\right) k^2 \lambda_2^2. \]
Multiplying by $k - 1$ makes this
\[ 0 \le C(k^2 - 2k) - \left(k^2 - k - 2C\right) k \lambda_2  + \left(k - 1 - C\right) k^2 \lambda_2^2 = (\lambda_2 - 1) \left((k - 1 - C) k^2 \lambda_2 - C(k^2 - 2k)\right)\]
so using the assumption $\lambda_2 < 1$ we have that
\[ \lambda_2 \le \frac{C(1 - 2/k)}{k - 1 - C} = \frac{\lambda(1 - 2/k)}{1 - \lambda} \]
if we define $\lambda = C/(k - 1)$. The conclusion follows from \cref{lem:si-local} and \cref{fact:si-flc}. 
\end{proof}
\subsection{Illustration: trickle-down of semi-log-concavity}\label{a:basic-trickledown}
Oppenheim's trickle-down shows how spectral independence/fractional log-concavity can trickle-down from top links. As another example of the trickle-down idea, we prove a simple and roughly analogous result about semi-log-concave measures. In the special case of strongly log-concave measures, the analogous fact follows from the Brascamp-Lieb theorem \cite{brascamp2002some}.

\begin{theorem}\label{thm:trickle-down-semi-logconcave}
Suppose that $\mu$ is a $\gamma$-semi-log-concave measure on $\mathbb{R}^N$. Then for $T < \gamma$, the probability measure
$\mathcal{G}_{-T} \mu$ defined by
\[ \frac{d\mathcal{G}_{-T} \mu}{d\mu}(x) \propto \exp(T \|x\|^2/2) \]
is $\frac{1}{\gamma - T}$ semi-log-concave.
\end{theorem}
\begin{proof}
Let $\nu_0 = \mathcal{G}_{-T} \mu$ and let $\nu_t$ be the stochastic localization process with identity driving matrix initialized at $\nu_0$. By \eqref{eqn:sl-trickledown}, for all $t \in [0,T]$ we have
\[	\cov{\nu_t} = \E{\cov{\nu_T} \given \F_t} + \int_t^T \E{\cov{\nu_s}^2  \given \F_t} ds \preceq \frac{1}{\gamma} +  \int_t^T \E{\cov{\nu_s}^2 \given \F_t} ds. \]
Letting $f(t) = \sup_w \|\cov{\mathcal T_w \mathcal G_{t} \nu_0}\|_{OP}$ we therefore find that
\[ f(t) \le \frac{1}{\gamma} + \int_t^T f(s)^2 ds. \]
Recall that the solution of $dg/dt = -f(s)^2$ and $g(T) = 1/\gamma$ is $g(t) = \frac{1}{\gamma + t - T}$, so by a standard comparison argument $f(t) \le \frac{1}{\gamma + t - T}$ and plugging in $t = 0$ gives the desired result. 
\end{proof}
            \section{Glauber dynamics for antiferromagnetic Ising models on low-degree expanders}\label{a:glauber-lowdegree}
        In this appendix, we observe the following general result, which follows by combining key observations of the work \cite{koehler2022sampling} (more specifically, Corollary B.2 there)
        with a version of the ``annealing'' argument from \cite{chen2022localization}. One of the applications of this result is $O(n \log n)$ time mixing of the Glauber dynamics for antiferromagnetic Ising models on \emph{low-degree} expanders.
        
        The resulting algorithm has an incomparable runtime guarantee versus the algorithm proposed and analyzed in \cite{koehler2022sampling} based on nonconvex optimization and rejection sampling. As an advantage, the dependence on $n$ in the runtime is nearly linear for the Glauber dynamics, whereas the algorithm from \cite{koehler2022sampling} requires $\poly(n)$ time to output a single sample. As a disadvantage, the guarantee for Glauber dynamics has a doubly-exponential instead of single-exponential dependence on $\Tr(J_-)$. (In the application to antiferromagnetic Ising models on expanders, this means the guarantee for the Glauber dynamics would have a doubly-exponential dependence on the degree, but see \cref{rmk:trickledown-alternative} below.) 
        \begin{theorem}\label{thm:klr-based}
        Suppose that $\nu$ is a probability measure on the hypercube $\set{\pm 1}^n$
        \[ \nu(x) \propto \exp\parens*{\frac{1}{2}\dotprod{ x, J x } + \dotprod{ h, x }} \]
        for some $J,h$ such that $J \preceq 1 - 1/c$ (but $J$ is \emph{not} required to be positive definite). Let $J = J_+ - J_-$ be the decomposition of $J$ into its positive and negative definite parts. Then
        \[ \opnorm{\cov{\nu}} \le c e^{c \Tr(J_-)} \]
        and $\nu$ satisfies ATE with constant at most
        \[ \exp\parens*{\int_0^T c e^{cs\Tr(J_-)} ds }\]
        where $T = \lambda_{\max}(J) - \lambda_{\min}(J)$.
        \end{theorem}
        \begin{proof}
            By Corollary B.2 of \cite{koehler2022sampling}, $\nu$ has a density with respect to another Ising model $\pi$ where $\pi$ has an interaction matrix of spectral diameter $\lambda_{\max}(J)$, and $\frac{d\nu}{d\pi} \le e^{c \Tr(J_-)}$ where $J = J_+ - J_-$ is the decomposition of $J$ into psd and negative definite components. It follows that for any function $f$,
            \[ \Var_{\nu} f = \E_{\nu}{(f - \E_{\nu} f)^2} \le  \E_{\nu}{(f - \E_{\pi} f)^2} = \E*_{\pi} {\frac{d\nu}{d\pi} (f - \E_{\pi} f)^2} \le e^{c \Tr(J_-)} \Var_{\pi} f.  \]
            In particular, if we consider $f$ to be a linear function and use the bound on the covariance matrix of $\pi$ from page 405 of \cite{brascamp2002some}, this proves the first conclusion. 
        
            We consider the stochastic localization process with driving matrix $C_t = J + \lambda I$ where $-\lambda$ is the smallest eigenvalue of $J$. By using the above argument to bound the covariance at every time $t$ and appealing to entropic stability and the supermartingale property exactly as in the end of the proof of \cref{thm:ising}, we obtain the conclusion. 
        \end{proof}
        As a consequence of this result, we get $O(n \log n)$ time mixing of the Glauber dynamics for antiferromagnetic Ising models on $\lambda$-spectral expanders with interactions of strength up to $O(1/\lambda)$, e.g., up to $O(1/\sqrt{d})$ on random graphs and other very good spectral expanders, \emph{provided} that $d = O(1)$.
        \begin{corollary} \label{cor:glauber ising low degree}
        Suppose that 
        \[ \nu(x) \propto \exp\left(-\frac{\beta}{2} \langle x, A x \rangle + \langle h, x \rangle\right) \]
        where $A$ is the adjacency matrix of a $d$-regular graph on $n$ vertices and suppose that $\max\{|\lambda_2(A)|, |\lambda_n(A)|\} \le \lambda$. If $\lambda \beta \le 1 - 1/c$ for some $c > 0$, then $\nu$ satisfies ATE with constant at most $e^{c e^{O(\beta d)}}$. 
        \end{corollary}
        \begin{remark}\label{rmk:trickledown-alternative}
        For the above corollary, instead of using the results of \cite{koehler2022sampling}, we could use the trickledown-based result of \cref{lem:trickle down d regular} to bound the covariance matrix. This would reduce the range of $\beta$ we can handle by a constant factor, but it would improve the dependence on $d$ to be $e^{O(\beta d)}$. 
        \end{remark}
}
	\Tag<sigconf>{\balance\PrintBibliography}
\end{document}